\documentclass[a4paper,11pt,reqno]{amsart}

\pdfoutput=1

\usepackage{times}

\usepackage{microtype}

\usepackage{mathrsfs}

\usepackage{MnSymbol}
\usepackage{color}

\usepackage[colorlinks=true, pdfstartview=FitV, linkcolor=blue, citecolor=blue, urlcolor=blue,pagebackref=false]{hyperref}

\let\emptyset \undefined
\newsymbol\emptyset    203F

\theoremstyle{plain}
\newtheorem{theorem}{Theorem}[section]
\newtheorem{corollary}[theorem]{Corollary}
\newtheorem{lemma}[theorem]{Lemma}
\newtheorem{proposition}[theorem]{Proposition}

\theoremstyle{remark}
\newtheorem{remark}[theorem]{Remark}

\numberwithin{equation}{section}

\newcommand{\C}{\mathbb{C}}
\newcommand{\E}{\mathbb{E}}
\newcommand{\N}{\mathbb{N}}
\renewcommand{\P}{{\mathbb P}}
\newcommand{\Q}{{\mathbb Q}}
\newcommand{\R}{\mathbb{R}}
\newcommand{\T}{\mathbb{T}}
\newcommand{\Z}{\mathbb{Z}}

\def\rrr{\mathbf{r}}
\def\zz{\mathbf{z}}

\def\yy{\mathbf{y}}

\newcommand{\nn}{\mathfrak{m}}
\newcommand{\rr}{\mathsf{r}}
\newcommand{\KK}{\mathfrak{K}}

\newcommand{\Aa}{\mathcal{A}}
\newcommand{\Bb}{\mathcal{B}}
\newcommand{\Cc}{\mathcal{C}}
\newcommand{\Dd}{\mathcal{D}}

\newcommand{\Mm}{\mathcal{M}}
\newcommand{\Ss}{\mathcal{S}}

\newcommand{\MM}{\mathscr{M}}
\newcommand{\GG}{\mathscr{G}}

\newcommand{\Kk}{\mathscr{K}}

\newcommand{\al}{\nu}
\newcommand{\be}{\beta}
\newcommand{\ga}{\gamma}
\newcommand{\de}{\delta	}
\newcommand{\eps}{\varepsilon}
\newcommand{\ka}{\kappa}
\newcommand{\la}{\lambda}
\newcommand{\om}{\omega}
\newcommand{\bom}{\boldsymbol{\om}}

\newcommand{\si}{\sigma}

\renewcommand{\subset}{\subseteq}

\newcommand{\ov}{\overline}
\newcommand{\td}{\tilde}
\newcommand{\msf}{\mathsf}
\newcommand{\Ll}{\left}
\newcommand{\Rr}{\right}
\newcommand{\Ca}{\Cc^{-\al}} 								% Our most commonly used Besov space 

\newcommand{\ag}{\alpha}
\newcommand{\bg}{\be}
\newcommand{\dg}{\delta}
\newcommand{\eg}{\varepsilon}

\newcommand{\hg}{h_{\ga}} 								%Rescaled field 
\newcommand{\kg}{\ka_\ga} 								% Interaction kernel
\newcommand{\Kg}{K_\ga} 								% Interaction kernel in macroscopic coordinates
\newcommand{\LN}{\Lambda_N} 							% discrete space microscopic coordinates
\newcommand{\Le}{\Lambda_\eps}		 					% discrete space macroscopic coordinates
\newcommand{\SN}{\Sigma_N} 							% Spin configurations
\newcommand{\Hg}{\mathscr{H}_{\ga}} 						% Hamiltonian
\newcommand{\lbg}{\lambda_{\ga}} 							% Gibbs measure
\newcommand{\Zbg}{\mathscr{Z}_{\ga}} 						% Partition function
\newcommand{\LgN}{\mathscr{L}_{\ga}} 						% Generator
\newcommand{\cg}{c_{\ga}}		 	 					% Jump rate microscopic coordinates
\newcommand{\Cg}{C_\ga} 								% Jump rate macroscopic coordinates
\newcommand{\mg}{m_\ga} 								% noise in microscopic coordinates 
\newcommand{\Mg}{M_\ga}  								% noise in macroscopic coordinates
\newcommand{\Dg}{\Delta_\ga}		 					% Approximation to Laplace operator
\renewcommand{\ae}{\star_\eps} 							% Convolution on macroscopic discrete torus
\newcommand{\Eg}{E_\ga} 								% Error term appearing in the equation due to expansion of tanh
\newcommand{\Xg}{X_{\ga}}		  						% The rescaled field - our main object
\newcommand{\co}{c_{\ga,2}}  								% Constant close to 1, due to inaccuracy of $\eps =\ga^2$
\newcommand{\ct}{c_{\ga,1}}  								% Constant close to 1, due to approximation of integral not exactly =1 
\newcommand{\Xng}{X_\ga^0} 							% Approximate initial condition 
\newcommand{\Xn}{X^0} 							% Initial condition 
\newcommand{\CGG}{\mathfrak{c}_{\ga}} 					% gamma - dependent renormalisation constant

\newcommand{\Xe}{X_\eps}		 						% Solution to approximated PDE
\newcommand{\Ze}{Z_\eps} 								% Approximation to SHE
\newcommand{\Ce}{\mathfrak{c}_\eps} 						% Renormalisation constant
\def\RR#1{R_{t}^{:{#1}:}} 									% continuous R
\def\ZZ#1{Z^{:{#1}:}} 										% continous Z
\newcommand{\tZ}{\widetilde{Z}} 							% processes with strange initial datum
\newcommand{\tZz}{ \widetilde{Z}^{\colon 2 \colon}} 			% and the second
\newcommand{\tZd}{ \widetilde{Z}^{\colon 3 \colon}} 			% and the third
\newcommand{\hRe}{\hat{R}_{\eps,t}}
\def\hRE#1{\hat{R}_{\eps,t}^{:{#1}:}}

\def\Pg#1{P_{#1}^{\gamma}} 								% Approximation to the heat semigroup
\def\hPg#1{\hat{P}_{#1}^{\gamma}} 							% Fourier transform of the approximated heat semigroup                
\newcommand{\Zg}{Z_{\ga}}  								% The approximate stochastic convolution
\newcommand{\Rg}{R_{\ga,t}}		 						% The martingale approximation to \Zg
\newcommand{\hRg}{\hat{R}_{\ga,t}}		 				% Its Fourier transform
\def\RG#1{R_{\ga,t}^{:{#1}:}} 								% discrete iterated integrals 
\def\hRG#1{\hat{R}_{\ga,t}^{:{#1}:}} 							% Fourier transform of the discrete iterated integral
\def\ZG#1{Z_{\ga}^{:{#1}:}} 								% discrete iterated integrals for $s=t$
\def\I#1{\mathfrak{I}_{#1}} 								% Iterated integral operator
\def\ER#1{\msf{Err}({#1})} 								% Error due to jumps in iterated integrals
\newcommand{\ZgH}{\Zg^{\mathrm{high}}} 					% Its Fourier transform

\newcommand{\Qg}{Q_{\ga,t}} 								% Difference between different quadratic variations
\def\EG#1{E_{\ga,t}^{:{#1}:}} 								% Difference between Rn and the Hermite Polynomial
\newcommand{\tZg}{\widetilde{Z}_{\ga}} 						% piecewise linearise version of \Zg - needed in tightness proof
\def\tZG#1{\widetilde{Z}_{\ga}^{:{#1}:}} 						% piecewise linearised version on \ZG{n}
\newcommand{\Len}{\Lambda_\eps^{(n)}}		 			% n fold discrete space macroscopic coordinates
\newcommand{\unR}{\underline{R}} 						% The tuple \R, [R]
\newcommand{\unZ}{\underline{Z}} 							% The tuple \Z, [R]

\newcommand{\taun}{\tau_{\ga,\nn}}							 % Stopping time
\newcommand{\Rgn}{R_{\ga,t,\nn}} 							% discrete integral changed after stopping time
\def\RGn#1{R_{\ga,t,\nn}^{:{#1}:}}		 					%  discrete iterated integrals changed after stopping time
\newcommand{\Zgn}{Z_{\ga,\nn}} 							% linear process extended beyond stopping time
\def\ZGn#1{Z_{\ga,\nn}^{:{#1}:}} 							% Zn changed after stopping time
\newcommand{\hgn}{h_{\ga,\nn}} 							% Rescaled field changed after stopping time 
\newcommand{\sgn}{\sigma_{\ga,\nn}}		 				% modified microscopic jumps
\newcommand{\ssgn}{\sigma_{\ga,\nn}^{\prime}}		 		% used to define microscopic jumps
\newcommand{\mgn}{m_{\ga,\nn}} 							% noise in microscopic coordinates changed after stopping time
\newcommand{\Mgn}{M_{\ga,\nn}} 							% noise in macroscopic coordinates changed after stopping time
\newcommand{\Lgn}{\mathscr{L}_{\ga,\nn}} 					% Generator changed after stopping time
\newcommand{\cgn}{c_{\ga,\nn}}							% Jump rate changed after stopping time in microscopic coordinates
\newcommand{\Cgn}{C_{\ga,\nn}}		 					% Jump rate changed after stopping time in macroscopic coordinates
\newcommand{\bZg}{\mathbf{Z}_\ga} 						% Vector of processes $Zg$ for convergence in law
\def\bRg#1{\mathbf{R}_{\ga,{#1} }} 							% Vector of processes $Rg$ for convergence in law
\newcommand{\bZ}{\mathbf{Z}} 							% limiting vector of processes $Z$ for  convergence in law
\newcommand{\bR}{\mathbf{R}_t}							% limiting vector of processes $R$ for convergence in law

\newcommand{\ttZ}{ \widetilde{Z}_{\ga,\nn}} 					% processes with effect of initial condition
\newcommand{\ttZz}{ \widetilde{Z}_{\ga,\nn}^{\colon 2 \colon}}		% and again for the second power
\newcommand{\ttZd}{ \widetilde{Z}_{\ga,\nn}^{\colon 3 \colon}}		% and for the third power
\newcommand{\Yg}{ Y_\ga} 								% Heat semigroup applied to initial datum
\newcommand{\oX}{\overline{X}_{\ga,\nn}} 					% the approximate triple $Z_\ga,Z^_\ga2,Z_\ga^3$ as Input for the limiting solution operator
\newcommand{\tXgn}{X_{\ga,\nn}} 							% approximation that coincides with $\Xg$ before stopping time
\newcommand{\tvg}{ v_{\gamma,\nn}} 						% the exact remainder before stopping time
\newcommand{\bvg}{ \overline{v}_{\gamma,\nn}} 				% the remainder of the continuous equation with discrete input
\newcommand{\bZZg}{ \overline{Z}_{\gamma,\nn}} 				% the linear process with discrete input                                          
\def\bZZG#1{\overline{Z}_{\ga,\nn}^{:{#1}:}} 					% Zn changed discrete Z with continuous initial condition
\newcommand{\Psg}{ \Psi_{\gamma,\nn}} 					% nonlinearity in fixed point argument
\newcommand{\bPsg}{ \overline{\Psi}_{\gamma,\nn}} 			% nonlinearity in fixed point argument with bar
\newcommand{\Ex}{\msf{Ext}} 								% Extension operator
\newcommand{\tCG}{\mathfrak{c}_{\ga}} 						% approximate value of CG(s)
\def\ERR#1{\msf{Err}^{({#1})}} 								% "Big" error terms in Sec 7
\def\err#1{\msf{err}^{({#1})}} 								% "small" error terms in Sec 7

\newcommand{\Lg}{\ga\Z^2} 								%Rescaled torus (once more) 
\newcommand{\hKK}{\widehat{\KK}} 						% Fourier transform of continuous Kac thing
\newcommand{\hKg}{\hat{K}_{\ga}} 							% Fourier transform of approximate Kac thing
\newcommand{\Delg}{\Delta_\ga} 					% yet another discrete Laplacian
\newcommand{\Dell}{\underline{\Delta}} 						% and another one
\newcommand{\Delk}{\underline{\Delta}_k} 					% yet another Laplacian in Fourier modes

\newcommand{\Del}{\underline{\Delta}_\eps} 					% Laplacian in Fourier modes
\newcommand{\dk}{\delta_k} 								% k-th Paley-Littlewood block
\newcommand{\ek}{\eta_k} 								% convolution kernel for $k$-th block

\renewcommand{\L}{\mathscr{L}}

\begin{document}
\begin{abstract}
The Ising-Kac model is a variant of the ferromagnetic Ising model in which each spin variable interacts with all spins in a neighbourhood of radius $\ga^{-1}$ for $\ga \ll1$ around its base point.  We study the Glauber dynamics for this model on a discrete two-dimensional torus $\Z^2/ (2N+1)\Z^2$, for a system size $N \gg \ga^{-1}$ and for an inverse temperature close to the critical value of the mean field model. We show that the suitably rescaled coarse-grained spin field converges in distribution to the solution of a non-linear stochastic partial differential equation.

This equation is the dynamic version of the $\Phi^4_2$ quantum field theory, which is formally given by a reaction diffusion equation driven by an additive space-time white noise. It is well-known that in two spatial dimensions, such equations are distribution valued and a \textit{Wick renormalisation} has to be performed in order to define the non-linear term. Formally, this renormalisation corresponds to adding an infinite mass term to the equation. We show that this need for renormalisation for the limiting equation is reflected in the discrete system by a shift of the critical temperature away from its mean field value.
\end{abstract}

\title[Convergence of dynamic Ising-Kac model to $\Phi_2^4$]
{Convergence of the two-dimensional dynamic Ising-Kac model to  $\Phi_2^4$}

\author{Jean-Christophe Mourrat}
\address{Jean-Christophe Mourrat, \'Ecole Normale Sup\'erieure de Lyon, CNRS
}
\email{jean-christophe.mourrat@ens-lyon.fr}

\author{Hendrik Weber}
\address{Hendrik Weber, University of Warwick
}
\email{hendrik.weber@warwick.ac.uk}



% \begin{abstract}

% \end{abstract}

%\dedicatory{}

\date\today

\maketitle

%%%%%%%%%%%%%%%%%%%%%%%%
\section{Introduction}
%%%%%%%%%%%%%%%%%%%%%%%%
\label{s:Intro}

The aim of this article is to show the convergence of a rescaled discrete spin system to the solution of a stochastic partial differential equation formally given by 
\begin{equation}\label{e:SPDE}
\partial_t X(t,x) = \Delta X(t,x) -\frac13 X^3(t,x) +A X(t,x) + \sqrt{2} \, \xi(t,x)\;.
\end{equation}
Here $x \in \T^2$ takes values in the two-dimensional torus,  $\xi$ denotes a space-time white noise, and $A \in \R$ is a real parameter. 

The particle system that we consider is an Ising-Kac model evolving according to the Glauber dynamics. This model is similar to the usual ferromagnetic Ising model. The difference is that every spin variable interacts 
with all other spin variables in a large ball of radius $\ga^{-1}$ around its base point. We consider this model on a discrete two-dimensional torus $\Z^2/(2N+1) \Z^2$, for $N \gg \ga^{-1}$. We then study the random fluctuations of a coarse-grained and suitably rescaled magnetisation field $\Xg$ in the limit $\ga \to0$, for inverse temperature close to the critical temperature of the mean field model. Our main result, Theorem~\ref{thm:Main}, states that under suitable assumptions on the initial configuration, these fields $\Xg$ converge in law to the solution of \eqref{e:SPDE}. A similar result in one spatial dimension was shown in the nineties in  \cite{BPRS,FR}. Our two-dimensional result was conjectured in \cite{GLP}.   

The two-dimensional situation is more difficult than the one-dimensional case, because the solution theory for \eqref{e:SPDE}
is more intricate. Indeed, it is well-known that in dimension higher than one, equation~\eqref{e:SPDE} does not make sense as it stands. We will recall below in Section~\ref{sec:ContinuousAnalysis} that the solution $X$ to \eqref{e:SPDE} is a distribution-valued  process. For each fixed $t$, the regularity properties of $X$ are identical to those of the Gaussian free field. In this regularity class, it is a priori not possible to give a consistent interpretation to the nonlinear term $X^3$. In order to even give a meaning to \eqref{e:SPDE}, a renormalisation procedure has to be performed. Formally, this procedure corresponds to adding an \textit{infinite mass term} to \eqref{e:SPDE}, i.e. \eqref{e:SPDE} is formally replaced by 
\begin{equation*}
\partial_t X(t,x) = \Delta X(t,x) -\frac13 \big( X^3(t,x) - 3\infty \times X(t,x) \big) +A X(t,x) + \sqrt{2} \, \xi(t,x)\;,
\end{equation*}
where ``$\infty$'' denotes a constant that diverges to infinity in the renormalisation procedure
(see Section~\ref{sec:ContinuousAnalysis} for a more precise discussion).

A similar renormalisation was performed for the equilibrium state of \eqref{e:SPDE} in the seventies, in the context of constructive quantum field theory (see \cite{GJ81} and the references therein). This equilibrium state is given by a measure on the space of distributions over $\T^2$ (or $\R^2$) which is formally described by 
\begin{equation}\label{e:phi_four_two}
\frac{1}{\mathscr{Z}} \exp\Big(- \int \frac{1}{12}(X^4(x) - 6 \infty \times X(x)^2 + 3 \infty ) + \frac12AX(x)^2 \, dx \Big) \, \nu(dX) \,,
\end{equation}
where $\nu$ is the law of a Gaussian free field. It was shown in  \cite{GlimmJaffe} that like the two-dimensional Ising model, this model admits a phase transition, with the parameter $A$ playing the role of the inverse temperature. The measure \eqref{e:phi_four_two} is usually called the $\Phi^4_2$ model, and we will therefore refer to the solution of \eqref{e:SPDE} as the \textit{dynamic} $\Phi^4_2$ model.

The construction of the dynamic $\Phi^4_2$ model was a challenge for stochastic analysts for many years. Notable contributions include \cite{ParisiWu, JLM, AlbRock, MiRo}. 
Our analysis builds on the fundamental work of da Prato and Debussche \cite{dPD}, in which strong solutions for \eqref{e:SPDE} were constructed for the first time.

Much more recently, there was great progress in the theory of singular SPDEs, in particular with \cite{MartinKPZ,Martin1}. In these papers, Hairer developed a theory of \textit{regularity structures} which allows to perform similar renormalisation procedures for much more singular equations such as the three-dimensional version of \eqref{e:SPDE} or the Kardar-Parisi-Zhang (KPZ) equation. Parallel to \cite{MartinKPZ,Martin1}, another fruitful approach to give a meaning to such equations was developped in \cite{Gubi,MassimilianoPHI4}. One of the motivations for these works is to develop a technique to show that fluctuations of non-linear particle systems are governed by such an equation. The present article establishes such a result in this framework for the first time. One interesting feature of our result is that it gives a natural interpretation for the infinite renormalisation constant as a shift of critical temperature away from its mean field value (as was already predicted in \cite{CMPres-2,GLP}).  

The study of the KPZ equation has recently witnessed tremendous developments (see \cite{quastel,corwin,corwinICM} and references therein). In their seminal paper \cite{bergia}, Bertini and Giacomin proved that the (suitably rescaled height process associated with the) weakly asymmetric exclusion process converges to the KPZ equation. This result relies on two crucial properties: (1) that the KPZ equation can be transformed into a \emph{linear} equation (the stochastic heat equation with multiplicative noise) via a Cole-Hopf transform; and (2) that the discrete system itself allows for a microscopic analogue of the Cole-Hopf transform \cite{gartner}. The result can be extended to other particle systems as long as some form of microscopic Cole-Hopf transform is available \cite{dembo}. However, for more general models, the question is still open, although notable results in this direction were obtained in \cite{gonjar}.

For the asymmetric exclusion process to converge to the KPZ equation, it is essential to tune down the asymmetry parameter to zero while taking the scaling limit (hence the name \emph{weakly} asymmetric). This procedure enables to keep the system away from the scale-invariant \emph{KPZ fixed point}. This KPZ fixed point remains partly elusive \cite{corqua}. Proving that a discrete system converges to the KPZ fixed point has so far only been possible (in a limited sense) by relying on the special algebraic properties of integrable models, see \cite{badejo,johan,praspo} for early works, and \cite{corwinICM} and references therein for more recent developments.

We would like to underline the analogy between these observations and the situation with the two-dimensional Ising, Ising-Kac and $\Phi^4$ models. The scaling limit of the (static) two-dimensional Ising model with nearest-neighbour interactions is now well understood, see \cite{smirnov, cagane, Ising-sle,chi} and references therein. We may call this limit the (static, critical) \emph{continuous Ising model}. Our replacement of nearest-neighbour by long-range, Kac-type interactions does \emph{not} simply serve as a technical convenience. It also plays the role that the tunable asymmetry has for the convergence of the weakly asymmetric exclusion process discussed above. That is, it enables to keep the model away from the (scale-invariant) continuous Ising model. In order to best see that this limit is qualitatively different from the continuous Ising model (and its near-critical analogues), we point out that the probability for the field averaged over the torus to be above a large value $x$ decays roughly like $\exp(-x^{16})$ for the continuous Ising model \cite{cagane2}, while one can check that this probability decays roughly like $\exp(-x^4)$ for the measure in \eqref{e:phi_four_two}, as it does for the Curie-Weiss model \cite[Theorem~V.9.5]{ellis}. It is expected that the measure \eqref{e:phi_four_two} with critical $A$ converges, under a suitable scaling, to the continuous Ising model.

\subsection{Structure of the article} In Section \ref{s:Setting}, we define our model and give a precise statement of our main results. In Section \ref{sec:ContinuousAnalysis}, we describe briefly the solution theory for the  limiting dynamics, following essentially the strategy of \cite{dPD}. Sections \ref{sec:linear} to \ref{sec:conv-lin} contain the core of our article. There, we introduce a suitable linearisation of the rescaled particle model and prove the convergence of this linearisation and several non-linear functionals thereof to the continuum model. In Section \ref{sec:Nonlinear}, we analyse the nonlinear evolution and complete the proof of our main theorem.  Finally, in Section \ref{sec:APPB}, we derive some additional bounds on an approximation to the heat semigroup that are referred to freely throughout the paper.  In several appendices, we provide background material on Besov spaces, martingale inequalities and the martingale characterisation of the solution to the stochastic heat equation.

\subsection{Notation} Throughout the paper, $C$ will denote a generic constant that can change at every occurrence. We sometimes write $C(a,b,\ldots)$ if the exact value of $C$ depends on the quantities $a,b,\ldots$ .  For $x \in \R^d$, we write $|x| = \sqrt{x_1^2 + \cdots + x_d^2}$ for the Euclidean norm of $x$. For $x \in \R^d$,  and $r>0$,  $B(x,r) =\{ y \in \R^d \colon |x-y| <r \}$ is the Euclidean ball of radius $r$ around $x$.
For $a,b \in \R$, we write $a \wedge b$ and $a \vee b$ to denote the minimum and the maximum of $\{a,b\}$. We denote by $\N=\{1,2,3, \ldots\}$ the set of natural numbers and by $\N_0=\N \cup \{0\}$. We also write $\R_+ = [0, \infty)$.

\medskip

\noindent \textbf{Acknowledgements.} We would like to thank David Brydges, Christophe Garban and Krzysztof Gawedzki for discussions on the nearest-neighbour Ising model and quantum field theory. We would also like to thank Rongchan Zhu and Xiangchan Zhu for pointing out an error in an earlier version of this article.
HW was supported by an EPSRC First Grant.

%%%%%%%%%%%%%%%%%%%%%%%%
\section{Setting and Main result}
%%%%%%%%%%%%%%%%%%%%%%%%
\label{s:Setting}
%%%%%%%%%%%%%%%%%%%%%%%%

For $N \geq 1$, let $\LN =\Z^2 / (2N+1) \Z^2$ be the two-dimensional discrete torus which we identify with the set  $  \{-N, -(N-1), \ldots, N \}^2$. Denote by $\SN = \{ -1, +1 \}^{\LN}$ the set of spin configurations on $\LN$. We will always denote spin configurations by $\si = (\si(k), \, k \in \Lambda_N )$. 

Let $\KK \colon \R^2 \to [0,1]$ be a $\Cc^2$ function with compact support contained in  $B(0,3)$, the Euclidean ball of radius $3$ around $0$ in $\R^2$. We assume that $\KK$ is invariant under rotations and satisfies
\begin{equation}
\label{e:norm-kk}
\int_{\R^2} \KK(x) \, dx = 1, \qquad \int_{\R^2} \KK(x) \, |x|^2 \, dx = 4 \;. 
\end{equation}
For $0 < \gamma< \frac13$, let $\kg \colon \LN \to [0, \infty)$ be defined as $\kg(0) =0$ and
\begin{equation}\label{e:samplek}
\kg(k) = \ct \, \gamma^2 \, \KK(\ga k)  \qquad k \neq 0\;,
\end{equation}
where $\ct^{-1} = \sum_{k \in \LN^{\star} }\ga^2 \, \KK(\ga k)$ and $\LN^{\star} = \LN \setminus \{0\}$.

For any $\si \in \SN$, we introduce the locally averaged field 
\begin{equation}\label{e:hg}
\hg(\si, k) := \sum_{j \in \LN } \kg(k-j) \, \si(j)=: \kg \star \si (k)\;,
\end{equation}
and  the Hamiltonian 
\begin{equation}\label{e:Hamiltonian}
\Hg(\si) := -  \frac12 \sum_{k, j \in \LN}  \kg(k-j) \,\si(j) \,  \si(k) \,= - \frac12 \sum_{k \in \LN} \si(k) \,\hg( \si, k) \;. 
\end{equation}
In both \eqref{e:hg} and \eqref{e:Hamiltonian}, subtraction on $\LN$ is to be understood with periodic boundary conditions. Throughout this article we will always assume that $N \gg \gamma^{-1}$, so that  the assumption  $\kg(0) = 0$ implies that there is no self-interaction in \eqref{e:Hamiltonian}.

For any inverse temperature $\be>0$, we define the Gibbs measure $\lbg$ on $\SN$ as 
\begin{equation*}
\lbg (\si) := \frac{1}{\Zbg}\exp\Big( - \be\Hg(\si) \Big)\; ,
\end{equation*}
where 
\begin{equation*}
\Zbg := \sum_{\si \in \SN} \exp\Big( - \be\Hg(\si) \Big) \; 
\end{equation*}
denotes the normalisation constant that makes $\lbg$ a probability measure. 
On $\SN$, we study the Markov process given by the generator 
\begin{equation}\label{e:Generator}
\LgN f(\si) = \sum_{j \in \LN} \cg(\si,j) \big(f(\si^j) -f(\si) \big) \;,
\end{equation}
 acting on functions $f \colon\SN \to \R$. Here $\si^j \in \SN$ is the spin configuration that coincides with $\si$ except for a flipped spin at position $j$.  
As jump rates $\cg(\si, j)$  we choose those of the Glauber dynamics, 
$$
\cg(\si, j) := \frac{\lbg(\si^j)}{\lbg(\si) + \lbg(\si^j)}\;.
$$
It is clear that these jump rates are reversible with respect to the measure $\lbg$. Since $\kg(0) = 0$, the local mean field $\hg(\si,j)$ does not depend on $\sigma(j)$. Using also the fact that $\si(j) \in \{-1,1\}$, we can conveniently rewrite the jump rates as
\begin{eqnarray}
\cg(\si, j)&  = &  \frac{e^{-\si(j) \be \hg(\si,j)}}{e^{ \be \hg(\si,j)}+e^{- \be \hg(\si,j)}}\notag \\
& = & \frac{1}{2}\big(1 - \si(j) \tanh\big( \be \hg(\si,j) \big)  \big)\;.  \label{e:rate}
\end{eqnarray}
We write $(\si(t))_{t \ge 0}$ for the (pure jump) Markov process on $\SN$ thus defined, with the notation $\si(t) = (\si(t,k))_{k \in \LN}$. With a slight abuse of notation, we let
\begin{equation}
\label{e:def-hgt}
\hg(t,k) = \hg(\si(t),k) \; .
\end{equation}
The aim of this article is to describe the critical fluctuations of the  local mean field~$\hg$ as defined in \eqref{e:def-hgt}, and to derive a 
non-linear SPDE for a suitably rescaled version of it. To this end we write, for $t \ge 0$ and $k \in \LN$,
\begin{equation}\label{e:evolution1}
\hg(t,k) = \hg(0,k) + \int_0^t \LgN \, \hg(s, k) \, ds + \mg(t,k)\;,
\end{equation} 
where the process $\mg(\cdot, k)$ is a martingale. Observing that for any $\si \in \SN$ and for any $j,k \in \LN$, we have $\hg(\si^j,k) - \hg(\si,k) = -2 \kg(k-j ) \si(j)$,
we get from \eqref{e:Generator} and \eqref{e:rate}
\begin{align}
\LgN  \hg(\si, k)
 &= - \hg(\si,k) + \kg \star \tanh\big( \beta \hg(\si,k)  \big)   \notag\\
& = \Big(\kg \star \hg(\sigma, k) - \hg(\sigma,k) \Big) + (\beta-1) \,\kg \star \hg(\sigma,k)  \notag\\
& \qquad   - \frac{\beta^3}{3} \Big( \kg \star \hg^3(\si, k)\Big) + \ldots   \label{e:Lh}\;,
\end{align}
where we have used the Taylor expansion $\tanh(\beta h) = \beta h - \frac13 (\beta h)^3 + \ldots$ .

The predictable quadratic covariations (see Appendix~\ref{AppA}) of the martingales $\mg(\cdot, k)$ are given by 
\begin{align}
\langle \mg(\cdot, k ), \mg (\cdot, j) \rangle_t 
=4 \int_0^t  \!\! \sum_{\ell \in \LN} \kg(k-\ell) \, \kg(j-\ell) \,  \cg\big(\si(s),\ell \big) \, ds\;. \label{e:QuadrVar}
\end{align}
Furthermore, the jumps of $\mg(\cdot,k)$ coincide with those of $\hg(\cdot,k)$. In particular, if for some $\ell \in \LN$ the spin $\sigma(\ell)$ changes sign, then $\mg(\cdot,k)$ has a jump of $-2\si(\ell) \kg(k-\ell)$.

\subsection{Rescaled dynamics}

For any $0<\gamma <1$ let $N= N(\gamma)$ be the microscopic system size determined below (in \eqref{e:scaling1}). Then set $\eg = \frac{2}{2N+1}$. Every \emph{microscopic} point  $ k \in \LN$ can be identified with $x = \eg k \in \Le = \{x = (x_1,x_2) \in \eg \Z^2 \colon \, -1 < x_1,x_2 < 1  \}$. We view $\Le$ as a discretisation of the 
continuous two-dimensional torus $\T^2$ identified with $[-1,1]^2$. For suitable scaling factors $\ag, \dg > 0$ and inverse temperature $\bg$ to be determined below we set 
\begin{equation}
\label{e:defXg}
\Xg(t,x) = \frac{1}{\dg} \hg\bigg( \frac{t}{\ag}, \frac{x}{\eg} \bigg)  \qquad x\in \Le, \, t \geq 0\;.
\end{equation}
In these \emph{macroscopic coordinates}, the evolution equation \eqref{e:evolution1} (together with \eqref{e:Lh}) reads 
\begin{align}
\Xg(t,x) =  & \Xg(0,x) + \int_0^t \!   \bigg(  \frac{ \eg^2}{\ga^2 } \frac{1}{\ag}  \widetilde{\Delta}_\gamma \Xg(s,x) + \frac{(\be -1)}{\ag} \Kg \ae \Xg(s,x) \notag \\  
&- \frac{\be^3}{3} \frac{\dg^2}{\ag} \Kg \ae \Xg^3  (s,x)  + \Kg \ae \Eg(s,x) \bigg)\, ds + \Mg(t,x) \;  ,\label{e:evolution2}
\end{align}
for $x \in \Le$.
Here we have set  $\Kg(x) = \eg^{-2} \kg(\eg^{-1}x) =  \ct \frac{\ga^2}{\eg^2} \KK\big( \frac{\ga}{\eg} x \big)$ (the second equality being valid for $x \neq 0$). The convolution $\ae$ on $\Le$ is defined through $X \ae Y(x) = \sum_{z \in \Le}\eg^2  X(x-z)  Y(z)$ (where subtraction on $\Le$ is to be understood with periodic boundary conditions on $[-1,1]^2$) and 
$\widetilde{\Delta}_\gamma X =\frac{\gamma^2}{\eg^2} (\Kg \ae X - X)$ (so that $\widetilde{\Delta}_\gamma$ scales like the continuous Laplacian). The rescaled martingale is defined as $\Mg(t,x) := \frac{1}{\dg} \mg\Big(\frac{t}{\ag}, \frac{x}{\eg} \Big) $. Finally, the error term $\Eg(t,x)$ (implicit in \eqref{e:Lh}) is given by 
\begin{equation}\label{e:error}
\Eg(t,\cdot) =  \frac{1}{\dg \ag}    \Big( \tanh\big( \beta \dg \Xg(t,\cdot)  \big)  -  \beta \dg \Xg(t, \cdot) + \frac{(\beta \dg)^3}{3   }  \Xg(t, \cdot)^3   \Big) \;.
\end{equation}
In these coordinates, the expression \eqref{e:QuadrVar} for the quadratic variation becomes
\begin{align}
\langle \Mg(\cdot, x ),& \Mg (\cdot, y) \rangle_t  \notag\\
&=4 \frac{\eg^2}{\dg^2\ag } \int_0^t  \!\! \sum_{z \in \Le}\eg^2  \Kg( x -z) \, \Kg(y-z) \,  \Cg\big(s, z\big) \, ds \;, \label{e:QuadrVar2}
\end{align} 
where $\Cg(s,z) :=  \cg(\si(s/\ag), z/\eg )$. In these macroscopic coordinates, a spin flip at the microscopic position $k = \eg^{-1}y$ causes a jump of $-2 \sigma(\eg^{-1}y)  \delta^{-1} \eps^2 \Kg(y-x)  $ for the martingale $
\Mg(\cdot, x)$.

The scaling of the approximated Laplacian, the term $\Kg \ae \Xg^3$ and the quadratic variation in \eqref{e:QuadrVar2} suggest that in order to see  non-trivial behaviour for these terms, we need to impose $
1 \approx \frac{\eg^2}{ \ga^2}\frac{1}{\ag}  \approx \frac{\dg^2}{\ag} \approx \frac{\eg^2}{\dg^2\ag }$.
Hence, from now on we set 
\begin{equation}\label{e:scaling1}
N = \left\lfloor\gamma^{-2}\right\rfloor, \quad \eg = \frac{2}{2N+1}, \quad \ag= \ga^2, \quad \dg = \ga \;.
\end{equation}
For later reference, we note that this implies that for $0 < \ga < \frac13$ we have 
\begin{equation}\label{e:scaling2}
\eg  =  \ga^2 \, \co   \qquad \text{with} \quad  \big(1- \ga^2   \big) \leq \co \leq   \big(1+ \ga^2   \big) \;.
\end{equation}
Under these assumptions, the leading order term in the expansion of the error term \eqref{e:error} scales like $\dg^4 \ag^{-1}  = \ga^2$. Hence it seems reasonable to suspect that it will disappear in the limit.
In order to prevent the (essentially irrelevant) factor $\co$ from appearing in too  many formulas, we define 
\begin{equation}\label{e:DefDg}
\Dg := \co^{2} \widetilde{\Delta}_\gamma = \frac{ \eg^2}{\ga^2 } \frac{1}{\ag}  \widetilde{\Delta}_\gamma \;.
\end{equation}

At first sight, \eqref{e:evolution2} suggests that $\beta$ should be so close to one that $(\be -1)/ \ag = O(1)$. It was already observed in \cite{CMPres-2}  (for the equilibrium system) that this naive guess is incorrect. Instead, we will always assume that 
\begin{align}\label{e:scalingbeta}
( \bg -1) = \ag (\CGG +A)\;,
\end{align}
where $A \in \R$ is fixed. The extra term $\CGG$ reflects the fact that the limiting equation has to be renormalised (see Section~\ref{sec:ContinuousAnalysis} for a detailed explanation). Its precise value is given below in \eqref{e:valueCGG}, but we mention right away that the difference between $\CGG$ and 
\begin{align*}
\sum_{\substack{\om \in \Z^2 \\ 0<| \om | < \ga^{-1}}} \frac{1}{4 \pi^2 |\om|^2}
\end{align*}
remains bounded as $\ga$ goes to $0$. In particular, $\CGG$ diverges logarithmically as $\ga$ goes to $0$.

Before we state our main results, we briefly discuss a suitable extension of the functions $\Xg$, $\Kg$ defined on $\Le$ onto all of the torus $\T^2$ (which we identify with the interval $[-1,1]^2$). A large part of the analysis below is based on tools from Fourier analysis, and hence it is  natural to extend these functions by suitable trigonometric polynomials. For any function $Y \colon \Le \to \R$, we define the Fourier transform $\hat{Y}$ through
\begin{equation}\label{e:FT}
\hat{Y} (\om) = 
\begin{cases}
\sum_{x\in \Le} \eg^2 \, Y(x) \,e^{- i \pi \om \cdot x}     \qquad & \text{ if } \om \in \{ -N, \ldots, N \}^2\;,\\
0 \qquad & \text{ if } \om \in \Z^2 \setminus \{ -N, \ldots, N \}^2\; .
\end{cases}
\end{equation} 
In this context Fourier inversion states
\begin{equation}\label{e:FI}
Y(x) =\frac{1}{4} \sum_{\om \in \Z^2} \,   \hat{Y}(\om) \, e^{ i \pi \om \cdot x }  \qquad \text{for all } x \in \Le\;.
\end{equation}
It is thus natural to extend $Y$ to all of $\T^2 = [-1,1]^2$ by taking  \eqref{e:FI} as a definition of $Y(x)$ for $x \in \T^2 \setminus \Le$. More explicitly, for $Y \colon \Le \to \R$, we define $(\Ex \, Y) \colon \T^2 \to \R$ as 
\begin{align*}
\Ex \, Y(x)  &= \frac{1}{4} \sum_{\om \in \{-N, \ldots, N\}^2}   \sum_{y \in \Le}  \eg^2   \; e^{i \pi \om \cdot (x-y)} \; Y(y) \\
&=    \sum_{y \in \Le}  \frac{\eg^2}{4} \; Y(y) \prod_{j=1,2} \frac{\sin\Big(\frac{\pi}{2} (2N+ 1)   (x_j-y_j)  \Big)}{ \sin\big(    \frac{\pi}{2}  (x_j-y_j) \big)}  \;.
\end{align*}
The extension of $Y$ defined in this way is smooth and real valued. The function $\Ex \,Y$ is the unique trigonometric polynomial of degree $\leq N$ that coincides with $Y$ on $\Le$. This extension has the nice property that several identities that should hold approximately are in fact exactly true. 

First of all, by the continuous Fourier inversion we see that for $\om \in \Z^2$
$$
\hat{Y}(\om) = \int_{\T^2} e^{ - i \pi \om \cdot x}\, (\Ex\, Y)(x) \, dx \ ,
$$
i.e. the discrete Fourier transform of $Y$ on $\Le$ and the continuous Fourier transform of its extension coincide. Furthermore, if $X,Y \colon \Le \to \R$ and $Z := X \ae Y$, then the discrete Fourier transform \eqref{e:FT} satisfies $\hat{Z}(\om) = \hat{X}(\om) \hat{Y}(\om)$. As the Fourier transform of $\tilde{ Z}(x) = \int_{\T^2} (\Ex\, X)(x-y) \, (\Ex \, Y)(y) \, dy$ satisfies the same identity, we can conclude that $Z$ and $\tilde{Z}$ coincide on $\Le$. Finally, in the current context, Parseval's identity states 
\begin{align}
\sum_{x \in \Le} \eg^2 X(x)\, Y(x) &= \frac{1}{4} \sum_{\om \in \Z^2} \hat{X}(\om) \, \overline{\hat{Y}(\om)} \notag\\
&= \int_{\T^2} (\Ex \,X)(x) \, (\Ex \,Y)(x) \, dx\;. \label{e:Pars}
\end{align}
Below, we will use these nice identities to freely jump from discrete to continuous expressions, depending on which one is more convenient. When no confusion is possible, we  omit the operator $\Ex$ and simply use the same symbol for a function $\Le \to \R$ and its extension.

\subsection{Main result}

For any metric space $\Ss$, we denote by $\Dd(\R_+,\Ss)$ the space of $\Ss$ valued cadlag function endowed with the Skorokhod topology (see \cite{bill} for a discussion of 
this topology). For any $\al>0$ we denote by $\Ca$ the Besov space $B^{-\al}_{\infty, \infty}$ discussed in Appendix~\ref{sec:Besov}.

We denote by $X$ the solution of the renormalised limiting SPDE \eqref{e:SPDE}, discussed in Section~\ref{sec:ContinuousAnalysis} for a fixed initial datum $\Xn \in \Ca$. This process $X$ is  continuous taking values in $\Ca$.

Assume that for $\ga>0$, the spin configuration at time $0$ is given by $\sigma_\gamma(0, k), \; k \in \LN$, and define for $x \in \Le$ 
\begin{align*}
\Xng(x) = \dg^{-1}\sum_{y \in \Le} \eg^2 \Kg(x -y) \,\sigma_\gamma(0, \eg^{-1}y ) \;.
\end{align*}
We extend $\Xng$ to a function on all of $\T^2$ as a trigonometric polynomial of degree $\leq N$ still denoted by $\Xng$. Let $\Xg(t,x), \, t \geq 0, \, x \in \Le^2$ be defined by \eqref{e:defXg} and extend $\Xg(t, \cdot)$ to $x \in \T^2$ as a trigonometric polynomial of 
degree $\leq N$, still denoted by $\Xg$. 
%
%
%%%%%%%%%%%%%%%%%%%%%
%%%%%%%%%%%%%%%%%%%%%
\begin{theorem}\label{thm:Main}
%%%%%%%%%%%%%%%%%%%%%
%%%%%%%%%%%%%%%%%%%%%
Assume that the scaling assumptions \eqref{e:scaling1}, \eqref{e:scaling2} and \eqref{e:scalingbeta} hold, where the precise value of $\CGG$ is given by 
\begin{align}
\CGG =\frac{1}{4}\sum_{\substack{\om \in \{-N, \ldots, N \}^2\\ \om \neq 0 }} \frac{|\hKg(\om)|^2}{ \ga^{-2}  (1 - \hKg(\om))} \;.
\label{e:valueCGG}
\end{align}
Assume also that $\Xng$ converges to $\Xn$ in $\Ca$ for $\al>0$ small enough and that  $\Xn, \, \Xng$ are uniformly bounded in  $ \Cc^{-\al +\ka}$ for an arbitrarily small $\ka>0$. Then $\Xg$ converges in law to $X$ with respect to the topology of $\Dd(\R_+, \Ca)$. 
%%%%%%%%%%%%%%%%%%%%%
%%%%%%%%%%%%%%%%%%%%%
\end{theorem}
%%%%%%%%%%%%%%%%%%%%%
%%%%%%%%%%%%%%%%%%%%%

\begin{remark}
In principle one can perform the analysis that leads to \eqref{e:scaling1} in any spatial dimension $n$. Indeed, the only necessary change is to replace the $\eg^2$ terms appearing in \eqref{e:QuadrVar2} by $\eg^n$, so that one wishes to impose $1 \approx \frac{\eg^2}{ \ga^2}\frac{1}{\ag}  \approx \frac{\dg^2}{\ag} \approx \frac{\eg^n}{\dg^2\ag }$. In this way, 
one obtains the scaling relation
\begin{equation}\label{e:scalingn}
\eg \approx \ga^{\frac{4}{4-n}}, \quad \qquad \ag\approx \ga^{\frac{2n}{4-n}}, \quad  \qquad \dg \approx \ga^{\frac{n}{4-n}} \;.
\end{equation}
This relation was already obtained in \cite{GLP} and 
for $n=1$ it is indeed the scaling used by \cite{BPRS,FR}.  For $n=3$, we expect that it is possible to combine the techniques developed in this article with the theory developed in 
\cite{Martin1,MassimilianoPHI4} to get a convergence result to the dynamic $\Phi^4_3$ model. For $n=4$, relation \eqref{e:scalingn} cannot be satisfied. This corresponds exactly to the fact that 
the $\Phi^4_4$ model fails to satisfy the \emph{subcriticality condition} in \cite{Martin1} and indeed, a limiting theory is not expected to exist  for $n \geq 4$ (see \cite{Aizenman}).
\end{remark}

\begin{remark}
We stress that our analysis is purely based on the dynamics of the system; we do not rely on properties of the invariant measure $\lbg$.
\end{remark}

\begin{remark}\label{rem:ApproxInt}
A simple analysis  (see e.g. \cite[Chapter 2.1]{Numerical}) of the the trapezoidal approximation
 $\ct^{-1} = \sum_{k \in \LN }\ga^2 \, \KK(\ga k)$ to $\int \KK(x) \, dx =1$ shows that under our $\Cc^2$ assumption on $\KK$ we have 
 $ | \ct - 1 | \leq C \ga^2$. If we were to replace $\ct$ by $1$ in \eqref{e:samplek}, this error term of the form $C\ga^2$ would not disappear in our scaling, but would produce an $O(1)$ contribution to 
 the ``mass" $A$ in the limiting equation.  But this effect could be removed under slightly different assumptions: If we assumed that $\KK$ is $\Cc^\infty$ and in addition removed the assumption $\kg(0)=0$, 
 then by the Euler-Maclaurin formula (see \cite[Chapter 2.9]{Numerical}) we could get an arbitrary polynomial 
 rate of convergence. Under these modified assumptions, setting $\ct=1$ would not change the result.
\end{remark}

\begin{remark}
The condition $\kg(0)=0$, which  makes the analysis of the invariant measure of the Markov process very convenient, causes some minor technical problems. Much of our analysis is performed in Fourier space, and due to this condition the Fourier transform $\hKg(\om)$ of $\Kg$ decays at most like $\frac{C}{|\ga \om|^2}$ for large $\om$ (see Lemma~\ref{le:Kg} below), whereas without this condition (and with a stronger regularity assumption on $\KK$) one could obtain $\frac{C_m}{|\ga \om|^m}$ for any $m \geq 1$. Fortunately, this only produces some irrelevant logarithmic error terms. 
\end{remark}

\begin{remark}
In order to state our result, we have made two assumptions that may seem debatable. On the one hand, we have chosen to define the coarse-graining $\hg$ in terms of the same kernel $\kg$ that determines the interaction. On the other hand, we have extended the field $\Xg$ as a trigonometric polynomial. 

The reader will see below that the first choice is necessary in order to get control on Fourier modes $\om$ satisfying $\ga^{-1} \ll |\om| \ll \ga^{-2}$. The second choice is convenient but essentially irrelevant. 

 A posteriori, it is not too hard to show that even without any coarse-graining, the evolution 
\begin{equation*}
\mathbf{X}_{\ga}(t,\varphi) = \sum_{x \in \Le} \eg^2   \varphi(x) \, \dg^{-1}\si(t/\ag, x/\eg)  ,
\end{equation*}
viewed as an evolution in the space of distributions $\mathcal{S}^{\prime}(\T^2)$ converges in law to the same limit, but only with respect to the weaker topology of $\Dd(\R_+,\mathcal{S}^{\prime}(\T^2))$. 
\end{remark} 

\begin{remark}
Of course, there are many  choices other than \eqref{e:rate} for a rate $\cg$ to define a Markov process on $\SN$ that is reversible with respect to $\lbg$. The Glauber dynamics is a standard choice, but it would be possible to extend our proof to a more general jump rate. We make strong use of the fact that $\cg(\si,k)$ is a function of $\hg(\si,k)$, and of the specific form of the Taylor expansion around $\hg(\si,k) =0$ (see \eqref{e:Lh}). We furthermore use the fact that $\cg$ is bounded by $1$.
\end{remark}

%%%%%%%%%%%%%%%%%%%%%%%%
\section{Analysis of the limiting SPDE}
%%%%%%%%%%%%%%%%%%%%%%%%
\label{sec:ContinuousAnalysis}
%%%%%%%%%%%%%%%%%%%%%%%%
%
%
As stated above, a well-posedness theory for the limiting equation \eqref{e:SPDE} was provided in \cite{dPD}. More precisely, in that article
local in time existence and uniqueness for arbitrary $\Ca$ initial data and for $\al >0$ small enough was shown (recall the definition of the Besov space $\Ca$ in Appendix~\ref{sec:Besov}). Furthermore, it was shown, using an idea due to Bourgain, that global well-posedness holds for \emph{almost every} initial 
datum with respect to the invariant measure. In our companion paper  \cite{JCH}, we extend these results to show \emph{global} well-posedness for \emph{every} initial datum in $\Ca$ for $\al>0$ small enough. In \cite{JCH}, we also show how to extend these arguments from the two-dimensional torus to the full space $\R^2$, but this extension is not relevant for the present article. 
The purpose of this section is to review the relevant results from \cite{dPD, JCH}. We give some ideas of the proofs where it helps the reader's understanding of our argument in the more complicated discussion of the discrete system that follows.

As it stands, it is not clear at first sight how equation~\eqref{e:SPDE} should be interpreted. Indeed, the space-time white noise is a quite irregular distribution, and it is well-known that the regularising property of the heat semigroup is not enough to turn the solution of the linearised equation (the one obtained by removing the $X^3$ term from \eqref{e:SPDE} ) into a function. We will see below that this linearised solution takes values in all of the distributional Besov spaces of negative regularity, but not of positive regularity. We cannot expect the solution $X$ to the non-linear problem to be more regular, and hence it is not clear how to interpret the non-linear term $X^3$. 
 
 A naive approach consists in approximating solutions to \eqref{e:SPDE} by solutions to a regularised equation. More precisely, let $\Xe$ be the solution to the stochastic PDE
 \begin{equation}\label{e:SPDEeps1}
 d\Xe = \Big(  \Delta \Xe - \frac13 \Xe^3 + A \Xe \Big)\, dt  + \sqrt{2} \, dW_\eps \;.
 \end{equation}
Here $W_\eps(t,x) = \frac{1}{4}\sum_{|\om| < \eps^{-1}} e^{i \pi \om \cdot x} \,\hat{W}(\om,t) $  is a spatially regularised cylindrical Wiener process. For every $\om \in \Z^2$ the process $\hat{W}(\om,t)$ is a complex valued Brownian motion with $\E|\hat{W}(\om, t)|^2  = 4 t$. These Brownian motions  are independent except for the constraint $\hat{W}(\om,t) = \overline{\hat{W} (-\om, t)}$ for all $\om \in \Z^2$ and $t \geq 0$. We choose to regularise the noise by considering a cut-off in Fourier space, but this choice of regularisation is inessential. A solution to equation~\eqref{e:SPDEeps1} for fixed value of $\eps>0$ can be constructed with standard methods,  see e.g.  \cite{dPZ, HairerIntro, prevot}. It seems natural to study the behaviour of these regularised solutions as $\eps$ goes to~$0$. Unfortunately, letting $\eps$ go to zero yields a trivial result. Indeed, it is shown in \cite{HRW} that $\Xe$ converges to zero in probability (in a space of distributions). 

In order to obtain a non-trivial result, the approximations \eqref{e:SPDEeps1} have to be modified. Indeed, it is shown in \cite{dPD} that if instead of \eqref{e:SPDEeps1}, we consider
 \begin{equation}\label{e:SPDEeps}
 d\Xe = \Big(  \Delta \Xe - \Big(\frac13 \Xe^3  -  \Ce \Xe \Big) + A \Xe \Big)\, dt  + \sqrt{2} \, dW_\eps,
 \end{equation}
 for a particular choice of constant $\Ce$, then a non-trivial limit can be obtained as $\eps$ goes to zero. Similar to  
 \eqref{e:valueCGG}, the precise value of $\Ce$ is given by
 \begin{equation}\label{e:norm-constant}
 \Ce = \sum_{ 0 <|\om| < \eps^{-1} } \frac{1}{ 4 \pi^2 |\om|^2} \;.
 \end{equation}
In particular,  the constants $\Ce$ diverge logarithmically as $\eps \to 0$.

As a first step to show this convergence, equation~\eqref{e:SPDEeps} is linearised. Let $\Ze$ be the unique mild solution to 
 \begin{align}
 d\Ze(t,x) &= \Delta \Ze(t,x) \, dt + \sqrt{2} \,dW_\eps(t,x) \notag\\
 \Ze(0,x)&=0\;, \label{e:linear}
 \end{align}
i.e. $\Ze(t,\cdot) = \sqrt{2} \int_0^t P_{t-s} \, dW_\eps(s, \cdot) $, where $P_t= e^{\Delta t}$ is the solution operator of the heat equation on the torus $\T^2$. Then the renormalisation is performed on the level of this linearised equation. We explain this procedure in some detail, because a similar study of a linearised version of \eqref{e:evolution2} constitutes the core of our argument.

We start by recalling that  the Hermite polynomials $H_n = H_n(X,T)$ are defined recursively by setting 
\begin{equation}
\label{e:def-Hermite1}
\left\{
\begin{array}{l}
H_0 = 1, \\
H_{n} = X H_{n-1} - T \, \partial_X H_{n-1} \qquad ( n \in \N),
\end{array}
\right.
\end{equation}
so that $H_1 = X$, $H_2 = X^2-T$, $H_3 = X^3 - 3 XT$, etc. 
One can check by induction that
\begin{equation}
\label{e:partialX}
\partial_X H_{n} = n H_{n-1}
\end{equation}
and
\begin{equation}
\label{e:partialT}
\partial_T H_{n} = - \frac{n(n-1)}{2} H_{n-2}
\end{equation}
(the identities are even valid for $n \in \{0,1\}$, for an arbitrary interpretation of $H_{-1}$ and $H_{-2}$).

Now,  for any fixed $\eps>0$,  $\Ze$ is a random continuous function, and there is no ambiguity in the definition of $\Ze^n$, but for $n \geq 2$ these random functions $\Ze^n$ fail to converge to random distributions as $\eps$ goes to zero.  If, however, $\Ze^n$ are  replaced by the Hermite polynomials  
\begin{equation*}
Z_{\eps}^{\colon n \colon}(t,x)  := H_n(\Ze(t,x), \Ce(t))
\end{equation*}
for 
\begin{align}
\Ce(t)& =  \E[Z_\eps(t,0)^2]  = \frac{1}{2}\sum_{ |\om| < \eps^{-1}} \int_0^t \exp\Ll(-2r\pi^2  |\om|^2\Rr) \, dr \;\notag\\
&= \frac{t}{2} + \sum_{0 < |\om| < \eps^{-1}} \frac{1}{4 \pi^2 |\om|^2 } \Big(1 - \exp\big( -2t \pi^2 \, |\om|^2  \big) \Big)   \;,\label{e:coft}
\end{align}
then a non-trivial limit is obtained. Note that $\Ce = \lim_{t \to \infty} ( \Ce(t) - \frac{t}{2}) $, where the term $\frac{t}{2}$ comes from the summand for $\om =0$ in \eqref{e:coft} which does not converge as $t \to \infty$. Furthermore, for every fixed $t>0$ the difference $|\Ce - \Ce(t)|$ is uniformly bounded in $\eps$.

The following result, essentially \cite[Lemma 3.2]{dPD}, summarises this convergence. 
\begin{proposition}\label{prop:daPratoDebussche}
For every $T>0$ and every  $\al>0$, the stochastic processes $\Ze$ and  $Z_{\eps}^{\colon n \colon }$ for $\, n \geq2$ converge almost surely  and in every stochastic $L^p$ space with respect to the metric of $\Cc([0, T],\Ca)$. We denote the limiting processes by $Z$ and $Z^{\colon n \colon}$. 
\end{proposition}
We outline an argument for Proposition~\ref{prop:daPratoDebussche} which is inspired by the treatment of a (more complicated) renormalisation procedure in \cite[Section 10]{Martin1}. We start with an alternative representation of the $Z_{\eps}^{\colon n \colon }$. As explained above we have $\Ze(t,\cdot) = \sqrt{2}\int_0^t P_{t-r} \; \, dW_\eps(r)$. It will be useful to introduce the processes 
\begin{align*}
R_{\eps,t}(s,x) =R_{\eps,t}^{\colon 1\colon}(s,x) = \sqrt{2} \int_0^s P_{t-r} \; dW_\eps(r,x) \;,
\end{align*} 
defined for $s\leq t$, and to define recursively 
\begin{equation*}
R_{\eps,t}^{\colon n\colon}(s,x) = n \int_{r=0}^s R_{\eps,t}^{\colon n-1\colon} (r,x) \, dR_{\eps,t}(r,x) \;.
\end{equation*}
It can be checked easily using It\^o's formula and the relations \eqref{e:partialX} and \eqref{e:partialT} that for any $\eps>0$, we have
\begin{equation*}
R_{\eps,t}^{\colon n\colon}(t,x) = \Ze^{\colon n\colon }(t,x)  \;.
\end{equation*}
For any smooth function $f \colon \T^2 \to \R$, the expectation
\begin{equation*}
\E \Ze^{\colon n\colon }(t,f)^2  := \E \Big(  \int_{\T^2} \Ze^{\colon n\colon }(t,x) \, f(x) \, dx  \Big)^2
\end{equation*}
can now be calculated explicitly via It\^o's isometry, and we obtain\footnote{To derive this formula we introduce the Fourier coefficients $\hRe(s,\om)$, $\hRE{n}(s,\om)$ ($\om \in \Z^2$) of $R_{\eps,t}(s,\cdot)$ and $R_{\eps,t}^{\colon n\colon}(s,\cdot)$ respectively, and observe that
$$
\hRe(s,\om) = \sqrt{2} \int_{r=0}^t \hat{P}_{t-r}(\om) \, d\hat{W}(\om,r) \qquad (|\om|\le \eps^{-1}),
$$
$$
\hRE{n}(s,\om) = n \sqrt{2} \int_{r=0}^s \frac{1}{4} \sum_{\om_1 + \om_2 = \om} \hRE{n-1}(r,\om_1) \hat{P}_{t-r}(\om_2) \, d\hat{W}(\om_2,r).
$$
For instance,
$$
\int_{\T^2} R_{\eps,t}(s,x) \, f(x) \, dx = \frac{\sqrt{2}}{4} \sum_{|\om|\le \eps^{-1}} \hat{f}(\om)\int_{r=0}^s \hat{P}_{t-r}(\om) \, d\hat{W}(\om,r).
$$
}

\begin{align*}
\E \Ze^{\colon n\colon }(t,f)^2  = n! \frac{2^n}{4^n} \int_{[0,t]^n}   \sum_{|\om_i| \leq \eps^{-1}}\big|  \hat{f}(\om_1 + \ldots + \om_n) \big|^2\; \prod_{j=1}^n \big| \hat{P}_{t-r_j}(\om_j) \big|^2 \; d\rrr \;,
\end{align*}
where $d\rrr= dr_1 \ldots dr_n$ and where $\hat{P}_{t}(\om) = \exp\Ll(- t\pi^2  |\om|^2\Rr)$ denotes the Fourier transform of the heat kernel on the torus.  This quantity converges as $\eps$ goes to zero, and the limit can be expressed as 
\begin{align*}
n! 2^n  \int_{[0,t]^n} \int_{(\T^2)^n}  \Big( \int_{\T^2}  f(x) \; \prod_{j=1}^n  P_{t-r_j}(x-z_j) \;dx \Big)^2  \; d\zz \, d\rrr \;,
\end{align*}
where $d\zz = dz_1 \ldots dz_n$. A crucial observation now uses the Gaussian structure of the noise and the fact that $ \Ze^{\colon n\colon }(t,f)$ is a random variable in the $n$-th homogenous Wiener chaos over this Gaussian noise (see \cite[Chapter 1]{Nualart} for a definition and properties of Gaussian Wiener chaos). According to Nelson's estimate (see \cite{Nelson} or \cite[Chapter 1.5]{Nualart}) the estimate on 
$\E \Ze^{\colon n\colon }(t,f)^2 $ can be turned into an equivalent estimate on $\E \Ze^{\colon n\colon }(t,f)^p$. Then one can specialise this bound to $f = \ek(u - \cdot)$ (defined in  \eqref{e:DEFek}), and apply Proposition~\ref{prop:Kolmogorov} to obtain bounds on 
$\E \| Z^{\colon n \colon}_{\eps}(t, \cdot) \|_{\Ca}^p$ that are uniform in $\eps$. 

For continuity in time, one can modify this argument to get bounds on $\E \| Z^{\colon n \colon}_{\eps}(t, \cdot) - Z^{\colon n \colon}_{\eps}(s, \cdot)  \|_{\Ca}^p$, and then apply the Kolmogorov criterion. 

In Sections~\ref{sec:linear} and \ref{sec:Tightness}, we will perform a similar argument for a linearised version of the evolution equation \eqref{e:evolution2}. One obstacle we need to overcome, is that without the Gaussian structure of the noise, Nelson's estimate is not available. In Lemma~\ref{le:ItInt}, we replace it by a suitable version of the Burkholder-Davis-Gundy inequality. The price we have to pay is that various error terms caused by the jumps need to be controlled. 

It is useful to note that for fixed values of $s$ and $t$, $s < t$, the processes $R_{\eps,t}^{\colon n\colon}(s,\cdot) $ actually converge in nicer spaces than the $\Ca$. Indeed $R_{\eps,t}(s, \cdot) = P_{t-s} Z_{\eps}(s,\cdot)$, and by the convergence of $Z_\eps(s)$ in $\Ca$ and standard regularising properties of the heat semigroup (see e.g. \eqref{e:heat-semigroup-reg} and the discussion following it) $R_{\eps,t}(s,x)$ converges to 
$$
R_t(s,x) := \sqrt{2} \int_{r=0}^s  P_{t-r} \,  dW (r,x) \;,
$$
in $\Cc^k$ for every $k \in \N$. In the same way, $R^{\colon n \colon}_{t,\eps}$ converges to 
\begin{equation}\label{e:Wick0}
\RR{n}(s,x) =  n \int_{r=0}^s \RR{n-1}(r,x) \; d R_t(r,x) =  H_n\Ll(R_t(s,x), \langle R_t(\cdot,x) \rangle_s\Rr)  \;.
\end{equation}
Here the quadratic variation of the continuous martingale $s \mapsto R_t(s,x)$ for $s < t$ is given by  
\begin{eqnarray}
\langle R_t(\cdot,x) \rangle_s & = & 2 \int_0^s P_{2(t-r)}(0) \, dr \notag \\
& = & \frac{1}{2} \sum_{\om \in \Z^2} \int_0^s \exp\Ll(-2(t-r) \pi^2 |\om|^2\Rr) \, dr \;, \notag
\end{eqnarray}
where $P_t(x)$ is the heat kernel associated with the semigroup $P_t$.

Finally, note that the convergence of $\RR{n}(t,\cdot)$ to  $\ZZ{n}(t,\cdot) $ in $\Ca$ can be quantified. Using an argument in the same spirit as the one sketched above, one can see that  for all $\al>0, \; 0 \leq  \la \leq 1, \; p \geq2$ and $T >0$,  there exists $C= C(\al,\la, p,T)$ such that 
\begin{equation}
\label{e:regR0n}
\E\| \ZZ{n}(t, \cdot) - \RR{n}(s,\cdot) \|^p_{\Cc^{-\al-\la}} \le C |t-s|^{\frac{\lambda p}{2}}
\end{equation}
for all $0 \leq s \leq t \leq T$. A similar bound in the more complicated discrete situation is derived below in  \eqref{e:RgRegularity3}.

In order to study the convergence of the  the non-linear equations \eqref{e:SPDEeps}, we study the remainder $v_\eps := X_\eps - Z_\eps$. For $\eps>0$, we observe that $v_\eps$ is a solution to the random partial differential equation 
\begin{align}
\partial_t v_\eps(t,x) &= \Delta v_\eps - \Big( \frac13 (v_\eps + \Ze)^3 -  \Ce (v_\eps + \Ze)  \Big) + A (v_\eps + \Ze) \notag \\
&= \Delta v_\eps - \big(v_\eps^3  + 3 v_\eps^2 \Ze + 3v_\eps \Ze^{\colon 2 \colon} + \Ze^{\colon 3\colon }  \big) + A_\eps(t) (v_\eps + \Ze)\;, \label{e:remainder}
\end{align}
where we have set $A_\eps(t) := A + \Ce - \Ce(t)$. 
Note that the noise term $dW_\eps$ has disappeared from equation \eqref{e:remainder}. Note furthermore that in the second line, we have rewritten the right-hand side in terms of the processes $\Ze,\Ze^{\colon 2 \colon},$ and $\Ze^{\colon 3 \colon}$, which converge to a non-trivial limit as $\eps$ goes to zero. This is possible due to the relation 
\begin{equation*}
H_n(z+v,c) = \sum_{k=0}^n {n \choose k}    H_k(z,c) \, v^{n-k},
\end{equation*}
which holds for arbitrary $z,v \in \R$ and $c>0$.

Equation \eqref{e:remainder} can be treated as a normal PDE, without taking into account stochastic cancellations or stochastic integrals. The argument in \cite{dPD} is  concluded by observing that (at least for small times) the solutions of \eqref{e:remainder} are stable under approximation of the functions $\Ze,\Ze^{\colon 2 \colon},$ and $\Ze^{\colon 3 \colon}$ in  $\Cc([0, \infty),\Ca)$ as well as the limit of $A_\eps$ as $\eps \to 0$. In this way, local in time solutions are obtained in \cite{dPD}. 

The authors then show that these solutions do not blow up in finite time for almost every initial datum with respect to the invariant measure. For our purposes, it is slightly  more convenient to have global existence for \emph{every} initial datum in $\Ca$ for $\al>0$ small enough, and we show this in the forthcoming article \cite{JCH}. In order to state the main result, it is necessary to briefly discuss the role of the initial datum $\Xn \in \Ca$. 

There are essentially two possibilities --  either equation \eqref{e:linear} is started with $\Xn$, in which case the initial datum for \eqref{e:remainder} is zero; or the linear equation is started with zero, and \eqref{e:remainder} is started with $\Xn$. The first option turns out to be slightly more convenient. Hence, we define $Y(t):=P_t \Xn$  and set
\begin{align}
 \widetilde{Z}(t, \cdot) &=Y(t, \cdot) + Z(t, \cdot)\;, \notag\\
\widetilde{Z}^{\colon 2 \colon} (t, \cdot)  &=  Z^{\colon 2 \colon} (t, \cdot) \colon + 2Y(t, \cdot) Z(t, \cdot)  +  Y(t,\cdot)^2\;,\notag \\
 \widetilde{Z}^{\colon 3 \colon} (t, \cdot) &=  Z^{\colon 3 \colon} (t, \cdot) \colon + 3Y(t, \cdot)  Z(t, \cdot)^{\colon 2 \colon} +  3Y(t, \cdot)^2 Z(t, \cdot)  +  Y(t,\cdot)^3\;.\label{e:Z_with_ic}
\end{align}
Notice that by the regularisation property of the heat semigroup (see \eqref{e:heat-semigroup-reg}), for any $t >0$, $Y(t,\cdot)$ is actually a smooth function and, in particular, the products appearing in these expressions are well-defined. More precisely, for every $\beta>\al$, there exists a constant $C=C(\al,\beta)$ such that for every $t >0$, we get $\| Y(t, \cdot) \|_{\Cc^\beta}  \leq C t^{-\frac{\al+\beta}{2}}  \| \Xn \|_{\Cc^{-\al}}$, and hence $\| Y(t, \cdot)^2 \|_{\Cc^{-\al}}  \leq C t^{-\frac{\al+\beta}{2}} \| \Xn \|_{\Cc^{-\al}}^2$, $\| Y(t, \cdot)^2 \|_{\Cc^\beta}  \leq C t^{-(\al+\beta)} \| \Xn \|_{\Cc^{-\al}}^2$ and $  \| Y^3(t, \cdot) \|_{\Cc^{-\al}} \leq C t^{-(\al+\beta)} \| \Xn \|_{\Cc^{-\al}}^3$, where we use the fact that $\Cc^{\beta}$ is an algebra for $\beta >0$ as well as the multiplicative inequality, Lemma~\ref{le:Besov-multiplicative}. Using Lemma~\ref{le:Besov-multiplicative} once more, we can conclude that for every $T>0$, there exists a \emph{random} constant $C_0$ (depending on $T,\al, \beta, \| \Xn\|_{\Ca}$ and on the particular realisation of $Z$, $ Z^{\colon 2\colon}$, and $\colon Z^{\colon 3 \colon}$)  such that
\begin{align}
\sup_{0 \leq t \leq T} \|  \widetilde{Z}(t, \cdot) \|_{\Cc^{-\al}}& \leq C_0\;, \qquad
\sup_{0 \leq t \leq T} t^{\frac{\beta+\al}{2}} \|   \widetilde{Z}^{\colon 2\colon}(t, \cdot)  \|_{\Cc^{-\al}} \leq  C_0\;,\notag\\
\sup_{0 \leq t \leq T}t^{\beta+\al}  \|  \widetilde{Z}^{\colon 3\colon}(t, \cdot)   \|_{\Cc^{-\al}} &\leq  C_0\; \label{e:tildeZRegularity}.
\end{align}
After this preliminary discussion, we are now ready to state the main existence result. For $\tZ, \tZz$,  and $\tZd$ satisfying \eqref{e:tildeZRegularity}, consider the problem. 
\begin{align}
\partial_t v &= \Delta v - \frac13 \big(v^3 + 3 \tZ v^2 + 3  \tZz v +  \tZd  \big)  + A(t) (\tZ + v)\;,\notag\\
v(0,\cdot) &= 0 \;, \label{e:vequation}
\end{align}
where
\begin{align}
A(t) &:= A + \lim_{\eps \to 0} (\Ce -\Ce(t)) 
= A- \frac{t}{2}  + \sum_{\om \in \Z^2 \setminus\{ 0\}}  \frac{e^{- 2t\pi^2 |\om|^2}}{ 4\pi^2 |\om|^2} \;.\label{e:A_value}
\end{align}
Note that $A(t)$ only diverges logarithmically  in $t$ as $t$ goes to $0$, and in particular any power of $A(t)$ is integrable at $0$.
The following theorem is essentially \cite[Theorem~6.1]{JCH}. (The continuity of the solution map is not stated explicitly there, but it is contained in the method of proof). 
\begin{theorem}\label{thm:ContSolution}
For $\al>0$ small enough, fix an initial datum $\Xn \in \Ca$. For $(Z, Z^{\colon 2\colon }, Z^{\colon 3 \colon })  \in\big(  L^{\infty}\big([0,T], \Ca \big)\big)^3$, let $(\tilde{Z},\widetilde{Z}^{\colon 2 \colon}, \widetilde{Z}^{\colon 3 \colon}) $ be defined as in \eqref{e:Z_with_ic}.  Let $\Ss_T(Z, Z^{\colon 2\colon }, Z^{\colon 3 \colon })$ denote the solution $v$ on $[0,T]$ of the PDE \eqref{e:vequation}. Then for any $\kappa >0$, the mapping $\Ss_T$ is Lipschitz continuous on bounded sets from $\big(  L^{\infty}\big([0,T], \Ca \big)\big)^3$ to $\Cc([0,T], \Cc^{2 - \al - \kappa}(\T^2) )$.
\end{theorem}

\begin{remark}
The choice of two different time dependent normalisation constants $\Ce$ and $\Ce(t)$ may not seem particularly elegant.  Indeed, in \cite{dPD, Martin1}, all processes are renormalised with time independent constants. This is possible because in those papers, the processes  $\Ze$ and  $ Z_{\eps}^{\colon n \colon}$ are replaced by similar processes that are \emph{stationary} in $t$. This can be done by adding a linear damping term and moving the initial condition in \eqref{e:linear}  to $t= -\infty$ (as in \cite{dPD}), or by cutting off the heat kernel $P_t$ outside of a ball  (as in \cite{Martin1}). In the discrete setting below, however, the choices for the linearised process presented here seem most convenient.

This discussion shows as well that there is no canonical choice of the renormalisation constant $\Ce$ and hence no canonical value of the constant $A$. The different procedures in \cite{dPD} and \cite{Martin1} yield different choices of $\Ce$. The difference between these two constants remains bounded as $\eps$ goes to zero, but it does not disappear in the limit.
\end{remark}

%%%%%%%%%%%%%%%%%%%%%%%%
\section{Bounds for the linearised system}
%%%%%%%%%%%%%%%%%%%%%%%% 
\label{sec:linear}
%%%%%%%%%%%%%%%%%%%%%%%%
We now come back to the study of the discrete system. By Duhamel's principle and using the scaling relations \eqref{e:scaling1} and \eqref{e:scalingbeta}, the evolution equation \eqref{e:evolution2} can be rewritten as 
\begin{align}
\Xg(t,\cdot) =&\Pg{t} \Xng + \int_0^t \Pg{t-r} \, \Kg \ae   \bigg( - \frac{\beta^3}{3}   \Xg^3  (r,\cdot) +  (\CGG +A)    \Xg(r,\cdot)
\notag \\
& + \Eg(r,\cdot) \bigg) \, dr  + \int_{r=0}^t  \,\Pg{t-r} d  \,\Mg(r, \cdot) \qquad\text{on }  \Le \;,    \label{e:evolution3} 
\end{align}
where we use the convention $\int_0^t = \int_{(0,t]}$, and we denote by $\Pg{ t} = e^{ \Dg  t}$ the semigroup generated by $ \Dg$. For every $t \geq 0$, the operator $\Pg{t}$ acts on a function $Y \colon \Le \to \R$ by convolution with a kernel, also denoted by $\Pg{t}$. 
This kernel is characterised by its Fourier transform (defined as in \eqref{e:FT})
\begin{equation}
\label{e:Fourier-semi}
\hPg{t}(\om) = \exp\Big(t \gamma^{-2} (\hKg(\om) - 1) \Big) \qquad \text{if } \om \in \{ -N, \ldots, N \}^2  \; .  
\end{equation} 
Viewing $\Pg{t}$ as a Fourier multiplication operator (with $\hPg{t}(\om) = 0$ if $\om  \notin \{ -N, \ldots, N\}^2$) enables to make sense of $\Pg{t} f$ for every $f : \T^2 \to \R$.
Further properties of the operator $\Pg{t}$ are summarised in Lemmas~\ref{le:Pgt} and~\ref{le:semi-group-regularity} as well as Corollary~\ref{cor:regPg}.

As explained above for the continuous equation, a crucial step in studying the limiting behaviour of $\Xg$ consists of the analysis of a linearised evolution. For $x \in \Le$, we denote by 
\begin{equation*}
\Zg(t,x) = \int_{r=0}^t  \Pg{t-r} \, d\Mg (r,x) 
\end{equation*}
the stochastic convolution appearing on the right-hand side of \eqref{e:evolution3}. The process $\Zg$ is the solution to the linear stochastic equation
\begin{align}
d\Zg(t,x) &= \Dg \Zg(t,x) dt + d\Mg(t,x)\notag\\
 \Zg(0,x)& = 0 \, , \label{e:DefZg}
\end{align}
for $x\in \Le,\: t \geq 0$. 
It will be convenient to work with the following family of approximations to $\Zg(t,x)$. For $s \leq t$, we introduce
\begin{equation*}
\Rg(s,x) := \int_{r=0}^s  \Pg{t-r} \, d\Mg (r,x) \;.
\end{equation*} 
As explained above (see the discussion following \eqref{e:FI}), we extend $\Rg(s, \cdot) \colon \Le \to \R$ and $\Zg( t , \cdot) \colon \Le \to \R$ to  functions on all of $\T^2$ by trigonometric polynomials of degree $\leq N$. Note that for any $t$ and any $x \in \T^2$, the process $\Rg(\cdot,x)$ is a martingale and $\Rg(t,\cdot) = \Zg(t,\cdot)$.

As in the case of the continuous process, it is not enough to control $\Zg$, the solution of the linearised evolution: we also need to control additional non-linear functions thereof. We introduce recursively the following quantities:
for a fixed $t\geq 0$ and $x \in \T^2$, we set $\RG{1}(s,x) = \Rg(s,x)$. For $n \geq 2, \; t \geq 0$  and $x \in \Le$, we set
\begin{equation}\label{e:ZnA}
\RG{n}(s,x) =  n \int_{r=0}^s \RG{n-1}(r^-,x) \; d \Rg(r,x)\;.
\end{equation}
We use the notation $\RG{n-1}(r^-,x)$ to denote the left limit of $\RG{n-1}(\cdot,x)$ at $r$. This definition ensures that $(\RG{n}(s,x))_{0 \le s \le t}$ is a martingale. To define an extension of $\RG{n}(s, \cdot)$ to arguments $x \in \T^2 \setminus \Le$ for $n \geq 2$, it is advisable \emph{not} to extend it by a trigonometric polynomial of degree $\leq N$. Indeed, products are not well captured by this extension. It is more natural to define the extension recursively through its Fourier series
\begin{equation}\label{e:ZnB}
\hRG{n}(s, \om) := n\int_{r=0}^s \frac{1}{4}\sum_{\tilde{\om} \in \Z^2} \hRG{n-1}(r^-, \om- \tilde{\om}) \; d\hRg(r, \tilde{\om})\;,
\end{equation} 
and set $\RG{n}(s,x) := \frac14 \sum_{\om \in \Z^2}\hRG{n}(s,\om)e^{i \pi \om \cdot x}$. This definition coincides with \eqref{e:ZnA} on $\Le$, and for every $n \geq 2$ the function $\RG{n}(s, \cdot) \colon \T^2 \to \R$ is a trigonometric polynomial of degree $\leq nN$. For any $n \geq 2$ and for $t \geq 0$, $x \in \T^2$ we define
\begin{equation}\label{e:DefZn}
\ZG{n}(t,x) := \RG{n}(t,x) \;.
\end{equation}

The main objective of this section is to prove uniform bounds on the Besov norms of the processes $\ZG{n}$ and $\RG{n}$. These bounds are stated in Proposition \ref{prop:RGBoundOne}. 

As a first step, we derive a general bound on $p$-th moments of iterated stochastic integrals. We start by introducing some more notation: Let $F \colon [0,\infty)^n \times \Le^n \times \Omega \to \R$ be adapted and left continuous in each of the $n$ ``time" variables. By adapted, we mean that if $s_1, \ldots, s_n \leq t$, then for all $y_1, \ldots,y_n$, the random variable $F(s_1, \ldots, s_n,  y_1, \ldots, y_n)$ is measurable with respect to the sigma algebra generated by $\Xg$ up to time $t$.  We recursively define iterated integrals $\I{n} F(t)$ as follows. For $n=1$, we set $\I{1}F(t) = \int_{r=0}^t \sum_{y \in \Le} \eps^2 F(r,y) \, d M(r,y) $. For $n \geq 2$, we set
\begin{align*}
\I{n} F(t) := \int_{r_1 =0}^t  \sum_{y_1 \in \Le} \eps^2 \, \I{n-1} F^{(r_1,y_1)} (r_1^-) \, dM(r_1,y_1) \;,
\end{align*}
where $F^{(r_1,y_1)} \colon [0,\infty)^{n-1} \times \Le^{n-1}$ is defined as 
\begin{equation*}
F^{(r_1,y_1)}(r_2, \ldots ,r_n, y_2, \ldots , y_n) = F(r_1, \ldots ,r_n , y_1 , \ldots, y_n)\;.
\end{equation*}
As above (but somewhat abusively here), we denote by $\I{n-1} F^{(r_1,y_1)} (r_1^-)$ the left limit of $r_1 \mapsto \I{n-1} F^{(r_1,y_1)} (r_1)$. With this definition, for every $n$ and $F$, the process $t \mapsto \I{n} F(t)$ is a martingale. 
Finally, given any $F$ as above and $1 \leq \ell \leq n$, we define
\begin{align}
\label{e:def:Fl}
F_\ell(r_1&, \ldots,  r_{n},  z_1 \ldots, z_\ell , y_{\ell+1}, \ldots y_n) \\
& :=\sum_{y_1, \ldots, y_\ell  \in \Le^\ell} \eps^{2\ell}  F(r_1, \ldots , r_n, y_1 \ldots , y_n) \prod_{i=1}^\ell \Kg(y_i - z_i) \;,
\end{align} 
i.e. $F$ is convolved with the kernel $\Kg$ in the first $\ell$ spatial arguments.  In the sequel, when we write $\I{n-\ell} F_\ell(r_1, \ldots, r_\ell; z_1, \ldots z_\ell)$ this means that the iterated stochastic integral is taken with respect to the variables $r_{\ell+1}, \ldots ,r_{n}$ and $y_{\ell+1}, \ldots, y_n $ and evaluated at time $t= r_\ell$ (the variables $r_1,\ldots, r_{\ell}, z_{1},\ldots,z_{\ell}$ are just treated as fixed parameters). 
Using the definition \eqref{e:ZnA}, it is easy to see that for $x \in \Le$ and $0 \leq s \leq t$, we can write $\RG{n}(s,x) = \I{n}F(s)$  for 
\begin{align}
 F(r_1, \ldots, r_n, y_1, \ldots, y_n) = n! \prod_{i=1}^n \Pg{t-r_i}(x - y_i)\;.  \label{e:itconv}
\end{align}
Furthermore, note that for every $n$, the mapping $F \mapsto \I{n}F$ is linear in $F$.
\begin{lemma}\label{le:ItInt}
Let $n \geq 1$ and let $F \colon [0,\infty)^n \times \Le^n \to \R$ and $F_\ell$ for $1 \leq \ell \leq n$ be deterministic and left-continuous in each time variable. Then for any $p \geq 2$, there exists a constant $C = C(n,p)$ such that 
\begin{align}\label{e:itintBound}
\bigg(\E \sup_{0 \leq r \leq t} \big| \I{n}F(r) \big|^p \bigg)^{\frac{2}{p}}  \leq C \int_{r_1=0}^t \!\ldots \! \int_{r_n=0}^{r_{n-1}}   \sum_{\substack{\zz \in \Le^n}} \eps^{2n}  \; F_n(\rrr,\zz)^2 \;  d\rrr + \ER{n} \;,
\end{align}
where we use the short-hand $ \rrr = (r_1, \ldots r_n)$,   and $ d\rrr = dr_{n}\cdots  d{r_2}\, d{r_1}$. 
The error term $\ER{n}$ is given by 
\begin{align}
\ER{n} =& C  \;\eps^4 \delta^{-2}
  \sum_{\ell=1}^{n}     \int_{r_1=0}^t  \!\ldots\! \int_{r_{\ell-1}=0}^{r_{\ell-2}} \; \sum_{z_1, \ldots , z_{\ell-1} \in \Le } \!\!\eps^{2(\ell-1)} \notag \\
& \qquad \qquad \qquad   \bigg( \E \sup_{\substack{0 \leq r_\ell \leq r_{\ell-1}\\z_\ell \in \Le  }}  \big| \I{n-\ell} F_\ell(\rrr_\ell, \zz_\ell) \big|^p   \bigg)^{\frac{2}{p}} \; d\rrr_{\ell-1}  \;,\label{e:itintError}
\end{align}
where $ \rrr_\ell = (r_1, \ldots r_\ell)$, $\zz_\ell  =( z_1, \ldots , z_\ell)$,  $ d\rrr_{\ell-1} = dr_{\ell-1}\cdots  d{r_2}\, d{r_1}$ (there is no variable to integrate for $\ell = 1$), and $r_0 =t$. 
\end{lemma}
\begin{proof}
We proceed by induction. Let us consider the case $n=1$ first. In order to apply the Burkholder-Davis-Gundy inequality (Lemma~\ref{le:BDG}), we need to bound the quadratic variation as well as the size of the jumps of the martingale $\I{1}F(t)$. The quadratic variation is given by 
\begin{align}
 \big\langle &\I{1} F \big\rangle_t  \notag\\
&=  \int_{r=0}^t  \sum_{\substack{y \in \Le\\ \bar{y} \in \Le}} \eg^4 \, F(r,y) \, F(r,\bar{y}) \;  d \langle  \Mg(\cdot, y) , \Mg( \cdot, \bar{y}) \rangle_r \, \notag\\
&= 4 \co^2 \int_{0}^t  \sum_{\substack{y \in \Le\\ \bar{y} \in \Le}} \eg^4 \, F(r,y) \, F(r,\bar{y})  \sum_{z \in \Le} \eg^2 \Kg(y-z) \,\Kg(\bar{y}-z)  \, \Cg(r,z) \,dr \notag\\
&=  4 \co^2 \int_0^t    \sum_{z \in \Le} \eg^2 \,  \Big( \sum_{y \in \Le}\eps^2 F(r,y) \, \Kg(y-z)   \Big)^2 \, \Cg(r,z) \, dr \notag\\
& \leq 4 \co^2 \int_0^t    \sum_{z \in \Le} \eg^2 \,  F_1(r,z)^2  \, dr \label{e:ItInt1}\;.
\end{align}
Here we have used \eqref{e:QuadrVar2} for the second equality and the deterministic estimate $0 \leq \Cg(r,z) \leq 1$ in the last inequality.
Let us now turn to the jumps. We had seen above that a jump of the spin $\sigma(k)$ at microscopic position $k = \eps^{-1} z$ causes a jump of size $2 \delta^{-1} \eps^2 \Kg(y-z)$ for $\Mg$. With probability one, two such jumps never occur at the same time, so that we only have to estimate the impact of such an event on $\I{1}F $. If such an event takes place at (macroscopic) time $r$, then the martingale $\I{1}F $ has a jump of absolute value 
 \begin{align}\label{e:ItInt2}
 2\eps^{2}\delta^{-1} &\Big| \sum_{y \in \Le} \eps^2 F(r,y)  \Kg(y-z) \Big| = 2 \eps^2 \delta^{-1} \big| F_1(r,z)  \big| \;. 
 \end{align} 
Hence, the Burkholder-Davis-Gundy inequality implies that for every $p >0$, we have 
\begin{align*}
\Big(\E  \sup_{0 \leq r \leq t} \big| \I{1}  F (r) \big|^p \Big)^{\frac{2}{p}} \leq C(p) \Big(  \int_0^t    \sum_{z \in \Le} \eg^2 \,  F_1(r,z)^2  \, dr +  \eps^4 \delta^{-2}\sup_{\substack{0 \leq r \leq t\\ z \in \Le}} F_1(r,z)^2 \Big) \;,
\end{align*}
which is what we wanted.

Let us now assume that \eqref{e:itintBound} is established for $n-1$. In order to bound moments of the martingale $\I{n} F(t)$, we bound again the quadratic variation and the size of the jumps. For the jumps, we can see as in \eqref{e:ItInt2} that a jump of $\sigma(k)$ at location $k = \eps^{-1}z_1$ at  time $r_1$ causes a jump of $\I{n} F$ of absolute value 
\begin{align*}
 2\eps^{2}\delta^{-1} \Big| \sum_{y_1 \in \Le} \eps^2   \I{n-1} F^{(r_1,y_1)}(r_1^{-}) \,\Kg(y_1 - z_1)  \Big| \;. 
 \end{align*} 
Recalling that
$$
\sum_{y_1 \in \Le} \eps^2 F^{(r_1,y_1)}(r_2,\ldots, r_n,y_2,\ldots, y_n) \Kg(y_1-z_1) = F_1(r_1,\ldots,r_n,z_1,y_2,\ldots, y_n)
$$
and that $F \mapsto \I{n-1} F$ is linear, we can rewrite the quantity above as
$$
 2\eps^{2}\delta^{-1} \Big| \sum_{y_1 \in \Le} \eps^2   \I{n-1} F_1(r_1^{-},z_1) \Big| \;. 
$$

Hence, the corresponding error term in the bound for $\Big(\E  \sup_{0 \leq r \leq t}\big| \I{n} F (r) \big|^p \Big)^{\frac{2}{p}} $ takes the form 
\begin{equation}
	\label{e:ItInt3}
C(p)\eps^{4}\delta^{-2} \,   \bigg( \E  \sup_{\substack{0 \leq r_1 \leq t \\z_1 \in \Le  }}  \big| \I{n-1}  F_1(r_1, z_1 ) \big|^p   \bigg)^{\frac{2}{p}} \;,	
\end{equation}
which is precisely the term corresponding to $\ell=1$ in \eqref{e:itintError}. 

For the quadratic variation of $\I{n}F(t)$, we get as above in \eqref{e:ItInt1} that
\begin{align}
 \big\langle &\I{n} F \big\rangle_t  \notag\\
&=  \int_{r_1=0}^t  \sum_{\substack{y_1 \in \Le\\ \bar{y}_1 \in \Le}} \eg^4 \,  \I{n-1} F^{(r_1,y_1)} (r_1) \,   \I{n-1} F^{(r_1,\bar{y}_1)} (r_1) \; d \langle \Mg(\cdot, y_1) , \Mg( \cdot, \bar{y}_1) \rangle_{r_1} \, \notag\\
& \leq 4 \co^2 \int_{0}^t    \sum_{z_1 \in \Le} \eg^2 \,  \Big( \sum_{y_1 \in \Le} \eps^2  \I{n-1} F^{(r_1,y_1)} (r_1) \, \Kg(y_1-z_1)   \Big)^2  \; dr_1 \notag \\
& = 4 \co^2 \int_{0}^t    \sum_{z_1 \in \Le} \eg^2 \,  \big( \I{n-1}F_1(r_1,z_1)\big)^2  \, dr_1 \notag\;.
\end{align}
In the first equality above, we have used the fact that $F^{(r_1,y_1)} (r_1^-) = F^{(r_1,y_1)} (r_1)$ for Lebesgue almost every $r_1$. Then Minkowski's inequality (for the exponent $\frac{p}{2} \geq 1$) implies that 
\begin{align}\label{e:ItInt4}
\Big( \E  \big\langle &\I{n} F \big\rangle_t^{\frac{p}{2}}  \Big)^{\frac{2}{p}} \leq  4 \co^2 \int_{0}^t    \sum_{z_1 \in \Le} \eg^2 \, \Big( \E   \big| \I{n-1}F_1(r_1,z_1)\big|^p \Big)^{\frac{2}{p}}  \, dr_1 \;. 
\end{align}
 But  the induction hypothesis implies that for some $C = C(n,p)$, for every $r_1 \geq 0$ and $ z_1 \in \Le$, we have 
 \begin{align}
  \Big( \E   \big| \I{n-1}F_1(r_1,z_1)\big|^p \Big)^{\frac{2}{p}} \leq & \,C  \int_{r_2=0}^{r_{1}} \ldots \int_{r_n=0}^{r_{n-1}}   \sum_{\substack{z_2, \ldots,  z_n \in \Le}} \eps^{2(n-1)}  \; F_n(\rrr,\zz)^2 \;  d\rrr' \notag \\
 & + \msf{Err}(r_1,z_1) \;, \label{e:ItInt5}
 \end{align}
 where $d \rrr' = dr_n \ldots dr_2$, and  where the error term is given by
 \begin{align}
  \msf{Err}(r_1,z_1) =  C \eps^4 \delta^{-2}    \; \sum_{\ell=2}^{n} &     \int_{r_2=0}^{r_1} \ldots \int_{r_{\ell-1}=0}^{r_{\ell-2}} \sum_{z_2, \ldots , z_{\ell-1} \in \Le } \eps^{2(\ell-2)}\notag\\
&  \quad \quad  \bigg( \E \sup_{\substack{0 \leq r_\ell \leq r_{\ell-1}\\z_\ell \in \Le  }} \big|  \I{n-\ell} F_{\ell}(\rrr_\ell, \zz_\ell) \big|^p   \bigg)^{\frac{2}{p}} \; d\rrr_\ell' \;, \label{e:ItInt6}
 \end{align}
 with $d \rrr_\ell' = dr_\ell \ldots dr_2$. We then obtain the desired estimate by plugging \eqref{e:ItInt5} and \eqref{e:ItInt6}  into \eqref{e:ItInt4}.
\end{proof}
%%%%%%%%%%%%%%%%%%%
%
%
%
%

With Lemma~\ref{le:ItInt} in hand, we now proceed to derive bounds on $\ZG{n}$. 
%
%%%%%%%%%%%%%%%%%
\begin{proposition}\label{prop:RGBoundOne}
%%%%%%%%%%%%%%%%%
There exists a constant $ \ga_0>0$ such that the following holds. For every $n \in \N$,  $p\geq 1$,  $\al>0$,  $T>0$,  $0 \leq \la \leq \frac12$ and $0 < \ka \leq 1$,  there exists a constant $C= C(n,p,\al,T,  \la,\ka)$ such that for every $0 \leq s \leq t \leq T$ and $0< \ga < \ga_0$,
\begin{align}
 \E \sup_{0 \leq r \leq t}\| \RG{n}(r, \cdot) \|^p_{\Cc^{-\al- 2\la}} &\leq C \,t^{\la \, p}   \qquad  +C \ga^{p(1 - \ka)}   \;, \label{e:RgRegularity1}\\
\E \sup_{0 \leq r \leq t} \| \RG{n}(r, \cdot) -R_{\ga,s}^{\colon n \colon}(r \wedge s, \cdot) \|^p_{\Cc^{-\al-  2\la}} &\leq C \, |t-s|^{\la \, p}  +C \ga^{p(1 - \ka)} \; ,  \label{e:RgRegularity2} \\
\E \sup_{0 \leq r \leq t} \| \RG{n}(r, \cdot) - \RG{n}(r \wedge s,\cdot) \|^p_{\Cc^{-\al-  2\la}} &\leq C  \; |t-s|^{\la \, p}   +C \ga^{p(1 - \ka)} \; .  \label{e:RgRegularity3} 
 \end{align}
\end{proposition}
%%%%%%%%%%%%%%%%%%%%
\begin{remark}
In particular, the bounds \eqref{e:RgRegularity1} -- \eqref{e:RgRegularity3} imply that (under the same conditions on $p,\al,\la,\ka$) we have
\begin{align*}
 \E \| \ZG{n}(t, \cdot) \|^p_{\Cc^{-\al- 2\la}} &\leq C \,t^{\la \, p}   \qquad  +C \ga^{p(1 - \ka)}   \;, \\
\E \| \ZG{n}(t, \cdot) -\ZG{n}(s, \cdot) \|^p_{\Cc^{-\al-  2\la}} &\leq C \, |t-s|^{\la \, p}  +C \ga^{p(1 - \ka)} \; , \\
\E \| \ZG{n}(t, \cdot) - \RG{n}(s,\cdot) \|^p_{\Cc^{-\al-  2\la}} &\leq C  \; |t-s|^{\la \, p}   +C \ga^{p(1 - \ka)} \; .  
 \end{align*}
 These weaker bounds are the key ingredient for the proof of tightness  in Proposition~\ref{prop:tight} and for the proof of convergence in law in Theorem~\ref{t:converg-lin-Wickbis} below.
\end{remark}
%
%%%%%%%%%%%%%%%%%%%%
\begin{proof}
Recalling that $R_{\ga,0}^{\colon n \colon}(0,\cdot) =0$, we see that the bound \eqref{e:RgRegularity1} is contained in \eqref{e:RgRegularity2}, so that it suffices to show \eqref{e:RgRegularity2} and \eqref{e:RgRegularity3}. Furthermore, note that by the monotonicity in $p$ of stochastic $L^p$ norms, it is sufficient to prove these bounds for $p$ large enough. 

For any smooth function $f \colon \T^2 \to \R$ and any $n$, we write 
\begin{equation*}
\RG{n}(s,f) := \int_{\T^2} \RG{n}(s,x) \, f(x) \, dx \, , 
\end{equation*}
and similarly, $\ZG{n}(t,f) := \RG{n}(t,f)$. Note that in general, neither $f$ nor $\RG{n}$ (defined for all $x \in \T^2$ as in \eqref{e:ZnB}) are trigonometric polynomials of degree $\leq N$, so that this integral does not coincide exactly with its Riemann sum approximation on~$\Le$. 

 In order to obtain  \eqref{e:RgRegularity2} and \eqref{e:RgRegularity3},  we derive bounds on 
\begin{align}
\E \sup_{0 \leq r \leq t} | \RG{n}(r, f) - R_{\ga,s}^{\colon n \colon}(r \wedge s,f) |^p    \quad \text{and} \quad
\E \sup_{0 \leq r \leq t}| \RG{n}(r,f) - \RG{n}(r \wedge s,f) |^p  \notag
 \end{align}
for an arbitrary smooth test function $f \colon \T^2 \to \R$ and for an arbitrary $p \geq 2$. Later, we will specialise to $f(x) = \eta_k(u-x)$ (defined in \eqref{e:DEFek})  for some $u \in \T^2$ and $k \geq -1$ to apply Proposition~\ref{prop:Kolmogorov}. 

A simple recursion based on \eqref{e:ZnB} shows that for $0 \leq s \leq t$, we have  $\RG{n}(s,f) = \I{n} F^t(s)$ where for $\yy= (y_1, \ldots, y_n)$, $\rrr=(r_1, \ldots, r_n)$, and $\bom= (\om_1, \ldots, \om_n)$,
\begin{align}
F^t(\yy,\rrr) &= \;\frac{n!}{4^n}   \sum_{\bom \in (\Z^2)^n} \overline{ \hat{f} (\om_1 + \ldots + \om_n)} \prod_{j=1}^n \hPg{t-r_j} (\om_j) e^{- i \pi \om_j  \cdot y_j}  \notag \\
&= n! \int_{\T^2} \, f(x)  \,  \prod_{j=1}^n  \Pg{t-r_j}(x- y_j)  \,   dx \;. \notag
\end{align}
Here, each $\Pg{t-r_j}$ is viewed as a function on all of $\T^2$. Note that the kernel $\prod_{j=1}^n  \Pg{t-r_j}(x- y_j) $ is a trigonometric polynomial of degree $\leq nN$ in the $x$ variable and a trigonometric polynomial of degree $\leq N$ in each $y_j$ coordinate.

By linearity of the operator $\I{n}$, we get for any $0 \leq s \leq t$ and $0 \leq r \leq t$ that
\begin{align*}
\RG{n}(r, f) - R_{\ga,s}^{\colon n \colon}(r \wedge s,f) &= \I{n}(F^t - F^s)(r) \;, \\
\RG{n}(r, f) - \RG{n}(r \wedge s,f) &= \I{n}(F^t  \mathbf{1}_{r_1 \in [s,t]} )(r) \;.
\end{align*}
Here we use the convention to set $ \Pg{s-r_j}(x- y_j)    = 0$ for $r_j > s$. 

  Hence, by  Lemma~\ref{le:ItInt} there exists $C=C(n,p)$ such that
\begin{align}
\big( \E & \sup_{0 \leq r \leq t}| \RG{n}(r, f) - R_{\ga,s}^{\colon n \colon}(r,f) |^p \big)^{\frac{2}{p}} \notag \\ 
& \leq  \frac{C}{n!} \int_{r_1, \ldots, r_n =0}^t   \sum_{\substack{\zz \in \Le^n}} \eps^{2n}  \; 
 \big( F^t_n(\rrr,\zz)  - F^s_n(\rrr,\zz)   \big)^2 \;  d\rrr + \msf{Err} \;, \label{e:RN1}
\end{align}
and
\begin{align}
\big( \E& \sup_{0 \leq r \leq t} | \RG{n}(r,f)  - \RG{n}(r \wedge s,f) |^p \big)^{\frac{2}{p}}  \notag\\
 &\leq\frac{C}{(n-1)!}\int_{r_1=s}^t \int_{r_2, \ldots ,r_n =0}^{r_1}    \sum_{\substack{\zz \in \Le^n}} \eps^{2n}  \; 
 F^t_n(\rrr,\zz)^2 \;  d\rrr + \msf{Err}^{\prime} \; \label{e:RN2}.
\end{align}
Here $F_n^t$ denotes the convolution of $F^t$ with the kernel $\Kg$ in all $n$ spatial arguments as defined in \eqref{e:def:Fl}. In \eqref{e:RN1} and \eqref{e:RN2}, we have used the symmetry of the kernels $F_n^t$ and $F_n^s$ in their time arguments to replace the integrals over the simplices $0 \leq r_n \leq r_{n-1}  \leq  \ldots \leq r_1\leq t$ and $0 \leq r_n \leq  \ldots \leq r_1$ by integrals over $[0,t]^n$ and $[0,r_1]^{n-1}$. The precise form of the error terms $\msf{Err}$ and $ \msf{Err}^{\prime}$ is discussed below.

We start by bounding the first term on the right-hand side of \eqref{e:RN1}. Using Parseval's identity \eqref{e:Pars} in each of the $z_j$ summations, we get  for any fixed $\rrr$
\begin{align}
 \sum_{\substack{\zz \in \Le^n}} \eps^{2n} & \; 
 \big( F^t_n(\rrr,\zz)  - F^s_n(\rrr,\zz)   \big)^2 \notag \\
= \ \frac{(n!)^2}{4^n} & \sum_{\substack{\bom \in (\Z^2)^n }}  \big| \hat{f}(\om_1 + \ldots + \om_n) \big|^2 
&\Big(  \prod_{j=1}^n  \hPg{t-r_j} \hKg(\om_j)     -  \prod_{j=1}^n  \hPg{s-r_j} \hKg(\om_j)    \Big)^2   \;, \notag
\end{align}
where as above we write  $\bom = (\om_1, \ldots, \om_n)$. For fixed $\bom$ and $\rrr$ we bound
\begin{align}
\Big( & \prod_{j=1}^n  \hPg{t-r_j} \hKg(\om_j)     -  \prod_{j=1}^n  \hPg{s-r_j} \hKg(\om_j)    \Big)^2 \notag \\
& \quad \leq n \sum_{k=1}^n \Big(   \prod_{j=1}^{k-1}  \big(  \hPg{t-r_j} \hKg(\om_j) \big)^2   \; \big(  \hPg{t-r_k} \hKg(\om_k) -  \hPg{s-r_k} \hKg(\om_k)    \big)^2 \notag\\
& \qquad \quad \times  \, \prod_{j=k+1}^{n}  \big( \hPg{s-r_j} \hKg(\om_j) \big)^2    \Big)\;.    \label{e:RN3}
\end{align} 
We performs the $\rrr$-integrations for fixed value of $k$ and $\bom$ for each term on the right-hand side of \eqref{e:RN3} separately, recalling that for $\om \in \{-N, \ldots, N\}^2$ and for any $t \geq 0$ we have $
\hPg{t}(\om) = \exp\big(- t \gamma^{-2} (1- \hKg(\om)) \big)$
according to \eqref{e:Fourier-semi}. For the integrals over $r_j$ for $j \neq k$ we use the elementary estimate
\begin{equation}
\label{e:elementary}
\int_0^t   e^{-(t-r) 2\ell} \, dr  \leq e \int_0^{\infty} e^{-r(2 \ell+\frac{1}{t})} \, dr = \frac{e}{\tfrac{1}{t}+ 2\ell}\;,
\end{equation}
for $\ell =\gamma^{-2} \big( 1 - \hKg(\om_j )\big) \geq 0$. We split the integral over $r_k$ into an integral over $[0,s]$ and an integral over $[s,t]$. Then for the same choice of $\ell$ we use the bounds 
\begin{align*}
\int_0^s  \big(   e^{ - (s-r) \ell } - e^{ -(t-r) \ell } \big)^2  \, dr &=  \Big( 1 -  e^{ -(t-s) \ell} \Big)^2 \int_0^s   e^{ - (s-r) 2\ell }    \, dr \\
&  \leq  \ell \int_0^{t-s} e^{-r \ell} dr \; \frac{1}{2\ell} \;\\
& \leq \frac{e}{\tfrac{1}{t-s}+ 2\ell}\; ,
\end{align*}
as well as
\begin{align}
\int_s^t   e^{ - 2(t-r) \ell }   \, dr &\leq \frac{e}{\tfrac{1}{t-s}+ 2\ell}\;. \label{e:elementary2} 
\end{align}
In this way we obtain for every $k \in \{1, \ldots, n\}$
\begin{align*}
\int_{[0,t]^n} &  \prod_{j=1}^{k-1} \big(   \hPg{t-r_j} \hKg(\om_j) \big)^2   \; \big(  \hPg{t-r_k} \hKg(\om_k) -  \hPg{s-r_k} \hKg(\om_k)    \big)^2  \\
&\qquad \qquad \qquad \qquad\qquad \quad \quad \times \prod_{j=k+1}^{n} \big(  \hPg{s-r_j} \hKg(\om_j)\big)^2 \, d\rrr  \\
&\leq e^n \prod_{j=1}^{k-1}   \,\frac{ \big|\hKg(\om_j)\big|^2}{\frac{1}{t} + 2 \gamma^{-2} \big( 1 -\hKg(\om_j)\big)} \;  \frac{ \big|\hKg(\om_k)\big|^2}{ \frac{1}{t-s} +2 \gamma^{-2} \big( 1 -\hKg(\om_k)\big)}\\
&  \qquad \qquad \qquad \qquad \qquad \quad \times \prod_{j=k+1}^{n}   \,\frac{ \big|\hKg(\om_j)\big|^2}{\frac{1}{s} + 2 \gamma^{-2} \big( 1 - \hKg(\om_j )\big) } \,.
\end{align*}
Here we have again made use of the convention $\hPg{s-r_j}(\om_j) = 0$ for $r_j >s$. We only make this bound worse, if we replace the $\frac{1}{s}$ appearing in last line of this expression by $\frac{1}{t}$.  Plugging this back into \eqref{e:RN1}, summing over $\bom$ and using the invariance of this expression under changing the value of $k \in \{1, \ldots, n \}$, we obtain
\begin{align}
\frac{1}{n!}\int_{r_1, \ldots, r_n =0}^t&   \sum_{\substack{\zz \in \Le^n}} \eps^{2n}  \; 
 \big( F^t_n(\rrr,\zz)  - F^s_n(\rrr,\zz)   \big)^2 \;  d\rrr \notag \\
\leq e^n \,n!  \,n^2 & \frac{1}{4^n}  \sum_{\substack{\bom \in (\Z^2)^n }}  \big| \hat{f}(\om_1 + \ldots + \om_n) \big|^2 \frac{ \big|\hKg(\om_1)\big|^2}{\frac{1}{t-s} + 2 \gamma^{-2} \big( 1 -\hKg(\om_1)\big)}   \notag \\
&\quad  \times \prod_{j=2}^{n}   \,\frac{ \big|\hKg(\om_j)\big|^2}{\frac{1}{t} + 2 \gamma^{-2} \big( 1 - \hKg(\om_j )\big) }   \;.\label{e:RN3A}
\end{align}
The corresponding calculation for the integral in \eqref{e:RN2} is very similar (but slightly simpler). After passing to spatial Fourier variables  and performing the  integrations over $r_2, \ldots, r_n$ using \eqref{e:elementary}, we get
\begin{align*}
\frac{1}{(n-1)!}\int_{r_1=s}^t  &\int_{r_2, \ldots ,r_n =0}^{r_1}    \sum_{\substack{\zz \in \Le^n}} \eps^{2n}  \; 
 F^t_n(\rrr,\zz)^2 \;  d\rrr \\
 \leq &\,e^{n-1} n! \, n  \frac{1}{4^n} \sum_{\substack{\bom \in (\Z^2)^n }}  \big| \hat{f}(\om_1 + \ldots + \om_n) \big|^2 \int_{r_1 =s}^t \big( \hPg{t-r_1}  \hKg(\om_1)\big)^2 \\
& \qquad \times \prod_{j=2}^{n}   \,\frac{ \big|\hKg(\om_j)\big|^2}{\frac{1}{r_1} + 2 \gamma^{-2} \big( 1 - \hKg(\om_j )\big) } \;  dr_1 \,.
\end{align*}
As above, we only make this bound worse, if we replace the expression $\frac{1}{r_1}$ appearing in the last line by $\frac{1}{t}$. Then we can perform the $dr_1$ integral using  \eqref{e:elementary2}. In this way we get  the same upper bound (up to an inessential factor $n$) \eqref{e:RN3A} for  the integrals appearing on the right-hand side of \eqref{e:RN1} and \eqref{e:RN2}.

Now we specialise to $f(x) = \eta_k(u -x)$ for some $u \in \T^2$ and $k \geq -1$. Note that according to \eqref{e:dkConvProp}, for this choice of $f$, we have $\RG{n}(s,f) = \dk \RG{n}(s,u)$. For this $f$, we recall from \eqref{e:def:chik} that
 $\hat{f}(\om) = \chi_k(\om) \, e^{-i \pi u \cdot \om}$. In particular,  
$ \big| \hat{f}(\om)\big| \leq 1$ for all $\om \in \Z^2$ and $\hat{f}(\om) = 0$ for $|\om|  \notin I_k$, where $I_k = 2^k\big[3/4, 8/3 \big] $
for $k \geq 0$ and $I_{-1} = [0, 4/3]$. Summarising, we can conclude that for every $n \geq 1$ and $p \geq 2$, there exists $C=C(n,p)$ such that for all $0 \leq s \leq t$, $k \geq -1$ and $u \in \T^2$ ,
\begin{align} 
\Big( &\E   \sup_{ 0 \leq r \leq t}\big| \dk  \RG{n}(r,u) -\dk R_{\ga,s}^{\colon n \colon}(r \wedge s, u) \big| ^p \Big)^{\frac{2}{p}} \notag \\
&\leq C \sum_{\sum \bom  \in I_k} \,   \frac{ \big|\hKg(\om_1)\big|^2}{\frac{1}{t-s} + 2 \gamma^{-2} \big( 1 -\hKg(\om_1)\big)}  \prod_{j=2}^n \,\frac{ \big|\hKg(\om_j)\big|^2}{\frac{1}{t} + 2 \gamma^{-2} \big( 1 -\hKg(\om_j)\big)}  + \msf{Err}\;, \label{e:RN4}
\end{align}
where, for $\bom = (\om_1, \ldots, \om_n) \in (\Z^2)^n$, we write $\sum \bom = \sum_{j=1}^n \om_j$. In the same way, we have
\begin{align} 
&\Big( \E \sup_{0 \leq r \leq t} \big| \dk  \RG{n}(r,u) -\dk \RG{n}(r \wedge s, u) \big| ^p \Big)^{\frac{2}{p}}  \notag\\
&\leq C\! \sum_{\sum \bom  \in I_k} \,   \frac{ \big|\hKg(\om_1)\big|^2}{\frac{1}{t-s} + 2 \gamma^{-2} \big( 1 -\hKg(\om_1)\big)}  \prod_{j=2}^n \,\frac{ \big|\hKg(\om_j)\big|^2}{\frac{1}{t} + 2 \gamma^{-2} \big( 1 -\hKg(\om_j)\big)}  + \msf{Err}^{\prime}\;. \label{e:RN5}
\end{align}
We defer the analysis of the sum appearing on the right-hand side of both \eqref{e:RN4} and \eqref{e:RN5} to Lemma~\ref{le:FourCalc} below, and proceed to analyse the error terms $\msf{Err}$ and $\msf{Err}^{\prime}$, going back to the setting of an arbitrary smooth $f : \T^2 \to \R$ for the moment. The term $\msf{Err}$ on the right-hand side of \eqref{e:RN1} is given by
\begin{align}
\msf{Err}  =
&C \eps^4 \delta^{-2}\sum_{\ell=1}^{n}   \int_{r_1=0}^t \ldots \int_{r_{\ell-1}=0}^{r_{\ell-2}} \sum_{z_1, \ldots , z_{\ell-1} \in \Le } \eps^{2(\ell-1)}\notag\\
& \qquad \qquad \qquad   \bigg( \E \sup_{\substack{0 \leq r_\ell \leq r_{\ell-1}\\z_\ell \in \Le  }}  \big|  \I{n-\ell} (F_{\ell}^t - F_{\ell}^s) (\rrr_\ell, \zz_\ell)  \big|^p   \bigg)^{\frac{2}{p}} \; d\rrr_{\ell-1}    \;, \label{e:RN6}
\end{align}
where as above $ \rrr_\ell = (r_1, \ldots r_\ell)$, $\zz_\ell  =( z_1, \ldots , z_\ell)$, and $ d\rrr_{\ell-1} = dr_{\ell-1}\cdots  d{r_2}\, d{r_1}$.
Here, for any $1 \leq \ell \leq n$ and for fixed $\rrr_\ell, \zz_\ell$, we have
\begin{align}
 \big| \I{n-\ell}& (F_{\ell}^t - F_{\ell}^s) (\rrr_\ell, \zz_\ell) \big|^p \notag\\
  =  & \Big| \int_{\T^2} f(x) \,\Big(  \RG{n-\ell}(r_{\ell }, x)    \prod_{j=1}^{\ell} \Pg{t-r_j} \star \Kg(x - z_j)\notag \\
&\qquad  \qquad\qquad\qquad  \qquad-  R_{\ga,s}^{\colon n-\ell \colon}(r_\ell,x)  \;   \prod_{j=1}^{\ell} \Pg{s-r_j} \star \Kg(x - z_j)   \Big)\, dx \; \Big|^p \; \notag \\
& \leq p \, \big\|    \RG{n-\ell}(r_{\ell }, \cdot) \big\|_{L^\infty}^p \Big(  \int_{\T^2} \big| f(x)   \big| \;      \prod_{j=1}^{\ell} \big| \Pg{t-r_j} \star \Kg(x - z_j)  \big| \, dx \Big)^p \notag\\
&\quad +    p  \big\|    R_{\gamma,s}^{\colon n- \ell \colon}(r_{\ell }, \cdot) \big\|_{L^\infty}^p \Big(  \int_{\T^2} \big| f(x) \big|  \;    \prod_{j=1}^{\ell}  \big| \Pg{s-r_j} \star \Kg(x - z_j)  \big| \, dx \Big)^p \label{e:RN7}\;,
\end{align}
where we use the convention $\RG{0} =1$ and $R_{\ga,s}^{\colon n-\ell \colon}(r_\ell,x)  = 0$ for $r_\ell >s$. By Lemma~\ref{l:unif-Fourier-cut} and the definition of $\RG{n-\ell}$ as a trigonometric polynomial of degree $\leq (n-\ell) N \leq nN$, for every $\ka^{\prime} >0$ there exists a constant $C = C(n,\ka^{\prime})$ such that
\begin{align*}
\big\|    \RG{n-\ell}(r_{\ell }, \cdot) \big\|_{L^\infty}  \leq & C \gamma^{- 2 \ka^{\prime}  } \;   \big\|    \RG{n-\ell}(r_{\ell }, \cdot) \big\|_{\Cc^{-\ka^{\prime}}} \;. 
\end{align*}
Hence, plugging \eqref{e:RN7} back into \eqref{e:RN6} and applying Minkowski's inequality, we get
\begin{align}
\msf{Err}  \leq&  \;C(n,p, \kappa^{\prime})  \, \eps^4  \,\delta^{-2} \, \gamma^{- 4 \kappa^{\prime}  } \notag\\
& \sum_{\ell =1}^{n} \bigg[   \Big( \E \sup_{0 \leq r \leq t} \big\|   \RG{n-\ell}(r,\cdot) \big\|_{\Cc^{-\kappa^{\prime}}}^p    + \E \sup_{0 \leq r \leq s}  \big\|   R_{\ga,s}^{\colon n-\ell \colon}(r,\cdot) \big\|_{\Cc^{-\kappa^{\prime}}}^p \Big)^{\frac{2}{p}} \notag
 \\
& \qquad \times \int_{r_1=0}^t \ldots \int_{r_{\ell-1}=0}^{r_{\ell-2}} \sum_{z_1, \ldots , z_{\ell-1} \in \Le } \eps^{2(\ell-1)} \notag \\
& \qquad \quad \sup_{\substack{0 \leq r_\ell \leq r_{\ell-1}\\z_\ell \in \Le  }}\Big(  \int_{\T^2} \big| f(x)  \big| \;      \prod_{j=1}^{\ell}  \big| \Pg{t-r_j} \star \Kg(x - z_j)  \big| \, dx \Big)^2   \; d\rrr_{\ell-1}   \bigg] \;. \label{e:RN8}
\end{align}
We turn to bound $\msf{Err}^{\prime}$. In the same way,
\begin{align}
\msf{Err}^{\prime}  &=  C \eps^4 \delta^{-2} \bigg( \E  \sup_{\substack{s \leq r_n \leq t \\z_1 \in \Le  }}  \big|  \I{n-1} F_{\ell}^t  (r_1, z_1)  \big|^p   \bigg)^{\frac{2}{p}} \notag\\
&\quad +C \eps^4 \delta^{-2}\sum_{\ell=2}^{n}   \int_{r_1=s}^t \ldots \int_{r_{\ell-1}=0}^{r_{\ell-2}} \sum_{z_1, \ldots , z_{\ell-1} \in \Le } \eps^{2(\ell-1)}\notag\\
& \qquad \qquad \qquad   \qquad  \bigg( \E \sup_{\substack{0 \leq r_\ell \leq r_{\ell-1}\\z_\ell \in \Le  }}  \big|  \I{n-\ell} F_{\ell}^t  (\rrr_\ell, \zz_\ell)  \big|^p   \bigg)^{\frac{2}{p}} \; d\rrr_{\ell-1}    \; \notag \\
 & \leq \;C(n,p, \kappa^{\prime})  \, \eps^4  \,\delta^{-2} \, \gamma^{- 4 \kappa^{\prime}  } 
 \sum_{\ell =1}^{n} \bigg[   \Big( \E \sup_{0 \leq r \leq t} \big\|   \RG{n-\ell}(r,\cdot) \big\|_{\Cc^{-\kappa^{\prime}}}^p     \Big)^{\frac{2}{p}} \notag
 \\
& \quad \times \int_{r_1=0}^t \ldots \int_{r_{\ell-1}=0}^{r_{\ell-2}} \sum_{z_1, \ldots , z_{\ell-1} \in \Le } \eps^{2(\ell-1)}  \notag  \\
& \qquad \quad \sup_{\substack{0 \leq r_\ell \leq r_{\ell-1}\\z_\ell \in \Le  }}\Big(  \int_{\T^2} \big| f(x)  \big| \;      \prod_{j=1}^{\ell} \big| \Pg{t-r_j} \star \Kg(x - z_j)  \big| \, dx \Big)^2   \; d\rrr_{\ell-1}   \bigg]  \label{e:RN9}\;.
\end{align}
The integral appearing both on the right-hand side of \eqref{e:RN8} and \eqref{e:RN9} specialised to the case $f(x) = \eta_k(u-x)$ for some $k \geq -1$ and $u \in \T^2$ is bounded in Lemma~\ref{le:ErrIntBound}.

Actually, in the case $n=1$, we obtain a slightly better bound, because the only stochastic process $\RG{n-\ell}$ appearing on the right-hand side of \eqref{e:RN8} and \eqref{e:RN9} is $\RG{0}=1$. Hence, the embedding from $\Cc^{-\ka^{\prime}} \to L^\infty$ is unnecessary, and we do not need to introduce the factor $\gamma^{-4 \ka^{\prime}}$.

Finally, summarising our calculations \eqref{e:RN4} and \eqref{e:RN8} as well as the bounds derived in Lemmas~\ref{le:FourCalc} and~\ref{le:ErrIntBound}, we can conclude that for every $n \geq 1$,  $p \geq 2$,  $T \geq 0$,  $\lambda \in [0,1]$, and  $\al > 0$, there exists a constant $C=C(n,p,T,\lambda,\al)$ such that for all $0 < \ga < \ga_0$,   $0 \leq s \leq t \leq T$,  $u \in \T^2$, and every $k \geq -1$ ,
\begin{align}
\Big( \E & \sup_{0 \leq r \leq t}\big| \dk  \RG{n}(r,u) -\dk R_{\ga,s}^{\colon n \colon}(r \wedge s, u) \big| ^p \Big)^{\frac{2}{p}}  
 \leq  C |t-s|^{\lambda}  \, 2^{2k \lambda} (k+2)^n \notag\\
& + C \ga^{2 -4 \al}  \big( \log(\ga^{-1})\big)^{2+6(n-1)}  \max_{\substack{\ell=0, \ldots, n-1\\\tau = s,t}} \bigg(\E \sup_{0 \leq r \leq \tau} \big\| R_{\ga,\tau}^{\colon \ell \colon}(r,\cdot)  \|_{\Ca}^p \bigg)^{\frac{2}{p}} \;, \label{e:RN10}
\end{align}
where as above we use the convention $\RG{0} =1$. We can bound  
\begin{equation*}
\Big( \E  \sup_{0 \leq r \leq t} \big| \dk  \RG{n}(r,u) -\dk \RG{n}(r \wedge s, u) \big| ^p \Big)^{\frac{2}{p}}  
\end{equation*}
by exactly the same quantity, with the only difference that the $\max$ in the second line only needs to be taken with respect to $\tau =t$. The desired bounds \eqref{e:RgRegularity2} and \eqref{e:RgRegularity3} now follow easily by induction, as we now explain.  The arguments are identical for both bounds, so we restrict ourselves to \eqref{e:RgRegularity2}.

For $n=1$, the bound \eqref{e:RN10} reduces to 
\begin{align}
\Big( \E  \sup_{0 \leq r \leq t} &\big| \dk  \Rg(r,u) -\dk r_{\ga,s}(r \wedge s, u) \big| ^p \Big)^{\frac{2}{p}} \notag  \\
& \leq  C |t-s|^{\lambda}  \, 2^{2k \lambda} (k+2)    + C \ga^{2 }  \big( \log(\ga^{-1})\big)^{2} \;, \notag
\end{align}
so that \eqref{e:RgRegularity2} for $n=1$ and $p$ large enough (choosing $p > \frac{2}{\al}$ suffices) follows from Proposition~\ref{prop:Kolmogorov}. To pass from $n-1$ to $n$, we 
observe that for every fixed $\ka^{\prime}>0$, the inductive hypothesis implies that 
\begin{align}
\label{e:supR}
\max_{\substack{\ell=0, \ldots, n-1\\\tau = s,t}} \bigg(\E \sup_{0 \leq r \leq \tau} \big\| R_{\ga,\tau}^{\colon \ell \colon}(r,\cdot)  \|_{\Cc^{-\ka^{\prime}}}^p \bigg)^{\frac{2}{p}} 
\end{align}
is uniformly bounded for $t \leq T$. Hence \eqref{e:RN10} turns into 
\begin{align}
\Big( \E \sup_{0 \leq r \leq t} & \big| \dk  \RG{n}(r,u) -\dk R_{\ga,s}^{\colon n \colon}(r \wedge s, u) \big| ^p \Big)^{\frac{2}{p}}  \notag\\
 &\leq  C |t-s|^{\lambda}  \, 2^{2k \lambda} (k+2)^{n} 
 + C \ga^{2 -4 \ka^{\prime}}  \big( \log(\ga^{-1})\big)^{2+6(n-1)}   \;, \notag
\end{align}
for $C=C(n,p,T,\lambda,\ka^{\prime})$. By choosing $\ka^{\prime} = \frac{\ka}{4}$ and applying Proposition~\ref{prop:Kolmogorov} for $p$ large enough, \eqref{e:RgRegularity2} follows for arbitrary $n \in \N$.
%
%%%%%%%%%%%%%%%%%
\end{proof}
%%%%%%%%%%%%%%%%%

%
%
%%%%%%%%%%%%%%%%%
\begin{lemma}\label{le:FourCalc}
%%%%%%%%%%%%%%%%%
Let $n \in \N$ and $\underline{c} < \bar{c}$ be fixed. Let $0 < \ga < \ga_0$, where $\ga_0$ is the constant appearing in Lemma \ref{le:Kg}. For any $k \geq 0$ let $I_k = 2^k \big[\underline{c}, \overline{c} \big] $,  and let $I_{-1} = [0, 2^{-1}\overline{c}]$. Then for every $T>0$ and  $\la \in [0,1]$, there exists a constant $C= C(n,\underline{c}, \overline{c},\la,T)$ such that for all $k \geq -1$, and $0 \leq s < t\leq T$, we have
\begin{align}
\sum_{\sum \bom  \in I_k} &\,\;   \frac{ \big|\hKg(\om_1)\big|^2}{\frac{1}{t-s} + 2 \gamma^{-2} \big( 1 -\hKg(\om_1)\big)} \:\; \prod_{j=2}^n \,\frac{ \big|\hKg(\om_j)\big|^2}{\frac{1}{t} + 2 \gamma^{-2} \big( 1 -\hKg(\om_j)\big)} \notag\\
&\leq C |t-s|^\la   \;2^{2k\la} (k+2)^{n-1}   \;,   \label{e:FourierCalculation}
\end{align}
where for any $\bom = (\om_1, \ldots, \om_n) \in (\{-N, \ldots, N\}^2)^n$, we use the short-hand notation  $\sum \bom =  \sum_{j=1}^n \om_j$.
\end{lemma}
%%%%%%%%%%%%%%%%%%%
%
%
%
%%%%%%%%%%%%%%%%%%%%%%
\begin{proof}
We assume that $0 < \ga < \ga_0$ where $\ga_0$ is the constant appearing in Lemma~\ref{le:Kg}. We only need to consider those  $\bom$ with $ \om_j \in \{-N, \ldots, N\}^2$ for all $j$ because all other summands vanish.  

For $|\om_j| \le \ga^{-1}$, we can use \eqref{e:K1} and \eqref{e:K2} to bound 
\begin{align*}
 \frac{ \big|\hKg(\om_j)\big|^2}{\frac1t + 2 \gamma^{-2} \big( 1 -\hKg(\om_j)\big)}  \leq \,\frac{1}{\frac1t + \frac{2}{C_1} |\om_j|^2}\;  \leq C(T)  \Big( \frac{1}{1+ |\om_j|^2}  \wedge t \Big)\;.
\end{align*}
For $|\om_j| > \ga^{-1}$, we get essentially the same bound using \eqref{e:K1},\eqref{e:K3} and \eqref{e:K2}:
\begin{align}\label{e:Four_Calc1}
 \frac{ \big|\hKg(\om_j)\big|^2}{\frac1t + 2 \gamma^{-2} \big( 1 -\hKg(\om_j)\big)}  \leq \,\frac{C}{|\ga \om_j|^2} \,\frac{1}{\frac1t + \frac{2}{C_1}\ga^{-2}} \;  \leq C  \Big( \frac{1}{1+ |\om_j|^2}  \wedge t \Big)\;.
\end{align}

To bound the sum over these terms we claim that for any $\om\in \Z^2$, $0 \leq \la \leq 1$ and $r>0$, we have the following bound
\begin{align}\label{e:FourierClaim}
G^{(n)}(\om)& := \sum_{ \substack{ \bom \in (\Z^2)^n\\ \sum \bom = \om}} \, \bigg( \frac{1}{1 +  |\om_1|^2} \wedge r\bigg)      \prod_{j=2}^n \,\frac{1}{1 +  |\om_j|^2}\notag\\
& \leq C(n,\la) \;   r^\la  \, \frac{1}{1 + |\om|^{2(1-\la)}}  \log(1 +|\om|)^{n-1}.
\end{align}
We show \eqref{e:FourierClaim} by induction. For  $n=1$, it follows from an easy interpolation. For $n \geq 2$, we observe that for any $\om \in \Z^2$,
\begin{equation*}
G^{(n)}(\om) = \sum_{\om_1 + \om_2 = \om} G^{(n-1)}(\om_1) \, \frac{1}{1+ |\om_2|^2} \;. 
\end{equation*}
To bound this sum, we split its index set $\Aa(\om) :=\{ (\om_1, \, \om_2 ) \in (\Z^2)^2 \colon\, \om_1 + \om_2 = \om \}$  into the following three sets
\begin{align*}
\Aa_{1}(\om)&:= \big\{ (\om_1, \, \om_2 ) \in \Aa\colon  \quad   |\om_2| \leq \tfrac{1}{2}|\om| \big\}\;,\\
\Aa_{2}(\om)&:= \big\{ (\om_1, \, \om_2 ) \in \Aa\colon  \quad  |\om_2| > \tfrac{1}{2}|\om| \,,\quad \text{and} \quad  |\om_1| \leq 3|\om| \big\}\;,\\
\Aa_{3}(\om)&:= \big\{ (\om_1, \, \om_2 ) \in \Aa \colon  \quad     |\om_1| > 3|\om| \big\}\;.
\end{align*}
On  $\Aa_1(\om)$, we have by the triangle inequality that  $|\om_1| \geq \tfrac12|\om|$, so that we can bound
\begin{align*}
\sum_{ \Aa_1(\om)} &G^{(n-1)}(\om_1) \, \frac{1}{1+ |\om_2|^2} \\
&\leq  C\frac{r^\la}{1+ \big| \tfrac12 \om  \big|^{2(1-\la)}} \log\big(1 + \tfrac12 |\om| \big)^{n-2} \sum_{|\om_2| \leq \tfrac12 |\om|} \frac{1}{1 + |\om_2|^2}\\
&\leq  C \,  \frac{r^\la}{1+ |\om|^{2(1-\la)}} \log\big( 1+ |\om| \big)^{n-1}.
\end{align*}
For the sum over $\Aa_2(\om)$ we write
\begin{align*}
\sum_{ \Aa_2(\om)} G^{(n-1)}(\om_1) \, \frac{1}{1+ |\om_2|^2} &\leq C \frac{1}{1+ \tfrac14|\om|^2}  \sum_{|\om_1| \leq 3|\om|} \frac{r^{\la}}{1+ |\om_1|^{2(1-\la)}} \log\big(1+ |\om_1| \big)^{n-2}\\
&\leq  C \,  \frac{r^\la}{1+ |\om|^{2(1-\la)}} \log\big(1+  |\om| \big)^{n-1}.
\end{align*}
Actually, the extra logarithm on the right-hand side of this bound is only necessary in the case $\la =0$, but we do not optimise this. 
For the sum over $\Aa_3(\om)$, we observe that on $\Aa_3$ we have $|\om_2| \geq \tfrac12 |\om_1| $ to write
\begin{align*}
\sum_{\Aa_3(\om)} G^{(n-1)}(\om_1) \, \frac{1}{1+ |\om_2|^2} &\leq C\sum_{|\om_1| \geq 3|\om|} \frac{r^\la}{1+ |\om_1|^{2(1-\la)}} \log\big( 1+|\om_1| \big)^{n-2} \frac{1}{1 + \tfrac14 |\om_1|^2}\\
&\leq  C \,  \frac{r^{\la}}{1+ |\om|^{2(1-\la)}} \log\big(1+  |\om| \big)^{n-1}\;,
\end{align*}
which finishes the proof of \eqref{e:FourierClaim}. Summing $G^{(n)}(\om)$ over all $\om$ with $|\om| \in I_k$ establishes \eqref{e:FourierCalculation}.
\end{proof}
%
%
%%%%%%%%%%%%%%%%%%%%
\begin{lemma}\label{le:ErrIntBound}
%%%%%%%%%%%%%%%%%%%%
There exists $\ga_0>0$ such for any $T >0$,  there exists a constant $C= C(T)$ such that for any $\ell \in \N$, $0 \leq t \leq T$,  $k \geq -1$, $0 < \ga < \ga_0$ and $u \in \T^2$, we have 
 \begin{align}
&\int_{r_1=0}^t \ldots \int_{r_{\ell-1}=0}^{r_{\ell-2}} \sum_{z_1, \ldots , z_{\ell-1} \in \Le } \eps^{2(\ell-1)}  \notag\\
& \qquad \qquad    \sup_{\substack{0 \leq r_\ell \leq r_{\ell-1}\\z_\ell \in \Le  }}    \bigg(   \int_{\T^2} \big| \ek(u-x) \big|      \prod_{j=1}^{\ell}  \big| \Pg{t-r_j} \star \Kg(x - z_j)  \big| \, dx \bigg)^2 \; d\rrr_{\ell-1} \notag\\
& \qquad  \leq C \gamma^{-4 }  \big(\log(\ga^{-1})\big)^{2+6 (\ell-1)} \;.  \label{e:ErrorIntegral}
\end{align}
\end{lemma}
%%%%%%%%%%%%%%%%%%%%%
%
%
\begin{proof}
Throughout the calculation we make heavy use of the  pointwise bounds on $ \Pg{t} \star  \Kg $ derived in Lemma~\ref{le:Pgt}. From now on we choose $\ga_0$ as the constant appearing in this lemma, and we assume that $0 < \ga < \ga_0$.
Furthermore, for notational simplicity we assume that $u=0$, but the argument and the choice of constants are independent of this.

We start by fixing $r_1 > r_2> \ldots r_{\ell-1}>0$ and $z_1, \ldots, z_{\ell -1} \in \Le$ and bounding 
\begin{align}
&\sup_{\substack{0 \leq r_\ell \leq r_{\ell-1}\\z_\ell \in \Le  }}    \bigg(   \int_{\T^2} \big| \eta_k(x) \big|      \prod_{j=1}^{\ell}  \big|\Pg{t-r_j} \star \Kg(x - z_j) \big| \, dx \bigg)^2 \; \notag\\
&\leq \sup_{0 \leq r \leq T }  \big\| \Pg{r} \star \Kg \big\|_{L^\infty(\T^2)}^2  \bigg(   \int_{\T^2} \big| \eta_k(x) \big|      \prod_{j=1}^{\ell-1} \big| \Pg{t-r_j} \star \Kg(x - z_j) \big| \, dx \bigg)^2  \;.\notag
\end{align}
Lemma~\ref{le:Pgt} implies that  for $0 \leq r_{\ell} \leq T$ we have
\begin{equation}\notag
\big\| \Pg{r} \star \Kg \big\|_{L^\infty(\T^2)}^2  \leq  C(T)  \, \gamma^{-4}   \; \big(\log(\ga^{-1})\big)^2 \;.
\end{equation}
Hence we can write
\begin{align}
&\int_{r_1=0}^t \ldots \int_{r_{\ell-1}=0}^{r_{\ell-2}} \sum_{z_1, \ldots , z_{\ell-1} \in \Le } \eps^{2(\ell-1)} \notag \\
& \qquad   \sup_{\substack{0 \leq r_\ell \leq r_{\ell-1}\\z_\ell \in \Le  }}    \bigg(   \int_{\T^2} \big| \ek(x) \big|      \prod_{j=1}^{\ell}  \big| \Pg{t-r_j} \star \Kg(x - z_j)  \big| \, dx \bigg)^2 \; d\rrr_{\ell-1} \notag \\
&\leq C \gamma^{-4}   \; \big(\log(\ga^{-1})\big)^2   \int_{\T^2}  \int_{\T^2} \big| \ek(x_1) \big|  \big| \ek(x_2) \big| \, \Kk(x_1, x_2)  \;  dx_1 \, dx_2 \;,  \label{e:ErrInt1}
\end{align}
where
\begin{align}
&\Kk(x_1,x_2) \notag\\
&=\int_{r_1=0}^t \ldots \int_{r_{\ell-1}=0}^{r_{\ell-2}} \sum_{z_1, \ldots , z_{\ell-1} \in \Le } \eps^{2(\ell-1)}\notag \\
& \qquad \qquad  \prod_{j=1}^{\ell-1}  \big| \Pg{t-r_j} \star \Kg(x_1 - z_j)  \big| \big| \Pg{t-r_j} \star \Kg( x_2 -z_j)  \big| \,d\rrr_{\ell-1} \notag  \\
&=  \frac{1}{(\ell -1)!}  \bigg(   \int_{0}^t  \sum_{z \in \Le } \eps^{2}  \,  \big| \Pg{r} \star \Kg(x_1 - z)  \big| \,\big| \Pg{r} \star \Kg(x_2-z)  \big| \, dr \bigg)^{\ell-1}\;.\notag
\end{align}
Here, in the second line we have used the symmetry of the integrand in the time variables to replace the integral over the simplex $r_{\ell-1} \leq r_{\ell -1} \leq \ldots r_1 \leq t$ by an integral over $[0,t]^{\ell-1}$. We claim that (up to a power of $\log(\ga^{-1})$) the convolution $ \sum_{z \in \Le } \eps^{2}  \,  \big| \Pg{r} \star \Kg(x_1 - z)  \big|\, \big| \Pg{r} \star \Kg(x_2- z)  \big| $ satisfies the same bounds \eqref{e:P0} and \eqref{e:P1} as $\Pg{r} \star \Kg$. Indeed, we get for all $x_1,x_2 \in \T$ that
\begin{align*}
 \sum_{z \in \Le } \eps^{2} & \,  \big| \Pg{r} \star \Kg(x_1 - z)  \big|\, \big| \Pg{r} \star \Kg(x_2- z)  \big| \\
 & \leq  \big\| \Pg{r} \star \Kg (x_1 - \cdot)   \big\|_{L^\infty(\Le)} \,  \big\| \Pg{r} \star \Kg(x_2 - \cdot)  \big\|_{L^1(\Le)} 
 \end{align*}
with the obvious conventions $ \big\| f \big\|_{L^1(\Le)} := \sum_{z \in\Le} \eps^2 \, |f(z)|$ and $ \big\| f \big\|_{L^\infty(\Le)} := \sup_{z \in \Le} |f(z)|$ for any function $f \colon \Le \to \R$. According to \eqref{e:P0}, we have uniformly over $x \in \T^2$ and $0 \leq r \leq T $ that $\big\| \Pg{r} \star \Kg (x - \cdot)   \big\|_{L^\infty(\Le)} \leq  C\big(\log(\ga^{-1})\big)^2   \big( t^{-1}    \wedge \gamma^{-2} \big) \;$. On the other hand, \eqref{e:P0} and  \eqref{e:P1} imply that
\begin{multline}
\label{e:P-L1}
 \big\| \Pg{r} \star \Kg (x - \cdot)  \big\|_{L^1(\Le)}  \\ \leq C \log(\ga^{-1})^2 \sum_{z \in \Le} \eps^2 \bigg( \frac{1}{|x-z|^2} \wedge \ga^{-1}\bigg) \leq C \log(\ga^{-1})^3\;.
\end{multline}
Combining these bounds, we get
\begin{align*}
 \sum_{z \in \Le } \eps^{2} & \,  \big| \Pg{r} \star \Kg(x_1 - z)  \big|\, \big| \Pg{r} \star \Kg(x_2- z)  \big| \leq C\big(\log(\ga^{-1})\big)^5   \big( t^{-1}    \wedge \gamma^{-2} \big)\;,
\end{align*}
for a constant $C$ that is uniform in $0 \leq r \leq T$ and $0 < \ga < \ga_0$.

Integrating this bound over $r$ for any fixed $x_1,x_2 \in \T^2$ brings
\begin{align*}
 \int_{0}^t  &\sum_{z \in \Le } \eps^{2}  \,  \big| \Pg{r} \star \Kg(x_1 - z)  \big| \,\big| \Pg{r} \star \Kg(x_2- z)  \big| \, dr \\
  &\leq C\big( \log(\ga^{-1})\big)^{5} \Big(\int_{0}^{\gamma^2}  \ga^{-2} \, dr +  \int_{\gamma^2 }^t  r^{-1}  \, dr   \Big)\\
  & \leq C\big( \log(\ga^{-1})\big)^{6}\;.
\end{align*}
Plugging this back into \eqref{e:ErrInt1} leads to
\begin{align*}
&\int_{r_1=0}^t \ldots \int_{r_{\ell-1}=0}^{r_{\ell-2}} \sum_{z_1, \ldots , z_{\ell-1} \in \Le } \eps^{2(\ell-1)}  \\
& \qquad   \sup_{\substack{0 \leq r_\ell \leq r_{\ell-1}\\z_\ell \in \Le  }}    \bigg(   \int_{\T^2} \big| \ek(x) \big|      \prod_{j=1}^{\ell}  \big| \Pg{t-r_j} \star \Kg(x - z_j)  \big| \, dx \bigg)^2 \; d\rrr_{\ell-1}\\
& \leq  C \gamma^{-4}   \; \big(\log(\ga^{-1})\big)^{2 +6 (\ell-1)} \,  \big\| \ek  \big\|_{L^1(\T^2)}^2 \;. 
\end{align*}
Recalling that according to Lemma~\ref{l:Bernstein}, $\big\| \ek(x_1) \big\|_{L^1(\T^2)}^2  \leq C$ uniformly in $k \geq -1$, we get the desired conclusion \eqref{e:ErrorIntegral}.
\end{proof}
%%%%%%%%%%%%%%%%%%%%%%%
For technical reasons, below in Lemma~\ref{le:Collection_of_Errors} we will need an additional bound on $\Zg$ that states that the very high frequencies of $\Zg$ are actually much smaller than predicted by Proposition~\ref{prop:RGBoundOne}. We define
\begin{align*}
 \ZgH(t,x)  = \sum_{2^k \geq \frac{\ga^{-2}}{10}} \delta_k \Zg(t,x)\;.
\end{align*}

\begin{lemma}\label{le:highfreq}
There exists a constant $ \ga_0>0$ such that  any $p\geq 1$,  $T>0$, and $\ka>0$, there exists a constant $C= C(p,T,\ka)$ such that for all $0 \leq s \leq t \leq T$ and $0< \ga < \ga_0$,
\begin{align}
 \E \| \ZgH(t, \cdot) \|^p_{L^\infty} \leq C \ga^{p(1 - \ka)}   \;. \label{e:RgRegularity_High}
\end{align}
\end{lemma}
\begin{proof}
In the proof of Proposition~\ref{prop:RGBoundOne}, we had already seen in \eqref{e:RN4} that for all $k \geq 0$, we have
\begin{align} 
\Big( &\E  \big| \dk  \Zg(t,u) \big| ^p \Big)^{\frac{2}{p}} \leq C \sum_{ |\om| \in 2^k [\frac{3}{4}, \frac{8}{3}] } \,   \frac{ \big|\hKg(\om)\big|^2}{\frac{1}{t} + 2 \gamma^{-2} \big( 1 -\hKg(\om)\big)}   + \msf{Err}\;.\label{e:high_freq1}
\end{align}
According to \eqref{e:RN10}, we have
\begin{align*}
\msf{Err} \leq  C \ga^{2 }  \big( \log(\ga^{-1})\big)^{2} . 
\end{align*}
In order to bound the first expression in \eqref{e:high_freq1}, recall that we are only interested in $k$ with $2^k \geq \frac{\ga^{-2}}{10}$. In particular, for $\ga$ small enough all frequencies appearing in this sum are much larger  than $\ga^{-1}$. For such $\om$, the bounds derived in Lemma~\ref{le:FourCalc} are suboptimal. Indeed, the bound \eqref{e:Four_Calc1} can be improved to
\begin{align}\label{e:high_freq2}
 \frac{ \big|\hKg(\om_j)\big|^2}{\frac1t + 2 \gamma^{-2} \big( 1 -\hKg(\om_j)\big)}  \leq \,\frac{C}{|\ga \om_j|^4} \,\frac{1}{\frac1t + \frac{2}{C_1}\ga^{-2}} \;  \leq C \ga^2  \Big( \frac{1}{1+ |\om_j|^2}  \wedge t \Big)\;,
\end{align}
where we have used the fact that  $|\om| \geq \frac{\ga^{-2}}{10} \frac{3}{4} $. Following the rest of the argument as before, we see that 
\begin{align*}
 \E \| \ZgH(t, \cdot) \|^p_{\Ca} \leq C \ga^p \,   +C \ga^{p} \log(\ga^{-1})^p \;,
\end{align*} 
and the desired bound follows from Lemma~\ref{l:unif-Fourier-cut}. 
\end{proof}

%%%%%%%%%%%%%%%%%%%%%%%
\section{Tightness for the linearised system}
%%%%%%%%%%%%%%%%%%%%%%%%
\label{sec:Tightness}
%%%%%%%%%%%%%%%%%%%%%%%%
In this section, we continue the discussion of the processes $\ZG{n}$ and $\RG{n}$ defined in \eqref{e:ZnB} and \eqref{e:DefZn}. The first main result is Proposition~\ref{p:discrete-Wick} which states that 
$\RG{n}$ can be approximately written as a Hermite polynomial applied to $\Rg$. In Proposition~\ref{prop:tight}, we combine this result with the bounds obtained in Proposition~\ref{prop:RGBoundOne} to show tightness for the family $\ZG{n}$ in an appropriate space. 

We start by comparing the quadratic variation $\langle \Rg(\cdot, x) \rangle_t$ to the bracket process $[\Rg(\cdot,x)]_t$ (see Appendix~\ref{AppA} for the different notions of quadratic variation for a martingale with jumps). Implicitly in the proof of Lemma~\ref{le:ItInt}, we have already seen that for $x\in \Le$ and $0 \leq s \leq t$, we have 
\begin{align}
\langle \Rg(\cdot, x) \rangle_s = 4 \co^2 \int_0^s \sum_{z \in \Le} \eps^2  \big(\Pg{t-r} \ae \Kg \big)^2 (x-z) \; \Cg(r,z) \, dr \;.\label{e:QuadrVarR}
\end{align}
%

%%%%%%%%%%%%%%%%%%
\begin{lemma}
%%%%%%%%%%%%%%%%%%
\label{le:quartic_var}
%%%%%%%%%%%%%%%%%%
For $x \in \Le$, let
\begin{equation}
\label{e:def:Qg}
\Qg(s,x) = [\Rg(\cdot,x)]_s - \langle \Rg(\cdot,x) \rangle_s\;,
\end{equation}
where $s \mapsto [\Rg(\cdot,x)]_s$ is the bracket process of the martingale $\Rg(\cdot,x)$. Let $\ga_0>0$ be the constant appearing in Lemma~\ref{le:Kg}. For any $t \ge 0$, $\ka > 0$ and $1 \le p < +\infty$, there exists $C = C(t,\ka,p)$ such that for $0 < \ga < \ga_0$,
$$
 \E\sup_{x \in \Le} \sup_{0 \le s \le t} |\Qg(s,x)|^p  \le C \ga^{p(1-\ka)}.
$$
\end{lemma}
%%%%%%%%%%%%%%%%%%

%%%%%%%%%%%%%%%%%%
\begin{proof}
%%%%%%%%%%%%%%%%%%
By monotonicity of stochastic $L^p$ norms, it suffices to show the statement for $p$ large enough. We wish to apply the Burkholder-Davis-Gundy inequality to the martingale $\Qg(\cdot,x)$ for a fixed $x \in \Le$ and in order to do so, we need to estimate the quadratic variation and the jumps of this martingale. 

Since the martingale $\Rg(\cdot,x)$ is of finite total variation, its bracket process is simply
$$
[\Rg(\cdot,x)]_s = \sum_{0 < r \le s} \Ll(\Delta_r \Rg(\cdot,x)\Rr)^2.
$$
Moreover, as the process $s \mapsto \langle \Rg(\cdot,x) \rangle_s$ is continuous, the jumps of $\Qg(\cdot,x)$ are identical to those of $s \mapsto[\Rg(\cdot,x)]_s$.

Recall that an update of the spin $\si(k)$ at microscopic position $k = \eps^{-1} z$ causes a jump of size $2\de^{-1}\eps^2 \Kg(y-z)$ for $\Mg(y)$. If such an event takes place at the macroscopic time $s$, then the martingale $\Rg(\cdot,x)$ has a jump at time $s$ of absolute value
\begin{equation}\label{e:JumpBoundDet}
2\eps^2\de^{-1} (\Pg{t-s} \ae  \Kg)(x-z) \le 2\eps^2\de^{-1} \|\Kg\|_{L^\infty(\Lambda_\eps)} \le  3 \ga \;,
\end{equation}
where we have used the fact that $\| \Pg{t}\|_{L^1(\Le)} =1$, cf.\ \eqref{e:Prepresentation}. 
Under the same circumstance, the jump of the martingale $\Qg(\cdot,x)$ at time $s$ is
$$
\Ll( 2\eps^2\de^{-1} (\Pg{t-s} \ae  \Kg)(x-z) \Rr)^2 
\le  (3 \ga)^2\;.
$$
This gives us the required estimate on the jumps of $\Qg(\cdot,x)$. We now turn to its quadratic variation. It is the same as the quadratic variation of the process $s \mapsto [\Rg(\cdot,x)]_s$, hence, 
\begin{eqnarray*}
\langle \Qg(\cdot,x) \rangle_t
&  = & \int_0^t \sum_{z \in \Le} \frac{\Cg(s,z)}{\ag} \Ll( 2\eps^2\de^{-1} (\Pg{t-s} \ae  \Kg)(x-z)\Rr)^4 \, ds \\
& \le & \frac{16\eps^6}{\ag \dg^4 } \int_0^t \sum_{z \in \Le} \eps^2 (\Pg{t-s} \ae  \Kg)^4(z) \, ds\;,
\end{eqnarray*}
where we used the fact that $\Cg(s,z) \leq 1$. Using again \eqref{e:JumpBoundDet} in the form
$$
\|\Pg{t-s} \ae  \Kg\|_{L^\infty(\Le)} \le 2\frac{\ga^2}{\eps^2}
$$
yields
\begin{equation}
\label{e:Qg}
\langle \Qg(\cdot,x) \rangle_t \le 2^6  \frac{ \eps^2 \ga^4}{\ag \dg^4  } \int_0^t \sum_{z \in \Le} \eps^2(\Pg{t-s} \ae  \Kg)^2(z) \, ds \;.
\end{equation}
We now use $(\eps^2 \ga^2 / \ag \dg^{4}) \le 2$ and \eqref{e:elementary} to obtain
\begin{equation}
\label{fromthis}
\langle \Qg(\cdot,x) \rangle_t \le 2^7\;e \, \ga^2 \sum_{\om \in \{ -N, \ldots, N \}^2} \frac{|\hKg(\om)|^2}{t^{-1} + 2 \ga^{-2} (1-\hKg(\om))} \;.
\end{equation}
For $0 < \ga < \ga_0$, the bounds  \eqref{e:K3} and \eqref{e:K2} imply that the sum over $|\om| \ge  \ga^{-1}$ is bounded uniformly.  By \eqref{e:K1} and \eqref{e:K2}, the sum over $|\om| \le C \ga^{-1}$ is smaller than $C(t) \log(\ga^{-1})$. We have shown the deterministic bound
\begin{equation}
\label{tothat}
\langle \Qg(\cdot,x) \rangle_t \le C \ga^2  \log(\ga^{-1})\;.
\end{equation}
By the Burkholder-Davis-Gundy inequality (Lemma~\ref{le:BDG}), it follows that for all $p >0$,
$$
\E \sup_{0 \le s \le t} |\Qg(s,x)|^p  \le C \ga^{p} \log^{p/2}(\ga^{-1})\;,
$$
where $C=C(p,t)$, and in particular the constant does not depend on $x \in \Le$. The conclusion is then obtained using the observation that
\begin{equation}
\label{e:easyobs}
\E \sup_{x \in \Le} \sup_{0 \le s \le t} |\Qg(s,x)|^p  \le \sum_{x \in \Le} \E\sup_{0 \le s \le t} |\Qg(s,x)|^p \;,
\end{equation}
and choosing $p$ sufficiently large.
\end{proof}
%%%%%%%%%%%%%%%%
%
%

%%%%%%%%%%%%%%%%
\begin{lemma}
%%%%%%%%%%%%%%%%
\label{le:jumps}
%%%%%%%%%%%%%%%%%
Let $\ga_0>0$ be the constant appearing in Lemma~\ref{le:Kg}. For any $t \ge 0$ and $1 \le p < +\infty$, there exists $C= C(t,p)>0$ such that for every $0 < \ga< \ga_0$,
$$
\Ll( \E\sup_{x \in \Le} \Big|\sum_{0 \le s \le t} \Ll(\Delta_s \Rg(\cdot,x) \Rr)^2 \Big|^p \Rr)^{1/p} \le C \log(\ga^{-1})\;.
$$
\end{lemma}
%%%%%%%%%%%%%%%%

%
%%%%%%%%%%%%%%%%%
\begin{proof}
%%%%%%%%%%%%%%%%%
We observe that
$$
\sum_{0 \le s \le t} \Ll(\Delta_s \Rg(\cdot,x) \Rr)^2 = \Qg(t,x) + \langle \Rg(\cdot,x) \rangle_t\;.
$$
By Lemma~\ref{le:quartic_var}, it thus suffices to show that
\begin{equation}
\label{leadshere}
\sup_{x \in \Le} \langle \Rg(\cdot,x) \rangle_t \le C \log(\ga^{-1})\;.
\end{equation}
We learn from \eqref{e:QuadrVarR} (as when passing from \eqref{e:Qg} to \eqref{fromthis}) that
$$
\langle \Rg(\cdot, x) \rangle_t  \le 
4 \co^2 \sum_{\om \in \{-N, \ldots,  N\}^2} \frac{ \big|  \hKg (\om)\big|^2}{ t^{-1} + 2\gamma^{-2}( 1 -\hKg(\om)\big)} \;. 
$$
We obtain \eqref{leadshere} arguing in the same way as from \eqref{fromthis} to \eqref{tothat}.
\end{proof}

We are now ready to prove that the $\RG{n}$ can approximately be represented as a Hermite polynomials applied to 
$\Rg$. Recall the recursive definition \eqref{e:def-Hermite1} of the Hermite polynomials $H_n = H_n(X,T)$ as well as 
the identities \eqref{e:partialX} and \eqref{e:partialT} for their derivatives.

We aim to bound the quantity  
\begin{equation}
\label{e:def:EGn}
\EG{n}(s,x) := H_n(\Rg(s,x), [\Rg(\cdot,x)]_s) - \RG{n}(s,x)\;,
\end{equation}
for any $x \in \T^2$. Here, we view $[\Rg(\cdot,x)]_s$ as defined on all of $\T^2$, by extending it as a trigonometric polynomial of degree $\leq N$.   Recall that according to \eqref{e:DefZn}, $\RG{n}(t,x) = \ZG{n}(t,x)$. 

\begin{proposition}[$\ZG{n}$ as a Hermite polynomial]
\label{p:discrete-Wick}
Let $\ga_0$ be the constant appearing in Lemma~\ref{le:Kg}. Then  for any $n \in \N$, $\ka > 0$, $t > 0$ and $1 \le p < \infty$, there exists $C =C(n,p,t,\ka)> 0$ such that for every $0 < \ga < \ga_0$,
\begin{equation}
\label{e:discrete-Wick}
\Ll( \E \sup_{x \in \T^2} \sup_{0 \le s \le t} | \EG{n}(s,x) |^p \Rr)^{1/p} \le C \ga^{1-\ka} \notag.
\end{equation}
\end{proposition}

\begin{proof}
We start by reducing the bound on the spatial supremum over $x \in \T^2$ to the pointwise bounds for $x$ in a grid. According to the definition, for every $n$ the function $\EG{n}$ is 
a trigonometric polynomial of degree $\leq nN$. For $n = 1$, Lemma~\ref{le:LinftyDiscreteCont} implies that we can control the supremum over $x \in \T^2$ by the supremum over $x \in \Le$ at the 
price of loosing an arbitrarily small power of $\eps$. For $n \geq 2$, this lemma does not apply directly, but we can circumvent this problem by refining the grid. Indeed, set $\Len = \{ x \in \frac{\eps}{n} \Z^2 \colon -1 < x_1, x_2 \leq 1\}  $. Then Lemma~\ref{le:LinftyDiscreteCont} applies (the fact that the number of grid points in this lemma is an odd multiple of the dimension is just for convenience of notation) and we can conclude that 
$\sup_{x \in \T^2}  \big| \EG{n}(s,x) \big| \leq C(\ka) \eps^{-\ka}\sup_{x \in \Len}  \big| \EG{n}(s,x) \big| $. Finally, as in \eqref{e:easyobs} we can reduce the bound on the supremum over $x \in \Len$ to bounds on a single point $x \in \Len$. We now proceed to derive such a bound.

The proof proceeds by induction on $n$. For $n = 1$, it is obviously true since $\EG{1} = 0$. We now assume $n \ge 2$.

To begin with, we observe that for any $n$, there exists $C$ such that for every $x,t \in \R$ and every $|h| \le 1, |s|\le 1$,
\begin{multline}
\label{e:controlH}
\!\!\!\!\!\! \Ll|H_n(x+h,t+s) - H_n(x,t) - \partial_X H_n(x,t) \, h - \frac{1}{2}\partial_X^2 H_n(x,t) \, h^2 - \partial_T H_n(x,t) \, s    \Rr| \\
\le C\Ll( |x|^{n-2} + |t|^{(n-2)/2} + 1 \Rr)\Ll( |h|^3 + |s|^2 \Rr).
\end{multline}
We fix $x \in \Len$ and use the shorthand notation 
$$
\unR(s) = (\Rg(s,x),[\Rg(\cdot,x)]_s)\;.
$$ 
By It\^{o}'s formula (see Lemma~\ref{e:ito-with-jumps} and Remark~\ref{r:ito-with-jumps}),
\begin{align}
\label{e:itoerr}
H_n(\unR(s))  = & \int_{r=0}^s \partial_T H_n\Ll(\unR(r^-)\Rr) \, d[\Rg(\cdot,x)]_r \notag\\
& +  \int_{r=0}^s \partial_X H_n\Ll(\unR(r^-)\Rr) \, d\Rg(r,x) \notag\\
& + \frac{1}{2} \int_{r=0}^s \partial^2_X H_n\Ll(\unR(r^-)\Rr) \, d[\Rg(\cdot,x)]_r \notag \\
& + \msf{Err}(s,x)\;,
\end{align}
where $\msf{Err}(s,x)$ is the error term caused by the jumps,
\begin{multline*}
\msf{Err}(s,x) = \sum_{0 < r \le s} \Big\{ \Delta_r H_n(\unR) - \partial_T H_n(\unR(r^-)) \Delta_r [\Rg(\cdot,x)]_\cdot \\
- \partial_X H_n(\unR(r^-)) \Delta_r \Rg(\cdot,x) - \frac{1}{2} \partial_X^2 H_n(\unR(r^-)) \Ll( \Delta_r \Rg(\cdot,x) \Rr)^2  \Big\}\;.
\end{multline*}
We have seen in \eqref{e:JumpBoundDet}  that uniformly over $r$ and $x \in \Le$, we have $\Ll|\Delta_r \Rg(\cdot,x)\Rr| \le 3 \ga$. Together with Lemma~\ref{le:LinftyDiscreteCont}, this implies that
\begin{equation}
\label{e:jumpss}
\sup_{0 \leq r \leq t} \sup_{x\in \Len}  \Ll|\Delta_r \Rg(\cdot,x)\Rr| \le C \ga^{1-\ka}\;
\end{equation}
for any given $\ka > 0$, where $C=C(\ka)$. Therefore,
$$
\Delta_r [\Rg(\cdot,x)]_\cdot = \Ll( \Delta_r \Rg(\cdot,x) \Rr)^2 \le \Ll(C \ga^{1-\ka}\Rr)^2.
$$
As a consequence, we can apply the estimate in \eqref{e:controlH} on the error term, to get
\begin{multline*}
\Ll|\msf{Err}(s,x)\Rr| \le C \Ll(\sup_{0 \le r \le s} |\Rg(r,x)|^{n-2} +\sup_{0 \le r \le s} [\Rg(\cdot,x)]_r^{(n-2)/2} + 1\Rr) \\
\times \sum_{0 < r \le s} \Ll\{ \Ll( \Delta_r \Rg(\cdot,x) \Rr)^3 + \Ll(\Delta_r \Rg(\cdot,x) \Rr)^4  \Rr\}.
\end{multline*}
In view of \eqref{e:jumpss}, it is clear that the summand $(\Delta_r \Rg(\cdot,x))^4$ can be neglected. From \eqref{e:jumpss} and Lemma~\ref{le:jumps}, we obtain that for every $1 \le p < +\infty$, there exists $C$ such that for every $s \le t$,
$$
\Ll(  \E \sup_{x\in \Le} \Big|\sum_{0 \le r \le s} \Ll(\Delta_r \Rg(\cdot,x) \Rr)^3 \Big|^p \Rr)^{1/p} \le C \ga^{1-\ka} \log(\ga^{-1})\;,
$$
and using Lemma \ref{le:LinftyDiscreteCont}, we can replace the supremum over $x \in \Le$ by a supremum over $x \in \Len$ in this bound (at the price of changing the exact value of the arbitrarily small $\ka$).

By the bounds derived in Lemma~\ref{le:ItInt}, Proposition~\ref{prop:RGBoundOne} and Lemma~\ref{le:jumps}, we also get that for every $\ka > 0$, there exists $C > 0$ such that for every $s \le t$ and $x \in \Len$, 
$$
\Ll(  \E\Big| \sup_{0 \le r \le s} |\Rg(r,x)|^{n-2} +\sup_{0 \le r \le s} [\Rg(\cdot,x)]_r^{(n-2)/2} + 1 \Big|^p\Rr)^{1/p} \le C\ga^{-\ka}.
$$
It follows from these observations that for every $\ka > 0$ and $1 \le p < +\infty$, there exists $C$ such that uniformly over $x \in \Le$,
\begin{equation}
\label{e:error-controlled}
\Big(\E \sup_{s \le t}\Ll|\msf{Err}(s,x)  \Rr|^p\Big)^{1/p} \le C \ga^{1-\ka}.
\end{equation}
Going back to the relation in \eqref{e:itoerr}, we can use \eqref{e:partialX} and \eqref{e:partialT} to see that 
$\partial_T H_n + \partial^2_X H_n/2 = 0$,
so the first and third integrals in \eqref{e:itoerr} cancel out. Using \eqref{e:partialX} again, we arrive at
$$
H_n(\unR(s)) 
= n \int_{r=0}^s H_{n-1}\Ll(\unR(r^-)\Rr) \, d\Rg(r,x) + \msf{Err}(s,x).
$$
In view of the definition \eqref{e:ZnA} of $\RG{n}(s,x)$ (which remains valid for $x \in \Len$), we can rewrite this as
\begin{equation}
\label{e:rec}
\EG{n}(s,x) = n \int_{r = 0}^s \EG{n-1}(r^-,x) \, d\Rg(r,x) + \msf{Err}(s,x) \;.
\end{equation}
Assuming that the Proposition is true for the index $n-1$, we want to prove that it holds for the index $n$. In fact, it suffices to prove that for every $\ka > 0$, every $t > 0$ and every $p$ sufficiently large, there exists $C > 0$ such that uniformly over $x \in \Le$,
\begin{equation}
\label{e:target}
\Big( \E \sup_{0 \le s \le t} | \EG{n}(s,x) |^p \Big)^{1/p} \le C \ga^{1-\ka},
\end{equation}
since we can later on argue as in \eqref{e:easyobs} to conclude.
The error term in \eqref{e:rec} will not cause any trouble by \eqref{e:error-controlled}. There remains to consider the integral in the right-hand side of \eqref{e:rec}. Since this integral is a martingale as $s$ varies, we can use the Burkholder-Davis-Gundy inequality, provided that we can estimate its quadratic variation and its maximal jump size. The quadratic variation at time $t$ is bounded by
$$
n^2 \sup_{s \le t} \Ll|\EG{n-1}(s,x)\Rr|^2 \langle \Rg(\cdot,x)\rangle_t,
$$
with $\langle \Rg(\cdot,x)\rangle_t \le C \log(\ga^{-1})$ by \eqref{leadshere}. The maximal jump size is bounded by
$$
n \, \sup_{s \le t} \Ll|\EG{n-1}(s,x)\Rr| \, \sup_{s \le t} \Ll|\Delta_s \Rg(\cdot,x)\Rr|,
$$
and we already saw that $\sup_{s \le t} \Ll|\Delta_s \Rg(\cdot,x)\Rr| \le C \ga^{1-\ka}$. The induction hypothesis and the Burkholder-Davis-Gundy inequality thus lead to \eqref{e:target}, and the proof is complete.
\end{proof}

Finally, we are ready to conclude and to prove the tightness of the processes $\ZG{n}$. Before we state the result, recall that for any separable metric space $\Ss$, we denote by $\Dd(\R_+, \Ss)$ the space of cadlag functions on $\R_+$ taking values in $\Ss$, endowed with the Skorokhod topology (see \cite[Chapters 16 and 18]{bill}). Recall in particular that according to \cite[Theorem~16.4]{bill}, a family of processes is tight on $\Dd(\R_+, \Ss)$ as soon as their restrictions to all compact time intervals are tight. 
%
%
%%%%%%%%%%%%%%%%%%
\begin{proposition}\label{prop:tight}
%%%%%%%%%%%%%%%%%%%
For any fixed $n\in \N$ and any $\al >0$, the family $\{ \ZG{n}, \ga \in (0,\frac13) \}$ is tight on $\Dd(\R_+,\Ca) $. Any weak limit is supported on $\Cc(\R_+,\Ca) $. Furthermore, for any $p \geq 1$ and $T>0$, we have%there exists $C = C(n,p,\al,T)$ such that 
\begin{equation}\label{e:FinalAPriori}
\sup_{\gamma \in (0,\frac13)} \E  \sup_{0 \leq t \leq T}  \big\| \ZG{n}(t, \cdot ) \big\|_{\Ca}^p   < \infty\;.
\end{equation}
\end{proposition}
%%%%%%%%%%%%%%%%%%

%
%
%%%%%%%%%%%%%%%
\begin{proof}
%%%%%%%%%%%%%%%
We can restrict ourselves to consider $0 < \ga< \ga_0$, where $\ga_0$ is the constant appearing in Lemma~\ref{le:Kg} and Proposition~\ref{prop:RGBoundOne}. We fix a $T>0$ and show tightness in $\Dd([0,T],\Ca)$.

Our strategy is similar to that of \cite{MuellerTribe}. Let $ m \in \N$ be fixed below. The estimate \eqref{e:RgRegularity2}
implies that for all $s \neq t \in \ga^m \N_0$, all $p \geq 1$, $\al^{\prime}>0$ ,$\la \leq \frac{1}{ 2 m}$, and $n \in \N$, we have
 \begin{equation}\label{e:tight0}
\E \| \ZG{n}(t, \cdot) - \ZG{n}(s,\cdot) \|^p_{\Cc^{-\al^{\prime}- 2 \la}} \leq C \Big( |t-s|^{\la}  + \ga^{\frac12} \Big)^p \leq  C |t-s|^{\la p}   \;,
 \end{equation}
where $C= C(n,p, \al^{\prime},  T, \lambda)$.
We now define the following continuous interpolation for $\ZG{n}$: set $\tZG{n}(t, \cdot) = \ZG{n}(t, \cdot)$ for $t \in \gamma^m \N_0$, and interpolate linearly between these points. It is easy to check that $\tZG{n}$ satisfies
 \begin{equation*}
\E \| \tZG{n}(t, \cdot) - \tZG{n}(s,\cdot) \|^p_{\Cc^{-\al^{\prime}- 2 \la}} \leq  C |t-s|^{\la p}   \;,
 \end{equation*}
for all values of $s,t \in [0,T]$ and hence, the Kolomogorov criterion implies the desired properties when $\ZG{n}$ is replaced by $\tZG{n}$.
 
 We claim that for any $\ka >0$ and $p \geq 1$, we have
 \begin{equation}\label{e:tight1}
 \E  \sup_{0 \leq t \leq T} \sup_{x \in \T^2} \big| \tZG{n}(t, x) - \ZG{n}(t, x) \big|^p  \leq C(n,p,T,\ka) \, \gamma^{(1 - \ka)p}\;.
 \end{equation}
 Once we have established this bound, the proof is complete.
 
By monotonicity of $L^p$-norms, it is sufficient to establish \eqref{e:tight1} for large $p$.  We treat the case $n=1$ first. In view of Lemma~\ref{le:LinftyDiscreteCont}, it suffices to establish \eqref{e:tight1} with the supremum over $x \in \T^2$ replaced by the supremum over $x \in \Le$.  We fix an interval $I_k = [k \ga^m, (k+1) \ga^m]$ for some $k \in \N_0$ and an $x \in \Le$, and we start with the estimate
 \begin{equation*}
\sup_{t \in I_k} |\tZg(t,x) - \Zg(t,x) | \leq 2 \sup_{t \in I_k} |\Zg(t,x) - \Zg(k \ga^m,x) |\;.
 \end{equation*}
 Using the definition \eqref{e:DefZg} of $\Zg$, the definition of $\Mg$ just above \eqref{e:error}, as well as \eqref{e:evolution1}, we get for any $t \in I_k$
 \begin{align}
& \Zg(t,x) - \Zg(k \ga^m,x) \notag\\
 &= \int_{k \ga^m}^t \Dg \Zg(s,x) \, ds + \big( \Mg(t,x) - \Mg(k \ga^m,x) \big) \notag\\
 &= \int_{k \ga^m}^t  \Dg \Zg(s,x) \, ds  - \frac{1}{\delta} \bigg(\int_{\tfrac{\ga^m k}{  \ag}}^{\tfrac{t}{\ag}} \LgN \hg\big(s,\tfrac{x}{\eps}\big)  \, ds \bigg) \notag\\
 & \qquad \qquad \qquad \qquad \qquad \qquad + \frac{1}{\delta} \Big(\hg\big(\tfrac{t}{\ag} , \tfrac{x}{\eps} \big) - \hg\big(\tfrac{k \ga^r}{\ag} , \tfrac{x}{\eps} \big)   \Big)\;. \label{e:tight2}
 \end{align}
We bound the terms on the right-hand side of \eqref{e:tight2} one by one:  using the definition of $\Dg$, we get for  the first term that
 \begin{align*}
\sup_{x \in \Le}\sup_{ t \in I_k} \bigg| \int_{k \ga^m}^t \Dg \Zg(s,x) \, ds  \bigg|& \leq \frac{C}{\ga^2} \int_{I_k}\big\|  \Zg(s,\cdot) \big\|_{L^\infty(\T^2)}  \, ds \\
&\leq  \frac{C(\ka^{\prime})}{\ga^{2+2 \ka^{\prime}}} \int_{I_k}\big\|  \Zg(s,\cdot) \big\|_{\Cc^{-\ka^{\prime}}} \, ds  \;.
 \end{align*}
In the second inequality we have used Lemma~\ref{l:unif-Fourier-cut} for an arbitrary $\ka^{\prime} >0$. Hence we get for any $p \geq 1$
\begin{align*}
\E &\sup_{k \leq T \ga^{-m}} \sup_{x \in \Le} \sup_{t \in I_k}  \bigg| \int_{k \ga^m}^t  \Dg \Zg(s,x) \, ds \bigg|^p   \notag\\
&\leq \sum_{k \leq T \ga^{-m}}    \E \sup_{x \in \Le}  \sup_{t \in I_k}  \bigg| \int_{k \ga^m}^t  \Dg \Zg(s,x) \, ds \bigg|^p   \\
&\leq C(\ka^{\prime},p)  \sum_{k \leq T \ga^{-m}}  \ga^{- (2+2\ka^{\prime})p }  \; \E \Big( \int_{I_k}\big\|  \Zg(s,\cdot) \big\|_{\Cc^{-\ka^{\prime}}} \, ds   \Big)^p  \\
&\leq C(\ka^{\prime}, p,T) \, \ga^{-m}  \ga^{- (2+ 2 \ka^{\prime}) p }  \ga^{m p} \; \sup_{0 \leq t \leq T+ \ga^m}\E \big\|  \Zg(s,\cdot) \big\|_{\Cc^{-\ka^{\prime}}}^p   \;.
\end{align*}
By \eqref{e:RgRegularity1}, the supremum on the right-hand side of this expression is bounded by a constant depending on $T, \ka^{\prime}$ and $p$, so that the whole expression can be bounded by 
 \begin{equation*}
C(\ka^{\prime},p,T) \, \ga^{p \Ll(m\Ll(1 - \frac{1}{p}\Rr) - (2 + 2 \ka^{\prime})   \Rr)}\;.
 \end{equation*}
 Choosing $m \geq 3$, $\ka^{\prime}$ small enough and $p$ large enough  we can obtain any exponent of the form $p (1-\ka)$ for $\ga$ ($\ka^{\prime} < \frac{\ka}{4}$ and $p > \frac{3}{2 \ka}$ suffice).

For the second term on the right-hand side of \eqref{e:tight2}, we use the deterministic estimate $| \LgN \hg(s,k)| \leq 2$ which holds for any $k \in \LN$ and any time $s$ to get for any $x \in \Le$ and $k \leq T \ga^{-m}$ and any $t \in I_k$
 \begin{equation*}
 \frac{1}{\delta} \bigg(\int_{\tfrac{\ga^m k}{  \ag}}^{\tfrac{t}{\ag}} \LgN \hg\big(s,\tfrac{x}{\eps}\big)  \, ds \bigg) \leq 2 \frac{\ga^m}{\ag \delta} \leq 2 \ga^{m-3}\;.
 \end{equation*} 
 Hence this term satisfies the estimate \eqref{e:tight1} as soon as $m \geq 4$.
 
 Let us turn to the third term on the right-hand side of \eqref{e:tight2}. The process $\hg\big(\cdot,\tfrac{x}{\eps}\big)$ only evolves by jumps. Let us recall that a jump-event at position $j \in \LN$ at time $s \in \big[ \tfrac{k \ga^m}{\ag} , \tfrac{t}{\ag} \big]$ causes a jump of magnitude $2 \kg\big(\tfrac{x}{\eps}, j\big) \leq 3 \ga^2$ for $\hg\big(\cdot, \tfrac{x}{\eps}\big)$. Hence, we have 
 \begin{equation*}
 \sup_{x \in \Le} \sup_{t \in I_k} \frac{1}{\de} \Big(\hg\big(\tfrac{t}{\ag} , \tfrac{x}{\eps} \big) - \hg\big(\tfrac{k \ga^m}{\ag} , \tfrac{x}{\eps} \big) \Big) \leq 3 \ga J_k \;,
 \end{equation*}
 where $J_k$ is the total number of jumps at all locations $j \in \LN$ during the time interval $\big[ \tfrac{k \ga^m}{\ag} , \tfrac{(k+1) \ga^m}{\ag} \big]$. According to \eqref{e:rate}, the jump rate at any given location is always bounded by $1$, so the total jump rate is bounded by $|\Le|$. This implies that for every~$k$, the random variable $J_k$ is stochastically dominated by $\msf{Poi}(\la)$, a Poisson random variable with mean $\la = \ga^{m} \ag^{-1} |\Le| \le C \ga^{m-6}$. We impose $m >6$, so that this rate goes to zero. We note that
\begin{equation*}
\E \sup_{k \le T \ga^{-m}} J_k^p   \le \sum_{k \le T\ga^{-m}} \E J_k^p  \le T\ga^{-m} \, \E \, \msf{Poi}(\la)^p \;.
\end{equation*}
Since $\E\Ll[ \msf{Poi}(\la)^p \Rr] \le C(p) \ga^{m-6}$, we arrive at
\begin{equation}\label{e:tight1B}
\ga^p  \E \Big[\sup_{k \leq T \ga^{-m}}  J_k^p \Big] \leq    C(p,T) \ga^{p-6}\;,
\end{equation}
and as above we can obtain the exponent $p(1-\ka)$ by choosing $p > \frac{6}{\ka}$. Summarising these calculations and invoking Lemma~\ref{le:LinftyDiscreteCont}, we see that for any $\ka >0$, $m>6$ and $p>\frac{6}{\ka}$ there exists a constant $C=C(p,T,\ka)$ such that 
\begin{equation}\label{e:tight3}
\E \sup_{k \leq T \ga^{-m}}  \sup_{x \in \T^2} \sup_{t \in I_k} |\Zg(t,x) - \Zg(k \ga^m,x) |^p \leq C \ga^{p(1-\ka)} \;.
\end{equation}
which establishes \eqref{e:tight1} in the case $n=1$.

We now proceed to prove \eqref{e:tight1} in the case $n \geq 2$. According to Proposition~\ref{p:discrete-Wick}, it suffices to show
 \begin{align*}
 \E & \sup_{0 \leq t \leq T} \sup_{x \in \T^2} \big| H_n(\unZ(t,x)) - \widetilde{H_n(\unZ) }(t,x)\big|^p \leq C \, \gamma^{(1 - \ka)p}\;.
 \end{align*}
where $\unZ(t,x):= (\Zg(t, x), [\Rg(\cdot,x)]_t)$, and $ \widetilde{H_n(\unZ)} $ denotes the process obtained by evaluating $H_n(\unZ(t,x))$ at points $t \in \ga^m \N$ exactly and taking the linear interpolation in between. It is easy to see that for fixed $x \in \T^2$ and $t \in I_k$ we have
\begin{align}
\big| & H_n(\unZ(t,x)) - \widetilde{H_n(\unZ) }(t,x)\big| \notag \\
&\leq C(n)  \big(1 +  \sup_{t\in I_k} |\Zg(t,x)|^{n} + \sup_{t \in I_k} [ R_{\ga,t}(\cdot,x)]_t^{(n/2)}  \big)  \notag\;\\
&  \times \Big( \sup_{t \in I_k}|\Zg(t,x) - \Zg(k \ga^m,x) | + \sup_{t \in I_k}|[ R_{\ga,t}(\cdot,x)]_t - [ R_{\ga,k\ga^m}(\cdot,x)]_{k \ga^m} |   \Big) \;.\label{e:tight3A}
\end{align}
We have already established \eqref{e:FinalAPriori}  for $n=1$, which yields a bound on all moments of $\sup_{t\in I_k} |\Zg(t,x)|^{n}$ as well as  \eqref{e:tight3}, so that it remains to bound the terms involving the bracket process. We start with the last term on the right-hand side of \eqref{e:tight3A}. For a fixed $k$, we write $s= k \ga^m$ and get for any $t \in I_k$
\begin{align}
|[ &R_{\ga,t}(\cdot,x)]_t - [ R_{\ga,s}(\cdot,x)]_{ s} | \notag\\
&  \leq \sum_{0 \leq r \leq s} \big| (\Delta_r  \Rg(\cdot,x) )^2 - (\Delta_r R_{\ga,s}(\cdot,x) )^2 \big| + \sum_{s \leq r \leq t}  \Delta_r  (\Rg(\cdot,x) )^2\;.\label{e:tight4}
\end{align}
Recall, that a jump of spin $\sigma(k)$ for $k = \eps^{-1}z$ at time $r$ causes a jump of absolute value $  2 \eps^2 \delta^{-1}\Pg{t-r} \ae \Kg (x-z)$ for $\Rg( \cdot,x)$, so that the first term on the right-hand side of \eqref{e:tight4} can be bounded by 
\begin{align*}
\sum_{0 \leq r \leq s} & \big| \Delta_r  (\Rg(\cdot,x) )^2 - (\Delta_r R_{\ga,s}(\cdot,x) )^2 \big|  \\
&\leq 4 \frac{\eps^4}{\delta^2} J \sup_{z \in \Le} \sup_{0 \leq r \leq s} \big| (\Pg{t-r} \ae \Kg(z))^2 -  (\Pg{s-r} \ae \Kg(z))^2 \big|\;,
\end{align*}
where $J$ is the total number of jumps at any point in $\LN$ and any point in time before $T$ in macroscopic time, which corresponds to $T/\ag$ in microscopic units. As the jump rate at any given point is always bounded by $1$, $J$ is stochastically dominated by a Poisson variable $\msf{Poi}(\la)$ with mean $\la \leq C T \eps^{-2} \ag^{-1} \leq CT \ga^{-6}$, which implies in particular that for every $p \geq1$,
\begin{align*}
\E J^p \; \leq C(p,T) \ga^{-6 p} \;.  
\end{align*}
On the other hand, for any $0 \leq r \leq s$ and $z \in \Le$, we can write
\begin{align*}
\big|& (\Pg{t-r} \ae \Kg(z))^2 -  (\Pg{s-r} \ae \Kg(z))^2 \big| \\
&\leq \big| \Pg{t-r} \ae \Kg(z)+  \Pg{s-r} \ae \Kg(z) \big| \, \big| \Pg{t-r} \ae \Kg(z) -  \Pg{s-r} \ae \Kg(z) \big|  \\
&\leq  C \ga^{-2}  \log(\ga^{-1}) \, \big| \Pg{t-r} \ae \Kg(z) -  \Pg{s-r} \ae \Kg(z) \big|\;, 
\end{align*}
where we have made use of \eqref{e:P0} in the last line. We continue by bounding brutally
\begin{align*}
\big| &\Pg{t-r} \ae \Kg(z) -  \Pg{s-r} \ae \Kg(z) \big|  \\
 &= \Big|  \frac{1}{4} \sum_{\om \in \{ -N,\ldots,N\}^2  } e^{i \pi \om \cdot z}  \hKg(\om)  \; \big( e^{\frac{t-r}{\ga^2} ( \hKg(\om) -1)   } -  e^{\frac{s-r}{\ga^2} ( \hKg(\om)-1)   } \big)  \Big| \\
&\leq C \ga^{-4} \frac{t-s}{\ga^2} \; ,
\end{align*}
where we have used the bound $|\hKg(\om)| \leq 1$ twice. Summarising these bounds, we get for any $p$ 
\begin{align*}
\E& \sup_{k \leq T\ga^{-m}}\sup_{x \in \Le}\sup_{t \in I_k} \Big(\sum_{0 \leq r \leq s}  \big| \Delta_r  (\Rg(\cdot,x) )^2 - (\Delta_r R_{\ga,k \ga^m}(\cdot,x) )^2 \big|  \Big)^p\\
& \leq C(p,T)\ga^{mp} \, \ga^{-8p }\log(\ga^{-1})^p\;,
\end{align*}
which is bounded by $C(p,T) \ga^{p(2-\ka)}$ as soon as $m \geq 10$.

For the second term on the right-hand side of \eqref{e:tight4}, we write using \eqref{e:JumpBoundDet}
\begin{align*}
\sum_{k \ga^m \leq r \leq (k+1)\ga^m}  \Delta_r  (\Rg(\cdot,x) )^2 \leq C \ga^2 J_k \;,
\end{align*}
where as above $J_k$ is the total number of jumps at all locations $j \in \LN$ during the time interval $\big[ \tfrac{k \ga^m}{\ag} , \tfrac{(k+1) \ga^m}{\ag} \big]$. Repeating the argument that leads to \eqref{e:tight1B}, we get that for $m >6$
\begin{align*}
\ga^{2p}  \E \sup_{k \leq T \ga^{-m}}  J_k^p  \leq    C(p,T) \, \ga^{2p-6}\;.
\end{align*}
Summarising these calculations, we get that for $m \geq 10$ and $p > \frac{6}{\ka}$, there exists $C=C(p,T,\ka)$ such that 
\begin{align}
\E \sup_{k \leq T\ga^{-m}}\sup_{x \in \Le}\sup_{t \in I_k} \big|[ R_{\ga,t}(\cdot,x)]_t - [ R_{\ga,k \ga^m}(\cdot,x)]_{ k \ga^m} \big|^p  \leq C \ga^{p (2 -\ka)}\,. \label{e:tight10}
\end{align}
Finally, going back to \eqref{e:tight3A}, it remains to bound 
\begin{align*}
\E &\Big( \sup_{k \leq T \ga^{-m}}\sup_{x \in \Le }\sup_{t \in I_k} [ R_{\ga,t}(\cdot,x)]_t^{(n/2)}  \Big) ^p \\
& \leq C  \sum_{k \leq T \ga^{-m} }  \E\big\| \;  [ R_{\ga,k \ga^m}(\cdot,\cdot)]_{k \ga^m}^{(n/2)} \; \big\|_{L^\infty(\Le)}^p  \\
&\qquad + C \; \E \sup_{k \leq T\ga^{-m}}\sup_{x \in \Le}\sup_{t \in I_k} \big|[ R_{\ga,t}(\cdot,x)]_t - [ R_{\ga,k \ga^m}(\cdot,x)]_{ k \ga^m} \big|^p\;.
\end{align*}
In view of \eqref{e:tight10}, the second term can be neglected. Using Lemma~\ref{le:jumps}, we see that the first term is bounded by $C(p,T) \ga^{-m} \log(\ga^{-1})^{mp}$. Plugging all of this back into \eqref{e:tight3A}, invoking Lemma~\ref{le:LinftyDiscreteCont} once more, and using H\"older's inequality (with a large exponent on the terms in the last line of \eqref{e:tight3A})  we get for  any $\ka>0$ and $p$ large enough
\begin{align*}
\E \sup_{x  \in \Le} \sup_{0 \leq t \leq T} \big| & H_n(\unZ(t,x)) - \widetilde{H_n(\unZ) }(t,x)\big|^p  \leq C(n,p,T,\ka) \ga^{p(1-\ka)} \;,
\end{align*}
and the proof is complete. 
%
%
%%%%%%%%%%%%%%
\end{proof}
%%%%%%%%%%%%%

%
%
%%%%%%%%%%%%%%%%%%%%%%%%
\section{Convergence in law of the linearised system}
%%%%%%%%%%%%%%%%%%%%%%%%
\label{sec:conv-lin}
%%%%%%%%%%%%%%%%%%%%%%%%
The aim of this section is to prove the convergence in law of $\Zg$  and the approximate Wick powers $\ZG{n}$ to the solution of the stochastic heat equation and its Wick powers. 
We will  only be able to show the convergence in law of $\Zg$ and $\ZG{n}$ up to  a stopping time that depends on $\Xg$ (the ``non-linear'' dynamics), which we do not control for now. For a fixed  $\al \in (0,\frac12)$, any $\nn>1$ and $0 <\ga <1$, we set 
\begin{equation}
\label{e:deftaug}
\taun := \inf \big\{ t \ge 0 : \|\Xg(t,\cdot)\|_{\Ca} \ge \nn \big\} \;.
\end{equation}
The following Theorem~\ref{t:converg-lin} states, roughly speaking, that $\Zg$ converges to $Z$ (the solution of the stochastic heat equation introduced in Section~\ref{sec:ContinuousAnalysis}) ``until $\taun$". In order to state this properly, we introduce a different extension of $\Zg$ beyond $\taun$. We start by modifying the microscopic jump process $\si$ for times $t \geq \taun$. Indeed, for $k \in \LN$ and for $t \geq 0$, define
\begin{equation*}
\sgn(t,k) := 
\left\{
\begin{array}{ll}
\sigma(t,k)  & \text{if } t <   \frac{\taun}{\ag}, \\
\ssgn(t,k) & \text{otherwise}\;.
\end{array}
\right.
\end{equation*}
Here $\ssgn$ is a spin system with $\ssgn(\taun/\ag,k ) =\si(\taun/\ag,k )$ and with jumps occurring for every $t > \taun/\ag$ and every $k \in \LN$ at rate $\frac12$, independently from $\si$. We now construct processes $\Mgn$ and $\Zgn$ following exactly the construction of $\Mg$ and $\Zg$ with $\si$ replaced by $\sgn$. For $t \geq 0$ and $k \in \LN$ we set $\hgn(t,k) = (\sgn(t, \cdot) \star \kg)(k)$ and as in \eqref{e:evolution1} we set 
\begin{equation}\label{e:evolution1n}
\mgn(t,k) := \hgn(t,k) - \hgn(0,k) - \int_0^t \Lgn^s \, \hgn(s, k) \, ds \;,
\end{equation} 
where $\Lgn^s$ is defined as in \eqref{e:Generator} with $\cg$ replaced by
\begin{equation}\label{e:raten}
\cgn^s =   
\left\{
\begin{array}{ll}
 \cg & \text{if } s <   \frac{\taun}{\ag} \;, \\
 \frac12 & \text{otherwise}\;
\end{array}
\right.
\end{equation}
(in other words, $\Lgn^s$ is the infinitesimal generator of $\ssgn$).
Finally, let  $\Mgn(t,x) := \frac{1}{\dg} \mgn\Big(\frac{t}{\ag}, \frac{x}{\eg} \Big) $. The processes $\Mgn(\cdot, x)$ are martingales with quadratic variations given by \eqref{e:QuadrVar2} with $\Cg$ replaced by
\begin{equation}
\label{e:def:Cgn}
\Cgn(s,z):= \cgn^s(\sgn(s/\ag), z/\eps) \; .
\end{equation}
Let $\RGn{n}$ and $\ZGn{n}$ be defined as iterated stochastic integrals just as $\RG{n}$ and $\ZG{n}$ in \eqref{e:ZnA}, \eqref{e:ZnB}, and \eqref{e:DefZn}, but with $\Mg$ replaced by $\Mgn$. It is clear that for $s \leq \taun$, we have 
$\RGn{n}(s, \cdot) = \RG{n}(s, \cdot)$. Furthermore, the main results of the previous two sections, i.e. Proposition~\ref{prop:RGBoundOne}, Lemma~\ref{le:quartic_var}, Proposition~\ref{p:discrete-Wick}, as well as Proposition~\ref{prop:tight} hold true unchanged if $\RG{n}$ and $\ZGn{n}$ are replaced by $\RGn{n}$ and $\ZGn{n}$. Indeed, the only property used in these proofs concerning the jump rate $\Cg$ is the fact that it is bounded by $1$, and this remains true for $\Cgn$.

%%%%%%%%%%
\begin{theorem}[Convergence of $\Zg$]
%%%%%%%%%%
\label{t:converg-lin}
%%%%%%%%%%
Let $\al \in (0,1/2) $ and $\nn > 1$.  As $\gamma$ tends to $0$, the processes $\Zgn$ converge in law to $Z$ with respect to the Skorokhod topology on $\Dd(\R_+,\Ca)$, where $Z$ is defined in Proposition~\ref{prop:daPratoDebussche}. 
%%%%%%%%%%
\end{theorem}
%%%%%%%%%%
%
%%%%%%%%%%
\begin{proof}
%%%%%%%%%%
%
As in equation \eqref{e:DefZg} on $\T^2$, we have that
\begin{equation}
\label{e:evolution-lin}
\Zgn(t,x) =     \int_0^t  \Dg \Zgn(s,x)\, ds  + \Mgn(t,x)\;.
\end{equation}
As explained above, Proposition~\ref{prop:tight} also applies to the family of processes $\Zgn$. Therefore, for any fixed $\nn > 1$ the family $( \Zgn)_{ \ga  \in (0,\frac13) }$ is  tight for the Skorokhod topology on $\Dd(\R_+,\Ca)$. Let $\ov{Z}$ be a sub-sequential limit of $\Zgn$ as $\ga$ tends to $0$. It suffices to show that the law of $\ov{Z}$ is that of $Z$. In order to see this, we appeal to the martingale characterization of this law, as recalled in Theorem~\ref{t:martchar} below. By Proposition~\ref{prop:tight}, $\ov{Z}$ must take values in $\Cc(\R_+,\Ca)$.

Let $\phi \in \Cc^\infty(\T^2)$. We define 
$$
\Mm_{\ga,\phi}(t) = (\Zgn(t) , \phi)  -  \int_0^t  (\Zgn(s), \Dg \phi)  \, ds\;,
$$
where 
$(f,g) = \int f(x) g(x) \, dx$. More generally, when $f$ is a distribution and $g$ a smooth test function, we write $(f,g)$ to denote the evaluation of $f$ on $g$. As $( f, \Dg g) = ( \Dg f, g)$, the process $\Mm_{\ga,\phi}$ is a martingale. 

Let us first see that for any $s \ge 0$,
\begin{equation}
\label{e:Dg-Delta}
\Ll| (\Zgn(s), \Dg \phi) - (\Zgn(s), \Delta \phi) \Rr| \le C \gamma^{2-2\al} \|\Zgn(s)\|_{\Ca}\;,
\end{equation}
for a constant $C = C(\phi)$. Indeed, according to the definitions of $\Dg$ and $\Kg$ (see the discussion below \eqref{e:evolution2} as well as \eqref{e:DefDg}) 
\begin{align*}
\Dg \phi(x) &= \ct \frac{1}{\ga^2}\sum_{z \in \Lg } \ga^2 \KK(z) \Big( \phi \big(x+\frac{\eg}{\ga}z\big) - \phi(x) \Big)\\
&= \ct \frac{1}{\ga^2}\sum_{z \in \Lg } \ga^2 \KK(z) \Big(  \sum_{j=1,2}\partial_j \phi(x) \frac{\eps}{\ga}z_j  + \frac{1}{2} \sum_{j_1,j_2=1,2} \partial_{j_1} \partial_{j_2} \phi(x) \frac{\eps^2}{\ga^2}z_{j_1} z_{j_2} \notag\\
&\qquad \qquad  + \frac{1}{6} \sum_{j_1,j_2,j_3=1,2} \partial_{j_1} \partial_{j_2} \partial_{j_3} \phi(x)\frac{\eps^3}{\ga^3} z_{j_1} z_{j_2} z_{j_3}    \Big)    + \msf{Err} \;.
\end{align*}
The error term $\msf{Err}$ is readily seen to be bounded by $C\ga^2$ uniformly in $x$. Furthermore, by the symmetry of $\KK$, all of the sums involving odd powers of $z_j$ vanish. The only remaining contribution is 
\begin{equation*}
 \Delta \phi  \; \frac{1}{2}  \ct\sum_{z \in \Lg } \ga^2 \KK(z) z_1^2\;,
\end{equation*}
where we used symmetry of the kernel again. The Riemann sum converges to the integral $\int \KK(z) z_1^2 \, dz = 2$ and the error is bounded by $C\ga^2$ (see Remark~\ref{rem:ApproxInt}). Therefore, we get
$$
\|\ \Dg \phi - \Delta \phi\|_{L^\infty} \le C \gamma^2.
$$
The estimate \eqref{e:Dg-Delta} thus follows from Lemma~\ref{l:unif-Fourier-cut} and the fact that the Fourier coefficients of $\Zg$ with frequency larger than $N \leq \ga^{-2}$ vanish.

For any $L > 0$ and $z \in \Dd(\R_+,\Ca)$, we define
$$
T_L(z) = \inf\{t \ge 0 : \|z\|_{\Ca} > L\}
$$
and for any $t \ge 0$,
$$
\MM_{z,\phi}(t) = 
(z(t),\phi) - \int_0^t (z(s),\Delta \phi) \, ds\;.
$$
The first condition from Theorem~\ref{t:martchar} that needs to be checked is that $\MM_{\ov{Z},\phi}$ is a local martingale.
For some $s \ge 0$, we give ourselves a bounded continuous function $F : \Dd(\R_+,\Ca) \to \R$ that is measurable with respect to the $\sigma$-algebra over $\Dd([0,s],\Cc^{-\al})$, and we consider, for $t \ge s$,
$$
\GG_{L,t}(z) = \Ll(\MM_{z,\phi}(t \wedge T_L(z)) - \MM_{z,\phi}({s \wedge T_L(z)})\Rr) F(z)\;,
$$
where we slightly abuse notation by writing $\MM_{z,\phi}({t \wedge T_L(z)})$ to denote the process that is equal to $\MM_{z,\phi}(t)$ if $t < T_L(z)$, and is equal to the left limit of $\MM_{z,\phi}$ at $T_L(z)$ otherwise.
Let us define
$$
\msf{Loc} = \Ll\{ L > 0 : \P\Ll[\|\ov{Z}\|_{\Cc^{-\al}} \text{ has a local maximum at height } L  \Rr]  > 0 \Rr\}\;.
$$
Noting that
$$
\msf{Loc} \subset \bigcup_{\substack{ n,m \in \N \\ s \in \Q_+ }} \Ll\{ L > 0 : \P\Ll[ \sup_{t:|t-s| \le 1/m} \|\ov{Z}(t)\|_{\Cc^{-\al}} = L \Rr] \ge \frac{1}{n} \Rr\}\;,
$$
we see that this set is countable. For $L \notin \msf{Loc}$, the process $\ov{Z}$ belongs a.s.\ to the set of continuity points of the mapping
$$
\Ll\{\begin{array}{lll}
\Dd(\R_+,\Ca) & \to & \R \\
z & \mapsto & T_L(z)\;.
\end{array}
\Rr.
$$
Similarly, $\ov{Z}$ belongs a.s.\ to the set of continuity points of the mapping
$$
\Ll\{\begin{array}{lll}
\Dd(\R_+,\Ca) & \to & \R \\
z & \mapsto & \GG_{L,t}(z) \;.
\end{array}
\Rr.
$$
By the continuous mapping theorem, $\GG_{L,t}(\Zgn)$ thus converges in law to $\GG_{L,t}(\ov{Z})$ along the sub-sequence $\gamma \to 0$. Since $z \mapsto \GG_{L,t}(z)$ is uniformly bounded, the expectations converge as well, i.e.\ $\E[\GG_{L,t}(\Zgn)]$ converges to $\E[\GG_{L,t}(\ov{Z})]$ along the sub-sequence $\gamma \to 0$. Moreover, $\MM_{\Zgn,\phi}({t \wedge T_L(\Zgn)})$ is very close to $\Mm_{\ga,\phi}({t \wedge T_L(\Zgn)})$, as shown by \eqref{e:Dg-Delta} (and where we use the same abuse of notation on $t \wedge T_L(\Zgn)$). More precisely,
\begin{multline*}
\E\Ll[ \Ll| \GG_{L,t}(\Zgn) - \Ll(\Mm_{\ga,\phi}({t \wedge T_L(\Zgn)}) - \Mm_{\ga,\phi}({s \wedge T_L(\Zgn)})\Rr) F(\Zgn) \Rr| \Rr] \\
\le C(\phi) \, t \gamma^{2-2\al} L \|F\|_{L^\infty}\;.
\end{multline*}
By the martingale property of $\Mm_{\ga,\phi}(t)$, we have
$$
\E\Ll[ \Ll(\Mm_{\ga,\phi}({t \wedge T_L(\Zgn)}) - \Mm_{\ga,\phi}({s \wedge T_L(\Zgn)})\Rr) F(\Zgn) \Rr] = 0\;,
$$
so $\E[\GG_{L,t}(\Zgn)]$ tends to $0$ as $\gamma$ tends to $0$. This implies that $\E[\GG_{L,t}(\ov{Z})] = 0$, and thus that $(\MM_{\ov{Z},\phi}({t \wedge T_L(\ov{Z})}))_{t \ge 0}$ is a martingale. Since the set $\msf{Loc}$ is countable, we can choose a sequence $L_n \notin \msf{Loc}$ that goes to infinity with $n$. For such a sequence, $T_{L_n}(\ov{Z})$ tends to infinity a.s.\ as $n$ tends to infinity. We have thus proved that $\MM_{\ov{Z},\phi}$ is a local martingale, which is what we wanted.

For $z \in \Dd(\R_+,\Cc^{-\al})$, let
$$
\Gamma_z(t) = \Ll(\MM_{z,\phi}(t)\Rr)^2 - 2 t \|\phi\|_{L^2}^2\;.
$$
In order to verify the assumptions of Theorem~\ref{t:martchar}, there remains to see that $\Gamma_{\ov{Z}}$ is a local martingale.
The reasoning is similar, except that unlike the first part, the argument relies on the presence of the stopping time $\taun$. Let us see how. Recalling $\Mgn$ from \eqref{e:evolution-lin}, we have
$$
\Mm_{\ga,\phi}(t) = (\Mgn(t), \phi)\;,
$$
which coincides with $ (\Mg(t), \phi)$ if $t \leq \taun$. If we assume further that $\phi$ is a trigonometric polynomial of degree $K < \infty$ (i.e.\ $\phi(x) = \frac14\sum_{|\om_j| \leq K} \hat{\phi}(\om) e^{i \pi \om \cdot x}$), then by \eqref{e:Pars}, we have the identity
$$
(\Mgn(t), \phi) =  \sum_{x \in \Lambda_\eg} \eg^2 \Mgn(t,x) \, \phi(x)
$$
for $\ga^{-2} \geq K$.  By \eqref{e:QuadrVar2}, the predictable quadratic variation of $\Mm_{\ga,\phi}$ is given by
$$
\langle \Mm_{\ga,\phi} \rangle_t = 4 \co^2  \int_0^t \sum_{x,y,z \in \Lambda_\eps} \eg^6 \, \phi(x) \phi(y) \, \Kg(x-z) \Kg(y-z) \, \Cgn(s,z)   \, ds
$$
with $\Cgn$ as in \eqref{e:def:Cgn}. The central point is the observation that for $t < \taun$, 
\begin{equation}
\label{e:cg-control}
\Ll|\Cgn(t,z) - \frac12 \Rr| = \Ll|\Cg(t,z) - \frac12 \Rr| \le C(\al) \gamma^{1-2\al} \|\Xg(t)\|_{\Ca}\;,
\end{equation}
whereas for $t > \taun$, we have $\Cgn = \frac12$ by definition.

Indeed, to see \eqref{e:cg-control}  we note that by \eqref{e:rate} and \eqref{e:defXg},
$$
\Ll| \Cg(t,z) - \frac12 \Rr| \le \bg \big| \hg(\si(t/\ag),z/\eg)  \big|= \bg \dg \big| \Xg(t,z) \big| \;.
$$
Moreover, since the Fourier coefficients of $\Xg$ with frequency larger than $\gamma^{-2}$ vanish, we obtain from Lemma~\ref{l:unif-Fourier-cut} that
$$
\|\Xg(t)\|_{L^\infty} \le C \gamma^{-2\al} \|\Xg(t)\|_{\Ca}\;.
$$
Recalling that $\dg = \gamma$ and $\bg \le 2$ for $\gamma$ sufficiently small, we obtain \eqref{e:cg-control}. From this, we deduce that the quadratic variation of $\Mm_{\ga,\phi}$ at time $t$ is close to
$$
2 \eps^6 \int_0^t \sum_{x,y,z \in \Lambda_\eps} \phi(x) \phi(y)  \,  \Kg(x-z) \Kg(y-z)   \, ds
$$
up to an error controlled by $\gamma^{1-2\al}$. It is then clear that this tends to $2 t \|\phi\|_{L^2}^2$ as $\gamma$ tends to $0$. We obtain that the martingale
$$
\Mm_{\ga,\phi}^2(t) - \langle \Mm_{\ga,\phi}\rangle_t
$$
is close to $\Gamma_{\Zgn}(t)$ up to an error that vanishes as $\gamma$ tends to $0$. We can now proceed as in the first part to obtain that ${\Gamma}_{\ov{Z}}$ is a local martingale. Recall that we have assumed that $\phi$ is a trigonometric polynomial. But by Remark~\ref{r:density}, this is sufficient to characterize the law of $\ov{Z}$, and the proof is complete.
%
%%%%%%%%%%
\end{proof}
%%%%%%%%%%
%
%
We can now prove the convergence in law of the iterated integrals as well. As above, for the moment we can only prove this ``before the stopping time $\taun$".  

Recall the definitions of the processes $\RR{n}$ and $\ZZ{n}$ in Section~\ref{sec:ContinuousAnalysis}. Furthermore, recall that for $x \in \T^2$, the process  $ s \mapsto R_t(s, x) = \RR{1}(s,x)$, defined for $s <t$, is a continuous martingale with quadratic variation given by 
\begin{align}
\langle R_t(\cdot,x) \rangle_s 
& =  \frac{1}{2} \sum_{\om \in \Z^2} \int_0^s \exp\Ll(2(t-r)\pi^2 |\om|^2\Rr) \, dr\;, \label{e:quadvar0}
\end{align}
and that for $s<t$ and $x \in \T^2$, we have the exact identity
\begin{equation}
\label{e:Wick}
\RR{n}(s,x) = H_n\Ll(R_t(s,x), \langle R_t(\cdot,x) \rangle_s\Rr).
\end{equation}
%

%%%%%%
\begin{theorem}[Convergence of $\ZG{n}$] 
%%%%%
\label{t:converg-lin-Wickbis}
%%%%%%
For every $\nn \in \N$ and $n \in \N$, the processes $(\ZGn{1},$ $\ldots,  \ZGn{n})$  (defined  as in the beginning of this section) converge (jointly) in law to $(\ZZ{1}, \ldots, \ZZ{n})$ with respect to the topology of  $\Dd(\R_+, \Ca)^n$.
\end{theorem}
%%%%%%
%
%%%%%%%%%
\begin{proof}
%%%%%%%%
As explained at the beginning of this section, the results from Sections~\ref{sec:linear} and~\ref{sec:Tightness} remain true if $\RG{n}$ and $\ZG{n}$
are replaced by $\RGn{n}$ and $\ZGn{n}$. In particular, by Proposition~\ref{prop:tight}, the family of processes $\Ll(\ZGn{n}\Rr)_{\ga \in (0,\frac13)}$ is tight with respect to the topology of $\Dd(\R_+,\Ca)$, for every given $n$. 
This implies immediately that for every $n$, the family of vectors $(\ZGn{1}, \ldots, \ZGn{n}), \, \ga \in (0,\frac13)$ is tight with respect to the topology of $\Dd(\R_+,\Ca)^n$. It remains to check the 
convergence of the finite dimensional distributions. 

From now on, we use the shorthand notations $\bZg = (\ZGn{1}, \ldots, \ZGn{n})$ and $\bZ = (\ZZ{1}, \ldots, \ZZ{n})$ to denote the random vectors of interest. It will be useful to also use $\bRg{t} = (\RGn{1}, \ldots, \RGn{n})$ and $\bR = (\RR{1}, \ldots, \RR{n})$. We fix a $K \in \N$ and times $t_1 < t_2 < \ldots < t_K$. Furthermore, let $F\colon (\Ca)^{n \times K} \to \R$ be bounded and uniformly continuous (with respect to the product metric on $(\Ca)^{n \times K}$). We need to show that 
\begin{equation*}
\lim_{\ga \to 0 } \big| \E\, F\big(\bZg(t_1), \ldots, \bZg(t_K)\big)  - \E \,F\big(\bZ(t_1), \ldots, \bZ(t_K)\big) \big| =0\;.  
\end{equation*}
To this end, fix $s_1 < t_1, \,  \ldots, s_K < t_K$ and write
\begin{align}
\big| &\E\, F\big(\bZg(t_1), \ldots, \bZg(t_K)\big)  - \E \,F\big(\bZ(t_1), \ldots, \bZ(t_K)\big) \big| \notag \\
&\leq  \E\,\big|  F\big(\bZg(t_1), \ldots, \bZg(t_K)\big)  -\,F\big(\bRg{t_1}(s_1), \ldots, \bRg{t_K}(s_K)\big) \big| \notag \\
& \qquad + \big| \E\, F\big(\bRg{t_1}(s_1), \ldots, \bRg{t_K}(s_K)\big)  - \E \,F\big(\bR(s_1), \ldots, \bR(s_K)\big) \big|\notag \\
& \qquad +  \E\,\big|  F\big(\bR(s_1), \ldots, \bR(s_K)\big)  -\,F\big(\bZ(t_1), \ldots, \bZ(t_K)\big) \big| \;. \label{e:CiL1}
\end{align}
Recall the estimates \eqref{e:regR0n} and \eqref{e:RgRegularity3} that allow us to control all moments of $\| \bZg(t_i) - \bRg{t_i}(s_i) \big\|_{(\Ca)^n}$ uniformly in~$\ga$. Since $F$ is uniformly continuous, we can thus make the first and last terms on the right-hand side of \eqref{e:CiL1} small uniformly in $\ga$ by choosing $|t_i - s_i|$ small enough. 
Therefore, in order to conclude, it suffices to show the convergence in law of the vector $(\bRg{t_1}(s_1), \ldots, \bRg{t_K}(s_K))$ to $(\bR(s_1), \ldots, \bR(s_K))$ for fixed values of $s_i < t_i$.

We will show the stronger statement that this convergence in law holds with respect to the topology of $(L^\infty)^{n \times K}$. By Proposition~\ref{p:discrete-Wick}, it suffices to show that 
$$
H_\ell(R_{\ga,t_i,\nn}(s_i,x), [R_{\ga,t_i,\nn}(\cdot,x)]_{s_i}) \qquad \ell = 1, \ldots, n, \quad i = 1, \ldots, K 
$$
converges in law to $(\bR(s_1), \ldots, \bR(s_K))$ in $L^\infty$. By \eqref{e:Wick} and Lemma~\ref{le:quartic_var}, it suffices to show the two convergences
\begin{equation}
\label{conv1}
\big( R_{\ga,t_1,\nn}(s_1) , \ldots, R_{\ga,t_K,\nn}(s_K) \big) \xrightarrow[\ga \to 0]{\text{(law)}} \big(R_{t_1}(s_1), \ldots, R_{t_K}(s_K) \big)\; ,
\end{equation}
and 
\begin{equation}
\label{conv2}
\big( \langle R_{\ga,t_1,\nn}(\cdot,\cdot)\rangle_{s_1}, \ldots, \langle R_{\ga,t_K,\nn}(\cdot,\cdot)\rangle_{s_K} \big)  \xrightarrow[\ga \to 0]{\text{(law)}} \big( \langle R_{t_1}(\cdot,\cdot)\rangle_{s_1}, \ldots, \langle R_{t_K}(\cdot,\cdot)\rangle_{s_K} \big) \;, 
\end{equation}
both being understood for the $(L^\infty)^K$ topology. (The joint convergence in law of \eqref{conv1} and \eqref{conv2} would follow immediately since the right-hand side of \eqref{conv2} is deterministic.) As for \eqref{conv1}, note that $\Rgn(s) = \Pg{t-s} \Zgn(s)$. We learn from Corollary~\ref{cor:regPg} and Proposition~\ref{prop:RGBoundOne} that for $i = 1, \ldots, K$, 
$$
\Ll\|\Ll(\Pg{t_i-s_i}- P_{t_i-s_i}\Rr) \Zgn(s_i)\Rr\|_{L^\infty} \xrightarrow[\ga \to 0]{} 0 \;,
$$
almost surely. It thus suffices to check the convergence of $P_{t_i-s_i} \Zg(s_i)$ to $R_{t_i}(s_i) = P_{t_i-s_i} Z(s_i)$. The estimate \eqref{e:heat-semigroup-reg}  ensures that the mapping
$$
\left\{
\begin{array}{lll}
\Cc^{-\al} & \to & L^\infty \\
\msf{Z} & \mapsto & P_{t-s} \msf{Z}
\end{array}
\right.
$$
is continuous, so the convergence in \eqref{conv1} follows from Theorem~\ref{t:converg-lin} and the continuous mapping theorem.

Turning now to \eqref{conv2}, we learn from \eqref{e:QuadrVarR} that for $x \in \Le$,
$$
\langle R_{\ga,t_i,\nn}(\cdot,x)\rangle_{s_i} = 4 \co^2 \int_0^{s_i}    \sum_{z \in \Le} \eg^2 \,  \big( \Pg{t_i-r} \ae \Kg \big)^2(z-x) \, \Cgn(r,z) \, dr,
$$
where $\Cgn(s,z)$ satisfies \eqref{e:cg-control}. It is thus clear that up to an additive error that is bounded by $C \ga^{1 - 2 \al}$ uniformly over $x \in \Le$, the quadratic variation $\langle R_{\ga,t_i,\nn}(\cdot,x)\rangle_{s_i}$ is given by
$$
2 \int_0^{s_i}    \sum_{z \in \Le} \eg^2 \,  \big( \Pg{t_i-r} \ae \Kg \big)^2(z-x) \, dr
$$
(we also used the bound \eqref{e:scaling2} which implies   that $|\co^2 - 1| \leq C\ga^2$ and the calculation \eqref{fromthis} -- \eqref{tothat} which yields a logarithmic bound on the sum).  The latter quantity is equal to
\begin{equation}
\label{e:quadvarlisse}
\frac{1}{2} \int_0^{s_i}   \sum_{\om \in \{-N, \ldots, N\}^2}\exp\left(-2(t_i-r) \ga^{-2}\Ll(1 -\hKg(\om)\Rr) \right)  \,  \big|  \hKg (\om)\big|^2\, dr. 
\end{equation}
By \eqref{e:K2.4} and \eqref{e:K2}, for any fixed $|\om| \leq \ga^{-1}$, we get
\begin{align}
\Big|& \exp\left(-2(t_i-r) \ga^{-2}\Ll(1 -\hKg(\om)\Rr) \right)  \,  \big|  \hKg (\om)\big|^2 - \exp\Ll(-2(t_i-r) \pi^2 |\om|^2\Rr) \Big| \notag\\
&\leq C \exp\Big(- 2(t_i - r) \frac{|\om|^2}{C_1} \Big)\big((t_i - r) \ga|\om|^3  + \ga^2 |\om|^2\big) \label{e:brutal-new1}
\end{align}
From the elementary bound
$$
\int_0^s e^{-(t-r)\ell} \, dr \le s \, {e^{-(t-s)\ell}},
$$
we can conclude that after integrating the bound \eqref{e:brutal-new1} over $[0,s_i]$ and summing over $|\om| \leq \ga^{-1}$, we obtain a quantity that is bounded by $C(s_i,t_i) \ga$. In the same way (using \eqref{e:K2} once more),
we obtain that the sum arising in \eqref{e:quadvarlisse} restricted to indices $|\om| > \ga^{-1}$ is smaller than
$$
\le s_i\sum_{\substack{\om \in \{-N, \ldots, N \}^2\\ |\om|>\ga^{-1}} } \exp\Ll( -2(t_i-s_i)\frac{1}{C_1 \ga^2}  \Rr) \leq C  s_i \ga^{-2} \exp\Ll( -2(t_i-s_i)\frac{1}{C_1 \ga^2}  \Rr) ,
$$
and a similar bound holds true for the limiting quantity $ \exp\Ll(-2(t_i-r) \pi^2 |\om|^2\Rr)$.
Hence we can conclude that uniformly over $x \in \Le$ we get the deterministic bound
\begin{align*}
\big| \langle R_{\ga,t_i,\nn}(\cdot,x)\rangle_{s_i}  - \langle R_{t_i}(\cdot,x)\rangle_{s_i} \big| \leq C(s_i,t_i) \ga^{1 - 2 \al}.
\end{align*}
It only remains to refer to Lemma~\ref{le:LinftyDiscreteCont} to conclude that the convergence \eqref{conv2} also holds for the extensions to arbitrary $x \in \T^2$.

\end{proof}

%%%%%%%%%%%%%%%%%%%%%%%%
\section{Analysis of the non-linear equation}
%%%%%%%%%%%%%%%%%%%%%%%%
\label{sec:Nonlinear}
%%%%%%%%%%%%%%%%%%%%%%%%

In this section, we summarise the calculations of the previous sections and prove our main result, Theorem~\ref{thm:Main}. Throughout this section, $\al>0$ will be assumed to be small enough and fixed. We furthermore fix $\Xn \in \Ca$. We denote by $X \in C(\R_+, \Ca)$ the solution of the renormalised limiting evolution equation with initial datum $\Xn$ as constructed in Section~\ref{sec:ContinuousAnalysis}. Throughout the section, we make use of the bounds on $\Kg$ and $\Pg{t}$ collected in Section~\ref{sec:APPB}.

Recall that the rescaled field  $\Xg(t,x)$ was defined in \eqref{e:defXg}. Recall furthermore that it satisfies the evolution equation \eqref{e:evolution2}, or equivalently, its mild form \eqref{e:evolution3}. For the reader's convenience, we repeat here that  \eqref{e:evolution3} states 
\begin{align}
\Xg(t,\cdot) =&\Pg{t} \Xng + \int_0^t \Pg{t-s}  \Kg \ae   \Big(  - \frac{\beta^3}{3}   \Xg^3  (s,\cdot) +  (\CGG +A) \Xg(s,\cdot)
 \notag \\
& + \Eg(s,\cdot) \Big) \, ds  + \Zg(t, \cdot) \qquad\text{on }  \Le\;.  \label{e:evolution3AA} 
\end{align}
This equation is only valid on the grid points $x \in \Le$ because 
the extension by trigonometric polynomials does not commute with cubing $\Xg$. Therefore, our first step 
consists of deriving an equation that holds for all $x \in \T^2$.

The problem is only caused by large $\om$ in Fourier space. Indeed, as long as $|\om_i| < \frac{N}{3}$, we have
\begin{equation}
\Ex \big[\big( e^{i \pi \om \cdot x}\big)^3 \big]= \big[ \Ex \; e^{i \pi \om \cdot x} \big]^3\;, \label{e:low_freq_nice}
\end{equation}
where as before, we have used $\Ex$ to denote  the extension operator for a function $\Le \to \R$ to a function $\T^2 \to \R$ by trigonometric 
polynomials of degree $\leq N$. In order to control the error caused by the extension, we will use the notation introduced above Lemma~\ref{le:highfreq} and write 
\begin{equation}\label{e:XhigH}
 \Xg^{\mathrm{high}} = \sum_{2^k \geq \frac{\ga^{-2}}{10}} \delta_k \Xg \; \qquad \text{and} \qquad  \Xg^{\mathrm{low}} = \sum_{2^k < \frac{\ga^{-2}}{10}} \delta_k \Xg \;,
\end{equation}
so that $\Xg =  \Xg^{\mathrm{high}} +  \Xg^{\mathrm{low}}$. It is convenient to already collect some error terms in the first lemma, and to this end we discuss	 some notation. For $x \in \Le$, we set $\tCG(s,x) = [ R_{\ga,s}(\cdot,x)]_{ s}$, and we extend this to all $x \in \T^2$ as a trigonometric polynomial. Recall that in \eqref{e:A_value} we had defined 
\begin{align*}
A(s) := A- \frac{s}{2} + \sum_{\om \in \Z^2 \setminus\{ 0\}}  \frac{e^{- 2s\pi^2 |\om|^2}}{ 4\pi^2 |\om|^2}  \;.
\end{align*}
Finally, recall that $Q_{\ga,s}$ was defined in Lemma~\ref{le:quartic_var}.

With this notation at hand, we are ready to state the following result.
%
%%%%%%%%%%%%%%
\begin{lemma}
%%%%%%%%%%%%%%
For every $t \geq 0$, we have on $\T^2$ 
\begin{align}
\Xg(t,\cdot) =&\Pg{t} \Xng + \int_0^t \Pg{t-s}  \Kg \star   \Big(  - \frac{1}{3}   \Xg^3  (s,\cdot) +  (\tCG(s, \cdot) +A(s)) \Xg(s,\cdot)
 \notag \\
& + \ERR{1}(s, \cdot) \Big) \, ds  + \Zg(t, \cdot) \;.  \label{e:evolution3CC} 
\end{align}
For every $ T>0$ and $\ka>0$, there exists $C = C(T,\ka,\al)$ such that the error term satisfies for all $0 \leq s \leq T$ and $0 < \ga < \ga_0$
\begin{align}
 \| \ERR{1}(s, \cdot) \|_{L^\infty(\T^2)} & \leq \; C \; \ga^{-10 \al - \ka}\big(  \| \Xg(s, \cdot) \|_{\Ca}^5 +1)\notag \\
\times & \Big(  \ga^{\frac23} s^{-\frac13}   + \| \Xg^{\mathrm{high}}(s, \cdot)  \|_{L^{\infty}(\T^2)} +\|  Q_{\ga,s}(s,\cdot) \|_{L^\infty(\Le)}    \Big) \;. \label{e:Final_Error1}
\end{align}
%
%%%%%%%%%%%%%%
\end{lemma}
%%%%%%%%%%%%%%
%
%%%%%%%%%%%%%%
\begin{proof}
%%%%%%%%%%%%%%
We get from \eqref{e:evolution3AA} that 
\begin{equation*}
\ERR{1} = \err{1} + \err{2} + \err{3} \;,
\end{equation*}
where 
\begin{align*}
\err{1}(s,\cdot) &= \Eg(s,\cdot) -  \frac{1}{3} (\be^3 -1)    \Ex (\Xg^3(s,\cdot)) \;, \\
\err{2}(s,\cdot)& = - \frac13\Big(  \Ex \; \big( \Xg^3(s,\cdot) \big) - (\Ex \; \Xg(s, \cdot))^3 \Big)\;,\\
\err{3}(s, \cdot)  &=   (\CGG +A -  \tCG(s, \cdot) - A(s)) \;  \Xg(s,\cdot) \;.
\end{align*}
The first term is easily bounded using the definition of $\Eg$ in \eqref{e:error} and the assumption \eqref{e:scalingbeta} and \eqref{e:valueCGG}  on $\be$. On $\Le$ we get
\begin{align}
\big\| \err{1} (s,\cdot) \big\|_{L^\infty(\Le)} &\leq C \ga^2 \|\Xg(s,\cdot) \|_{L^\infty(\Le)}^5 + C \ga^2 \log(\ga^{-1}) \|\Xg(s,\cdot)\|_{L^\infty(\Le)}^3 \;.  \notag
\end{align}
Using Lemma~\ref{le:LinftyDiscreteCont} this bound can easily be extended to $x \in \T^2$ at the expense of loosing an arbitrarily small power of $\ga$. Furthermore, as we have already seen at several places, we can 
use Lemma~\ref{l:unif-Fourier-cut} and the fact that $\hat{X}_{\gamma}(\om) = 0$ for $\om \notin \{ -N, \ldots,N\}^2$ to bound the $L^\infty$ norm by the $\Ca$ norm. In this way we get
for any arbitrary small $\ka >0$
\begin{align}
\big\| \err{1} (s,\cdot) \big\|_{L^\infty(\T^2)} &\leq  C(\ka,\al) \ga^{2-\ka - 10 \al}  \big( \|\Xg(s,\cdot) \|_{\Ca}^5  + 1 \big)\;.  \notag
\end{align}
For the second term, we get on $\T^2$
\begin{align*}
 \Ex& \; \big( \Xg^3\big) - (\Ex \; \Xg)^3  \\
& = \Big[ \Ex \; \big( (\Xg^{\mathrm{low}}   )^3\big) - (\Ex \; \Xg^{\mathrm{low}} )^3 \Big] \\
& \qquad +  \Ex \Big[ \Xg^{\mathrm{high}}\Big( 3 \big( \Xg^{\mathrm{low}}\big)^2 + 3\Xg^{\mathrm{low}}  \Xg^{\mathrm{high}}  + \big( \Xg^{\mathrm{high}} \big)^2   \Big)  \Big] \\
& \qquad- \Big[   \Ex\;  \Xg^{\mathrm{high}} \Big( 3 \big( \Ex \,\Xg^{\mathrm{low}}\big)^2 + 3 \Ex \,\Xg^{\mathrm{low}} \;\; \Ex \, \Xg^{\mathrm{high}}  + \big( \Ex\, \Xg^{\mathrm{high}} \big)^2   \Big)  \Big] \;.
\end{align*}
According to \eqref{e:low_freq_nice}, the term $ \Big[ \Ex \; \big( (\Xg^{\mathrm{low}}   )^3\big) - (\Ex \; \Xg^{\mathrm{low}} )^3 \Big]$ vanishes. 
Using Lemma~\ref{le:LinftyDiscreteCont}, the remaining terms can easily be bounded, so that we obtain  
\begin{align}
\| \err{2}(s, \cdot) \|_{L^\infty(\T^2)} \leq& C(\ka) \ga^{-\ka} \| \Xg^{\mathrm{high}}(s, \cdot)  \|_{L^{\infty}(\T^2)} \notag \\
& \quad \times \Big(  \| \Xg^{\mathrm{low}}(s, \cdot)  \|_{L^{\infty}(\T^2)}^2  +  \| \Xg^{\mathrm{high}}(s, \cdot)  \|_{L^{\infty}(\T^2)}^2 \Big) \;. \label{e:ugly-1}
\end{align}
Finally, we bound the terms appearing in the bracket on the right-hand side of \eqref{e:ugly-1} by
\begin{align*}
\| \Xg^{\mathrm{low}}(s, \cdot)  \|_{L^{\infty}(\T^2)}&  \leq  \sum_{2^k <\frac{\ga^{-2}}{10}  }  \| \delta_k \Xg(s, \cdot)  \|_{L^\infty(\T^2)} \leq\sum_{2^k <\frac{\ga^{-2}}{10}  }  2^{-k \nu}  \| \Xg(s, \cdot) \|_{\Ca}\\
&\leq C(\nu) \ga^{- 2\al} \| \Xg(s, \cdot) \|_{\Ca} \;,\\
\end{align*}
and in the same way
\begin{align*}
\| \Xg^{\mathrm{high}}(s, \cdot)  \|_{L^{\infty}(\T^2)} & \leq  \sum_{2^k \geq \frac{\ga^{-2}}{10}}  \| \delta_k \Xg(s, \cdot)  \|_{L^\infty(\T^2)} \leq C\ga^{- 2\al} \| \Xg(s, \cdot) \|_{\Ca} \;.
\end{align*}
Note that the number of terms appearing in the last sum is uniformly bounded in $\ga$. We will see below that this last bound does not capture the true behaviour of $\Xg^{\mathrm{high}}$, but it is sufficient for this particular part of the estimate \eqref{e:ugly-1}.

In order to bound $\err{3}$, we recall that the precise value of $\CGG$ was defined in \eqref{e:valueCGG}. Then we get for $x \in \T^2$
\begin{align*}
 \CGG&  +A - \tCG(s, x) - A(s) \\
 &= \sum_{\substack{\om \in \{-N, \ldots, N \}^2\\ \om \neq 0 }} \frac{|\hKg(\om)|^2}{4 \ga^{-2}  (1 - \hKg(\om))} - [ R_{\ga,s}(\cdot,x)]_{ s}+ \frac{s}{2} -  \sum_{\substack{\om \in \Z^2  \\ \om \neq 0} }  \frac{\exp(-2s \pi^2|\om|^2 )}{4 \pi^2 |\om|^2} \;.
\end{align*}
Recall from \eqref{e:def:Qg} that for $x\in \Le$,
$[ R_{\ga,r}(\cdot,x)]_{ r}  = \langle R_{\ga,r}(\cdot,x) \rangle_r\ + Q_{\ga,r}(s,x)$.
Furthermore, according to \eqref{e:QuadrVarR} we get for $x\in \Le$
\begin{align}
\langle& R_{\ga,s}(\cdot, x) \rangle_s\notag\\
 &= 4 \co^2 \int_0^s \sum_{z \in \Le} \eps^2  \big(\Pg{s-r} \ae \Kg \big)^2 (x-z) \; \Cg(r,z) \, dr  \notag \\
&=  2  \int_0^s \sum_{z \in \Le} \eps^2  \big(\Pg{s-r} \ae \Kg \big)^2 (x-z) \;  \, dr  +\err{4}(s,x) \; \notag\\
&= \frac{1}{2} \int_0^s  \!\!\! \sum_{\om \in \{-N, \ldots, N\}^2}\exp\left(-\frac{2 r}{\ga^2}\Ll(1 -\hKg(\om)\Rr) \right)  \,  \big|  \hKg (\om)\big|^2\, dr +\err{4}(s,x) \; \notag\\
&=   \frac{s}{2}+  \sum_{\substack{\om \in \{-N, \ldots, N\}^2\\\om \neq 0}}  \frac{\big|  \hKg (\om)\big|^2 }{4 \ga^{-2}\Ll(1 -\hKg(\om)\Rr)} \notag\\
& \qquad \qquad\qquad \times \bigg[1- \exp\left(-\frac{2 s}{\ga^2}\Ll(1 -\hKg(\om)\Rr) \right) \bigg]  +\err{4}(s,x) \; , \notag
\end{align}
where for any $\ka>0$
\begin{align*}
\err{4}(s,x)  = & 4 (\co^2 -1)  \int_0^s \sum_{z \in \Le} \eps^2  \big(\Pg{s-r} \ae \Kg \big)^2 (x-z) \; \Cg(r,z) \, dr \\
&+ 4  \int_0^s \sum_{z \in \Le} \eps^2  \big(\Pg{s-r} \ae \Kg \big)^2 (x-z) \; \Big( \Cg(r,z)- \frac{1}{2} \Big) \, dr \\
\leq & \;C(T,\ka) \; \ga^{2-\ka} + C(T,\ka,\al) \gamma^{1-2\al - \ka} \|\Xg(t)\|_{\Ca} \;,
\end{align*}
using  \eqref{e:scaling2} and \eqref{e:cg-control} as well as the fact that  $ \int_0^s \sum_{z \in \Le} \eps^2  \big(\Pg{s-r} \ae \Kg \big)^2 (x-z) dr$ is bounded by $C(T) \log(\ga^{-1} )\leq C(T,\ka) \ga^{-\ka}$.

Summarising these calculations, we get
\begin{align}
 \CGG&  +A - \tCG(s, \cdot) - A(s)+ Q_{\ga,s}(s,x) + \err{4}(s,x) \notag\\
& =  \sum_{\substack{\om \in \{-N, \ldots, N\}^2 \\ \om \neq 0}}   \frac{\big|  \hKg (\om)\big|^2 }{4 \ga^{-2}\Ll(1 -\hKg(\om)\Rr)}   \exp\left(- \frac{2 s }{\ga^2}\Ll(1 -\hKg(\om)\Rr)\right)  \,\notag\\
&\qquad \qquad  -  \sum_{\substack{\om \in \Z^2  \\ \om \neq 0} }  \frac{\exp(-2s \pi^2|\om|^2 )}{4 \pi^2 |\om|^2} \;.\label{e:ugly0}
\end{align}
We bound the difference between the sums over $|\om| < \ga^{-1}$ term by term. We write
\begin{align}
 \bigg|& \frac{\big|  \hKg (\om)\big|^2 }{4 \ga^{-2}\Ll(1 -\hKg(\om)\Rr)}   \exp\left(- \frac{2 s }{\ga^2}\Ll(1 -\hKg(\om)\Rr)\right) - \frac{\exp(-2s \pi^2|\om|^2 )}{4 \pi^2 |\om|^2}\bigg| 
 \notag\\
 \leq & \exp(-2s \pi^2|\om|^2 )\bigg|  \frac{\big|  \hKg (\om)\big|^2 }{4 \ga^{-2}\Ll(1 -\hKg(\om)\Rr)} - \frac{1}{4 \pi^2 |\om|^2}  \bigg| \notag\\
 & \qquad + \frac{\big| \hKg (\om)\big|^2 }{4 \ga^{-2}\Ll(1 -\hKg(\om)\Rr)} \bigg|  \exp\left(- \frac{2 s }{\ga^2}\Ll(1 -\hKg(\om)\Rr)\right) -\exp(-2s \pi^2|\om|^2 )  \bigg| \;.\label{e:ugly1}
\end{align}
 Using Lemma~\ref{le:Kg0}, we can bound the expression on the right-hand side of \eqref{e:ugly1}  for $|\om| < \ga^{-1}$ by 
\begin{align*}
Ce^{- 2 s \pi^2 |\om|^2}\frac{\ga}{|\om|} + Ce^{-\frac{s|\om|^2}{C_1}}\frac{1}{|\om|^2} s \ga |\om|^3 \leq C \, \ga \; \frac{1}{|\om|^{1+\frac23}}  \; s^{-\frac13} \;.
\end{align*}
So, summing over $0 < |\om| < \ga^{-1}$ gives a contribution which is bounded by $C\ga^{\frac23} s^{-\frac{1}{3}}$. Of course, the choice of the exponents here is essentially arbitrary -- we could have obtained any bound of the type $\ga^{\la} s^{-\frac{\la}{2}}$ for $0 \leq \la < 1$ -- but this choice is convenient.

For $|\om| \geq \ga^{-1}$, the terms in the sum involving $\hKg$ are not a good approximation to the corresponding limiting terms, so we simply bound each sum separately. For the first one, using \eqref{e:K3} and \eqref{e:K2},
\begin{align*}
&\sum_{\substack{\om \in \{-N, \ldots, N\}^2 \\ |\om| \geq \ga^{-1}}}   \frac{\big|  \hKg (\om)\big|^2 }{4 \ga^{-2}\Ll(1 -\hKg(\om)\Rr)}   \exp\left(- \frac{2 s }{\ga^2}\Ll(1 -\hKg(\om)\Rr)\right)  \\
&\qquad \leq C \sum_{\substack{\om \in \{-N, \ldots, N\}^2 \\ |\om| \geq \ga^{-1}}}   \frac{\ga^2 }{|\ga \om|^4}    \exp\left(- \frac{2 s }{C_1 \ga^2} \right)\leq C \ga^{\frac23} s^{-\frac13}\!\!\! \sum_{\substack{\om \in \{-N, \ldots, N\}^2 \\ |\om| \geq \ga^{-1}}}  \frac{1}{|\om|^2}  \leq  C\ga^{\frac23} s^{-\frac13}\;,
\end{align*}
while for the second one,
\begin{align*}
 \sum_{\substack{\om \in \Z^2  \\ \om \geq \ga^{-1}} }  \frac{\exp(-2s \pi^2|\om|^2 )}{4 \pi^2 |\om|^2}\leq C(\ka)  s^{-\frac13} \ga^{\frac23-\ka} \;.
\end{align*}
Summarising our calculations, we get the desired estimate on $\ERR{1}$.
%%%%%%%%%%%%%
\end{proof}
%%%%%%%%%%%%%

For $\nn \geq1$, recall the definition of the stopping time $\taun$ in \eqref{e:deftaug} and the definition of the processes $\Zgn$ (see the discussion following  \eqref{e:raten}). It will be useful to introduce an auxiliary process $\tXgn(t,x)$, $t \geq 0$, $x \in \T^2$ as the solution of  \eqref{e:evolution3CC} with $\Zg$ replaced by $\Zgn$ and $\ERR{1}$ replaced by 
\begin{equation}\label{e:err11}
\ERR{1}_{\nn}(s) :=
\begin{cases}
 \ERR{1}(s ) \qquad &\text{if } s \leq \taun \;,\\
 0 \qquad &\text{else}.
\end{cases}
\end{equation}
 The existence and uniqueness of this process for fixed $\gamma>0$ follows by an elementary ODE argument which we omit.  Of course, $\tXgn(t, \cdot)$ and $\Xg(t, \cdot)$ coincide on the event $\{ t < \taun\}$. 
Note that according to \eqref{e:Final_Error1}, by choosing $\al>0$ and $\ka>0$ small enough, for $s \leq \taun$ we have the deterministic bound 
\begin{align}
\|& \ERR{1}_{\nn}(s,\cdot) \|_{L^\infty(\T^2)}\notag \\
& \leq C(T,\nn)  \Big(  \ga^{\frac12} s^{-\frac13}   + \ga^{-\frac16} \| \Xg^{\mathrm{high}}(s, \cdot)  \|_{L^{\infty}(\T^2)} + \ga^{-\frac16}\|  Q_{\ga,s}(s,\cdot) \|_{L^\infty(\Le)}    \Big) \;. \label{e:error0control}
\end{align}

We will require yet another process that approximates $\Xg$. For any $0 \leq t \leq T$, we define
\begin{equation*}
\oX(t, \cdot) = P_{t}\Xn +  \Zgn(t, \cdot)   + \Ss_T(\Zgn, \ZGn{2},\ZGn{3})(t, \cdot)\;,
\end{equation*}
where $\Ss_T$ is the solution operator to the continuous problem, defined in Theorem~\ref{thm:ContSolution}. Recall that the definition of the operator $\Ss_T$ involves the choice of initial datum $\Xn$, which we keep fixed. The following lemma is an immediate consequence of the main results of Sections~\ref{sec:ContinuousAnalysis} and~\ref{sec:conv-lin}. (The initial datum of the limiting evolution $X$ is $\Xn \in \Ca$.)
%
%
%
%%%%%%%%%%%%%%
\begin{lemma}
%%%%%%%%%%%%%%
\label{le:FirstPartOfMainThm}
%%%%%%%%%%%%%%
For any $T >0$, let $F \colon \Dd([0,T], \Ca) \to \R$ be uniformly continuous and bounded. We have
\begin{equation*}
\lim_{\gamma \to 0 }\Big| \E \big[ F\big(   \oX \big) \big] - \E \big[ F(X ) \big]  \Big| =0\;.
\end{equation*}
\end{lemma} 
%%%%%%%%%%%%%%%%%
%
%
%%%%%%%%%%%%%%%%
\begin{proof}
%%%%%%%%%%%%%%%%
For $ t \in [0,T]$, we can write 
\begin{equation*}
X(t, \cdot) =  P_{t}\Xn(\cdot) +  Z(t, \cdot)   + \Ss_T(Z, Z^{\colon 2 \colon},Z^{\colon 3 \colon})(t, \cdot)\;.
\end{equation*}
In Theorem~\ref{t:converg-lin-Wickbis}, it was shown that the triple $(\Zgn, \ZGn{2}, \ZGn{3})$ convergences in law with respect to the topology of $\Dd([0,T],\Cc^{-\al})^{3}$ to the triple $(Z, Z^{\colon 2 \colon},Z^{\colon 3 \colon})$. As the limit is continuous, we can conclude that the convergence also holds with respect to the metric of $L^{\infty}([0,T], \Ca)^3$.

On the other hand, Theorem~\ref{thm:ContSolution} ensures that the operator  $\Ss_T$ is uniformly continuous on bounded sets from $  L^{\infty}\big([0,T], \Ca \big)^3$ to, say, $\Cc([0,T], \Cc^{1}(\T^2) )$.  
 Note that the evolution $Y(t) = P_t \Xn$ is in $\Cc([0,T] , \Ca)$ by properties of the heat semigroup $P_t$. In particular, the mapping that sends $(Z, Z^{\colon 2 \colon},Z^{\colon 3 \colon})$ to 
 \begin{equation*}
 P_{t}\Xn +  Z(t, \cdot)   + \Ss_T(Z, Z^{\colon 2 \colon},Z^{\colon 3 \colon})(t, \cdot)
 \end{equation*}
 is continuous from $L^\infty([0,T],\Cc^{-\al})^{3}$ to $\Dd([0,T],\Cc^{-\al})$. This implies the desired result.  
%
%%%%%%%%%%%%%%%
\end{proof}
%%%%%%%%%%%%%%%
%
%
It remains to bound the difference between $\oX$ and $\Xg$ (or $\tXgn$). Note that these two processes are naturally defined on the same probability space, so that we can derive almost sure error bounds. For most of the remainder of this section we collect bounds  to show that for every fixed $\nn \geq 1$, $T >0$ and bounded uniformly continuous function $F \colon \Dd([0,T], \Ca) \to \R$, we have
\begin{equation}\label{e:Aim1}
\limsup_{\ga \to 0}\E\big| F(\oX) - F(\tXgn) \big|  =0\;.
\end{equation}
Once \eqref{e:Aim1} is established, it can be combined with Lemma~\ref{le:FirstPartOfMainThm}, implying the convergence in law of $\tXgn$ to $X$ for every fixed $\nn \geq1$. The stopping time $\taun$ can then be removed by a 
soft argument that will be explained in the proof of our main result, Theorem~\ref{thm:Main}, at the end of the section.

We start by treating the initial datum. We  set $\Yg(t, \cdot) = \Pg{t} \Xng$ on $ \T^2$. Lemma~\ref{le:semi-group-regularity} and the uniform bound on $\| \Xng\|_{\Ca}$ imply that for $\ka >0$ and $\be >0$, we have
\begin{equation}\label{e:Ygbo1}
\| \Yg(t, \cdot) \|_{\Cc^{\be-\ka}}  \;t^{\al + \beta} \leq C(T,\be, \ka) \; ,
\end{equation}
uniformly in $\ga$ and in $t \in [0,T]$. 
This is slightly worse than the bound available for the continuous equation, where we have uniform control on $\| Y(t, \cdot) \|_{\Cc^\beta} t^{\frac{\al + \beta}{2}}$ for all values of $\nu,\be >0$. The bound is weaker due to the bad behaviour of $\Pg{t}$ on high Fourier modes observed in Lemma~\ref{le:semi-group-regularity}, but sufficient for our needs. 

We set
\begin{align}
\tvg(t,x)  &:= \tXgn(t,x)  -  \Zgn(t,x) - \Yg(t,x) \qquad t \geq 0, \; x \in \T^2 \;, \notag\\
\bvg(t, x) &:= \oX(t, x)  -  \Zgn(t, x) -  Y(t, x)   \qquad t \geq 0, \; x \in \T^2    \;. \label{e:Defvgm}
\end{align}

So the desired bound \eqref{e:Aim1} follows as soon as we have established bounds on $\sup_{0 \leq t \leq T} \|  Y(t,\cdot) -\Yg(t,\cdot)  \|_{\Ca}$ 
as well as $\sup_{0 \leq t \leq T} \| \bvg(t, \cdot) - \tvg (t, \cdot) \|_{\Cc^{\frac12}}$  for fixed $T$ and $\nn$. 

The first of these bounds is established in the following Lemma.
%
%%%%%%%%%%%%%%%%%%%%%
\begin{lemma}
%%%%%%%%%%%%%%%%%%%%%
\label{le:ic}
%%%%%%%%%%%%%%%%%%%%%
For every $T >0$ we have
\begin{equation}\label{e:Aim1A}
\sup_{0 \leq t \leq T} \|  Y(t,\cdot) -\Yg(t,\cdot)  \|_{\Ca} \to 0
\end{equation} 
\end{lemma}
%%%%%%%%%%%%%%%%%%%%%%
%
%%%%%%%%%%%%%%%%%%%%%%
\begin{proof}
For any $t \geq 0$ we can write
\begin{align}
\| Y(t, \cdot) - \Yg(t, \cdot)  \|_{\Ca} \leq   \| P_t (\Xn - \Xng) \|_{\Ca} + \| (P_t -\Pg{t} ) \Xng \|_{\Ca} \;. %\label{e:Finale1}
\end{align}
The operator $P_t$ is bounded on $\Ca$ uniformly in $t$, so that  the first term is bounded by $C\|\Xn - \Xng \|_{\Ca} \to 0$. For the second term, we use the assumption that states that the $\Xng$ are uniformly bounded in $\Cc^{-\al + \ka}$ for some small $\ka>0$, as well as the fact that $ \|\Pg{t} - P_t\|_{\Cc^{-\al + \ka} \to \Ca } \leq C(T, \nu, \ka) \ga^{\frac{\ka}{2}}$ according to \eqref{e:reg_Cor}.  So \eqref{e:Aim1A} is established. 
\end{proof}
%%%%%%%%%%%%%%%%%%%%%%
%
%
%
%
 
We now turn to the terms $\tvg$ and $\bvg$. As in Section~\ref{sec:ContinuousAnalysis}, we have to include the initial conditions in the renormalised powers of $\Zgn$. We define  
\begin{align}
 \ttZ(t, \cdot) &=\Yg(t, \cdot) + \Zgn(t, \cdot) \; \notag,\\
 \ttZz(t, \cdot) \colon &=  \ZGn{2}(t, \cdot)  + 2\Yg(t, \cdot) \Zgn(t, \cdot)  +  \Yg(t,\cdot)^2\;, 
 \label{e:Zg_with_ic}\\
 \notag
 \ttZd(t, \cdot) \colon &=  \ZGn{3}(t, \cdot)  + 3\Yg(t, \cdot)  \ZGn{2}(t, \cdot)  +  3\Yg(t, \cdot)^2 \Zgn(t, \cdot)  +  \Yg(t,\cdot)^3\;. 
\end{align}
Furthermore, we define $\bZZg$, $\bZZG{2}$ and $\bZZG{3} $ in exactly the same way as we defined $\ttZ, \, \ttZz, \ttZd$ in \eqref{e:Zg_with_ic}, with the only difference that all occurrences of $\Yg(t, \cdot)$ are replaced by $Y(t,\cdot) = P_t \Xn$.

With these notations in place, we can observe that $\bvg =\Ss_T(\Zgn, \ZGn{2},\ZGn{3}) $ satisfies
\begin{align}
 \bvg(t, \cdot)  =
&- \int_0^t P_{t-s}   \bPsg( s) \, ds\; , \label{e:evolution3aa}
\end{align}
where we have set 
\begin{align}
\bPsg(s) :=& \frac13 \Big( \bvg^3(s) + 3 \bZZg(s) \, \bvg^2(s) + 3  \bZZG{2}(s) \, \bvg(s) +  \bZZG{3}(s) \Big) \notag \\
& - A(s)\,\big(\bvg(s) + \bZZG{1}(s) \big) \;. \label{e:Def_Psi_bar}
\end{align}
In the next lemma, we establish a similar expression for $\tvg$. To this end, recall the definition of $\EG{n}$ in \eqref{e:def:EGn}. We denote by $E^{\colon n \colon}_{\ga,t,\nn}$ the corresponding quantity with $\Rg$, $\RG{n}$ replaced by $\Rgn$, $\RGn{n}$. As discussed above, Proposition~\ref{p:discrete-Wick} holds as well (in fact, with the same constant) with $\EG{n}$ replaced by $E^{\colon n \colon}_{\ga,t,\nn}$.

%%%%%%%%%%%%%%%%%%
\begin{lemma}
%%%%%%%%%%%%%%%%%%
\label{le:Collection_of_Errors}
%%%%%%%%%%%%%%%%
For all $  \ga>0 $ and $t \geq 0$ we have on $\T^2$
\begin{align}
\tvg(t,\cdot) =&  - \int_0^t \Pg{t-s}  \Kg \star   \big( \Psg(s)+ \ERR{1}_{\nn} +\ERR{2}_{\nn}(s,\cdot) \big) \, ds \;,   \label{e:evolution3b} 
\end{align}
where
\begin{align}
\Psg(s) :=&\frac13 \Big( \tvg^3(s) + 3 \ttZ(s) \, \tvg^2(s) + 3  \ttZz(s) \, \tvg(s) +  \ttZd(s) \Big) \notag \\
& - A(s)\,\big(\tvg(s) + \ttZ(s) \big) \;. \label{e:def_Psi_n}
\end{align}
The error term $\ERR{1}_{\nn} $ was bounded in \eqref{e:error0control}, while for $\ERR{2}_{\nn}$, we have  for $0 \leq s \leq T$ and for $0 < \ga < \ga_0$
\begin{align}
\| &\ERR{2}_{\nn}(s, \cdot) \|_{L^\infty(\T^2)} \notag\\
& \qquad  \leq C(T, \al, \ka)  \big( \| E_{\ga,s,\nn}^{\colon 3 \colon }\|_{L^\infty(\T^2)}  + s^{-\al-\ka} \| E_{\ga,s,\nn}^{\colon 2 \colon }\|_{L^\infty(\T^2)}\big)\;. \label{e:e2Bound}
\end{align}
\end{lemma}
%%%%%%%%%%%%%%
%
%
%%%%%%%%%%%%%%%
\begin{proof}
Using \eqref{e:evolution3CC}, we see that $\tvg$ satisfies for any $t \geq 0$ on $ \T^2$
\begin{align*}
\tvg(t,\cdot) &= 
 \int_0^t \Pg{t-s}  \Kg \ae   \Big(  
- \frac{1}{3}  ( \tvg + \ttZ)^3(s,\cdot) \\
& \quad + \big(\tCG(s, \cdot) +  A(s)\big)
 (\tvg + \ttZ) (s,\cdot)\notag +\ERR{1}_{\nn}(s,\cdot) \Big)  \, ds \;  .
\end{align*}
The term $\ERR{1}_{\nn}(s,x) $ was already controlled in \eqref{e:error0control}. For the rest, we write
\begin{align*}
\frac13& ( \tvg + \ttZ)^3 - \tCG(\cdot,\cdot) \big( \tvg + \ttZ \big)  \\
&= \frac{1}{3} \tvg^3 + \tvg^2 \ttZ + \tvg \ttZz + \frac13 \ttZd  + \err{1} + \err{2}\;,
\end{align*}
where
\begin{align*}
\err{1}(s,\cdot) &=  \tvg(s,\cdot) \big( \ttZ^2(s,\cdot) - \tCG(s,\cdot)  -\ttZz(s,\cdot))\;, \\
\err{2} (s,\cdot)&=  \frac13 \big( \ttZ^3(s,\cdot) - 3 \tCG(s,\cdot) \ttZ(s,\cdot)  -\ttZd(s,\cdot)) \;.
\end{align*}
To control $\err{1}$, we write using \eqref{e:Zg_with_ic}
\begin{align*}
\ttZ^2& (s, \cdot)- \tCG(s, \cdot)  -\ttZz(s, \cdot)\\
 & = \big( \Zgn^2 + 2 \Yg \, \Zgn +\Yg^2    \big)(s, \cdot) - \tCG(s, \cdot) - \big(  \ZGn{2}  + 2\Yg \Zgn  +  \Yg^2 \big)(s, \cdot) \\
& = \Zgn^2(s, \cdot)  - \tCG(s, \cdot) - \ZGn{2}(s, \cdot) \\
& = E_{\ga,s,\nn}^{\colon 2 \colon }(s, \cdot) \;.
\end{align*}
For $\err{2}$, we write using \eqref{e:Zg_with_ic} once more (and dropping the arguments to improve readability)
\begin{align*}
 \big( & \ttZ^3 - 3 \tCG \ttZ  -\ttZd \big) \\
 & = \big(\Zgn^3  + 3\Yg \Zgn^2    + 3\Yg^2 \Zgn+  \Yg^3   \big) - 3 \tCG (\Zgn +\Yg)  \\
 & \qquad \qquad \qquad -  \big( \ZGn{3}  + 3\Yg  \ZGn{2}  +  3\Yg^2 \Zgn  +  \Yg^3 \big) \\
 & = \big( \Zgn^3  - 3 \tCG \Zgn  - \ZGn{3}  \big) + 3 \Yg \big( \Zgn^2   - \tCG - \ZGn{2} \big) \\
 & = E_{\ga,s,\nn}^{\colon 3 \colon } + 3 \Yg E_{\ga,s,\nn}^{\colon 2 \colon } \;. 
\end{align*}
By assumption, $\Xng$ is bounded in $\Ca$.The bound \eqref{e:Ygbo1}  thus ensures that for any $\ka>0$, we have a uniform bound on $s^{\al + \ka} \| \Yg\|_{L^\infty(\T^2)}$, from which the desired bound follows. 
%
%%%%%%%%%%%%%%
\end{proof}
%%%%%%%%%%%%%%

In the next lemma, we combine the expressions \eqref{e:evolution3aa} and \eqref{e:evolution3b} to derive a bound on $\sup_{0 \leq t \leq T} \| \bvg(t, \cdot) - \tvg (t, \cdot) \|_{\Cc^{\frac12}}$.
%
%
%
%%%%%%%%%%%%%%%
\begin{lemma}
%%%%%%%%%%%%%%%
Let $\ka>0$ be the constant appearing in the boundedness assumption in Theorem~\ref{thm:Main} and let $\al >0$ be small enough. For every $0 \leq t \leq T$ and $0 < \ga < \ga_0$, we have
\begin{align}
\|  \bvg(t, \cdot)-  \tvg(t, \cdot)  \big\|_{\Cc^{\frac12}} \leq  & \overline{C}_1  \int_0^t  (t-s)^{-\frac{1}{3}} s^{-\frac{1}{6} }\|  \bvg(s,\cdot) - \tvg(s,\cdot) \|_{\Cc^{\frac12}}   \, ds \notag\\
&+ \overline{C}_1 (\ga^{\frac{\ka}{2}}  + \| \Xng - \Xn \|_{\Ca}) + \ERR{3}(t)  \label{e:Gronwall1} \;,
\end{align}
where the constant $\overline{C}_1$ depends on $\nu,\ka,T$ $\| \Xn \|_{\Cc^{- \al + \ka}}$, $\| \Xng \|_{\Cc^{-\al + \ka}}$ as well  as the random quantities $\sup_{0 \leq s \leq T} \| \bvg(s, \cdot) \|_{\Cc^{\frac12}}$,   $\sup_{0 \leq s \leq T}\|\tvg(s, \cdot) \|_{\Cc^{\frac12}}$,  $\sup_{0 \leq s \leq T}\| \Zgn(s, \cdot) \|_{\Ca}$,  $\sup_{0 \leq s \leq T} \| \ZGn{2}(s, \cdot) \|_{\Ca}$, $\sup_{0 \leq s \leq T} \| \ZGn{3}(s, \cdot) \|_{\Ca}$.

The error term $\ERR{3}$ satisfies for every $T \geq 0$, $p \geq 2$ and $0 <\la< \frac{5}{6}$
\begin{align}\label{e:very_good_bound_A}
\E \sup_{0 \leq t \leq T}  \big| \ERR{3}(t) \big|^p \leq \overline{C}_2 \ga^{\la p} \;,
\end{align}
for a constant $\overline{C}_2= \overline{C}_2(p,T,\la)$.
\end{lemma}
%%%%%%%%%%%%%%%
%
%
%
%%%%%%%%%%%%%%%
\begin{proof}
%%%%%%%%%%%%%%%

For any $ t \geq 0$ and $\gamma >0$, we get combining  \eqref{e:evolution3aa} and \eqref{e:evolution3b}
\begin{align}
 \bvg(t, \cdot)-  \tvg(t, \cdot)  =& -\int_0^t \big(  P_{t-s} - \Pg{t-s} \star \Kg   \big)  \,\bPsg\big(  s\big) \,ds \notag  \\
&- \int_0^t \Pg{t-s}  \star \Kg \star\, \big( \bPsg( s) -  \Psg(s)    \big) \, ds\; \notag\\
&+ \int_0^t \Pg{t-s}\star  \,\Kg \star \big(\ERR{1}_{\nn}(s, \cdot) +\ERR{2}_{\nn}(s,\cdot) \big)   \, ds\;, \label{e:bound_diff1}
\end{align}
where $\bPsg( s) $ was defined in \eqref{e:Def_Psi_bar} and   $\Psg(s)$ was defined in  \eqref{e:def_Psi_n}.  
Using the multiplicative inequality, Lemma~\ref{le:Besov-multiplicative}, we can bound for every $s \geq 0$ 
\begin{align}
\| \bPsg( s)& \|_{\Ca} \notag \\
&\leq  \|\bvg(s, \cdot) \|_{\Cc^\frac12}^3 +C(\nu)  \| \bZZg(s, \cdot)\|_{\Ca}  \|\bvg(s, \cdot) \|_{\Cc^{\frac12}}^2 
\notag \\
 &+ C(\nu) \| \bZZG{2}(s, \cdot)\|_{\Ca}  \|\bvg(s, \cdot) \|_{\Cc^{\frac12}} 
  + \| \bZZG{3}(s, \cdot)\|_{\Ca}  \notag\\ 
 & + A(s) \big( \|  \bvg(s, \cdot) \|_{\Cc^{\frac12}} + \| \bZZg(s,\cdot) \|_{\Ca}\big)\;,\label{e:FinaleA1}
\end{align}
where we have chosen $\al>0$ small enough to assure  $\al < \frac12$. Recalling the definition of $\bZZg$, $\bZZG{2}$ and $\bZZG{3}$ just below \eqref{e:Zg_with_ic} we get for $s \geq 0$ 
\begin{align}
 \|  \bZZg(s, \cdot) \|_{\Ca}& \leq \; \|  \Zgn(s, \cdot) \|_{\Ca} + \|Y(s, \cdot) \|_{\Ca} \leq \|  \Zgn(s, \cdot) \|_{\Ca} +C \|\Xn \|_{\Ca} \notag
 \end{align}
 and for any  $\bar{\ka}>0$
 \begin{align}
\|   \bZZG{2}  \|_{\Cc^{-\al}} &\leq  \| \ZGn{2} \|_{\Ca} +  C(\al,\bar{\ka}) \big( \|  \Zgn \|_{\Ca}  \|Y \|_{\Cc^{\al + \bar{\ka}}} + \|Y \|_{\Cc^{\al + \bar{\ka}}} \|Y \|_{\Ca}\big) \notag\\
\|   \bZZG{3}  \|_{\Cc^{-\al}} &\leq  \| \ZGn{3} \|_{\Ca} +  C(\al,\bar{\ka})\Big(  \|  \ZGn{2} \|_{\Ca}  \|Y \|_{\Cc^{\al + \bar{\ka}}} 
+  \|  \Zgn \|_{\Ca}  \|Y \|_{\Cc^{\al + \bar{\ka}}}^2\notag\\
 &  +  \|Y \|_{\Cc^{\al + \bar{\ka}}}^2 \|Y \|_{\Ca} \Big) \;, \notag
\end{align}
where we have omitted the arguments $(s, \cdot)$ for all functions in the inequality. The regularity bound \eqref{e:heat-semigroup-reg} for the heat semigroup $P_t$  implies that for any $\la \geq -\al$
\begin{equation}\label{e:FinaleAB}
\| Y(t) \|_{\Cc^{\la}} \leq C(\la) t^{-\frac{\la +\al}{2}} \| \Xn \|_{\Ca}\;.
\end{equation}
Furthermore, the definition \eqref{e:A_value} of $A(s)$ shows that $|A(s)| \leq C\log(s^{-1})$. Combining these bounds, we get the following (brutal) bound uniformly over $0 \leq s \leq T$
\begin{align}
 \big\| \bPsg( s)  \big\|_{\Ca} \leq& \, C(\al,T) \, \big( \|\bvg(s, \cdot) \|_{\Cc^\frac12}^3 +1    \big) \;  s^{-\frac{1}{4}}  \big(  \| \Xng \|_{\Ca}^3 +1 \big) \notag\\
& \qquad  \big( \| \Zgn(s) \|_{\Ca}  + \| \ZGn{2}(s) \|_{\Ca}  + \| \ZGn{3}(s) \|_{\Ca}\big)  \;,\label{e:Finale1C}
\end{align}
where we have assumed that $\al$ and $\bar{\ka}$ are small enough to ensure that the exponent on $s$ in the term $ \|Y(s) \|_{\Cc^{\al + \bar{\ka}}}^2 \leq  C(T, \al,\bar{\ka}) s^{-(2\al +\bar{\ka})} \| \Xn \|_{\Cc^{\al }}^2$ satisfies $2\al  +\bar{\ka} \leq \frac{1}{4}$.

Combining Lemma~\ref{le:semi-group-regularity} and Corollary~\ref{cor:regPg}, we observe that for any $\la < \frac{3}{4} - \frac{\al}{2}$ and  $\bar{\ka} >0$,  there exists a constant $C= C(T,\la,\bar{\ka})$ such that uniformly over $0 \leq t \leq T$,
\begin{align}\label{e:Finale34}
\big\| \big(P_{t} -\Pg{t} \star \Kg   \big) \big\|_{\Ca  \to \Cc^{\frac12}} \leq C \ga^\la t^{-\la - \frac{1}{4} -\frac{\al}{2} - \bar{\ka} }\;.
\end{align}
Indeed, a combination of \eqref{e:Regular1}, \eqref{e:Regular3} and \eqref{e:Regular6} in conjunction with the standard regularity estimate for the heat semigroup \eqref{e:heat-semigroup-reg} yields that  for any $0 \leq \be\leq2$ and $\bar{\ka} >0$, we have uniformly over $0 \leq t \leq T$
\begin{equation*}
\big\| \big(P_{t} -\Pg{t} \star \Kg   \big) \big\|_{\Ca  \to \Cc^{-\al + \be - \bar{\ka}}} \leq C(\be, \bar{\ka},T)\, t^{- \frac{\be}{2}}\;.
\end{equation*}
Interpolating between this bound and \eqref{e:reg_Cor2} yields \eqref{e:Finale34}.

For any $\bar{\ka} >0$ and $\la >0$ small enough, we can thus bound the first term on the right-hand side of  \eqref{e:bound_diff1} uniformly over $0 \leq s \leq T$  
\begin{align}
\Big\|&  \int_0^t \big(P_{t-s} -\Pg{t-s} \star \Kg   \big)  \,\bPsg\big(  s\big) \,ds \Big\|_{\Cc^{\frac12}} \notag\\
&\leq  C(\bar{\ka},\la, T) \int_0^t (t-s)^{- \frac{\al}{2} -\frac14 -\bar{\ka}} (t-s)^{-\frac{\la}{2} - \bar{\ka}}  \, \ga^{\la} \big\| \bPsg( s)  \big\|_{\Ca} \notag \\
&\leq C t^{ \frac12 -\frac{\al + \la}{2} -2\bar{\ka}} \ga^{\la} \;, \label{e:diff_bound2}
\end{align}
where the constant $C$  on the right-hand side depends on $\bar{\ka}$, $\la$,  $T$,   $ \| \Xg \|_{\Ca}$ but also on the random quantities $\sup_{0 \leq s \leq T} \|\bvg(s, \cdot) \|_{\Cc^\frac12}$, $\sup_{0 \leq s \leq T} \| \Zgn(s, \cdot)) \|_{\Ca}$, $\sup_{0 \leq s \leq T} \| \ZGn{2}(s,\cdot) \|_{\Ca}$  and $\sup_{0 \leq s \leq T}\| \ZGn{3}(s,\cdot) \|_{\Ca}  $. In the last inequality in \eqref{e:diff_bound2} we have chosen $\nu,\la,\bar{\ka}$ small enough to ensure that the quantity on the right-hand side of the integral in the second line is integrable at as $s \to t$. By choosing $\nu$ and $\bar{\ka}$ small enough we can still choose $\la = \frac12$. Then the right-hand side of \eqref{e:diff_bound2} is bounded by $C \ga^{\frac12}$ where $C$ depends on $T, \|\Xg\|_{\Ca}$ as well as all the random quantities listed above.

As we will now see, the second term on the right-hand side of \eqref{e:bound_diff1} can be treated in a similar way. Using again the multiplicative inequality, Lemma~\ref{le:Besov-multiplicative}, as well as the definition of $\ttZ$, $\ttZz$, $\ttZd$, $\bZZg$, $\bZZG{2}$ and $\bZZG{3}$ in \eqref{e:Zg_with_ic} and below, we get 
for all $s \geq 0$  and any $\bar{\ka}>0$ that
\begin{align*}
\big\| &  \bPsg( s)  - \Psg(s) \big\|_{\Ca}\\
\leq &C(\nu, \bar{\ka},T) \Big( \|  \bvg(s, \cdot) - \tvg(s,\cdot) \|_{\Cc^{\frac12}} + \| Y(s, \cdot) - \Yg(s,\cdot) \|_{\Ca}\Big) \\
&\times (1 + A(s))  \big( \|\bvg(s, \cdot) \|_{\Cc^{\frac12}}^2 + \|\tvg(s, \cdot) \|_{\Cc^{\frac12}}^2+\| \Zgn(s, \cdot) \|_{\Ca}^2  + \| \ZGn{2}(s, \cdot) \|_{\Ca} \\
& \qquad + \| Y(s, \cdot)\|_{\Cc^{\al + \bar{\ka}}}^2 +\| \Yg(s,\cdot) \|_{\Cc^{\al + \bar{\ka}}}^2  +1    \big).
\end{align*}
Arguing as in the proof of Lemma~\ref{le:ic}, we see that 
\begin{equation*}
\| Y(s, \cdot) - \Yg(s,\cdot) \|_{\Ca} \leq C \| \Xn - \Xng \|_{\Ca} + C(\al, \ka,T) \ga^{\frac{\ka}{2}}\,,
\end{equation*}
where $\ka>0$ is the constant appearing in the boundedness condition in the statement of our main result, Theorem~\ref{thm:Main}. Furthermore, using \eqref{e:Ygbo1} (or  Lemma~\ref{le:semi-group-regularity}),
we get for any $\bar{\ka} >0$
\begin{align*}
\| Y(s, \cdot)\|_{\Cc^{\al + \bar{\ka}}}^2 &\leq C(\nu, \bar{\ka}) s^{- 2 \nu - \bar{\ka}} \| \Xn \|_{\Ca} \;, \\
\| \Yg(s,\cdot) \|_{\Cc^{\al + \bar{\ka}}}^2 & \leq C(\nu, \bar{\ka},T) s^{-4 \nu - 3 \bar{\ka}} \| \Xng \|_{\Ca} \;.
\end{align*}
(As stated above, the bounds for the approximated heat semigroup are weaker than the bounds in the continuous equation. )

This estimate, together with \eqref{e:Regular6}, gives the following bound on the second term on the right-hand side of \eqref{e:bound_diff1}, for any $\bar{\ka}>0$:
\begin{align}
\Big\| \int_0^t & \Pg{t-s}   \Kg \star  \, \big(  \bPsg(s) -\Psg(s)  \, \big) \, ds \Big\|_{\Cc^{\frac12}} \notag \\
& \leq  C(\al,\bar{\ka},T) \int_0^t  (t-s)^{-\frac14 -\frac{\al}{2}-\bar{\ka}} \big\|  \bPsg(s) -\Psg(s)  \, \big\|_{\Ca}  \, ds \notag \\
& \leq  C  \int_0^t  (t-s)^{-\frac14 -\frac{\al}{2}-\bar{\ka}} s^{-4 \nu - 3 \bar{\ka} } \notag\\
& \qquad \qquad \qquad \big(\|  \bvg(s) - \tvg(s) \|_{\Cc^{\frac12}}  + \|  \Xn  -\Xng \|_{\Ca}  +  \ga^{\frac{\ka}{2}}\big)\, ds.\label{e:FinaleB17}
\end{align}
for a constant $C$ that depends on $\al,\bar{\ka},\ka,T$, $\| \Xn \|_{\Cc^{- \al + \ka}}$, $\| \Xng \|_{\Cc^{-\al + \ka}}$ as well  as the random quantities $\sup_{0 \leq s \leq T} \| \bvg(s, \cdot) \|_{\Cc^{\frac12}}$,   $\sup_{0 \leq s \leq T}\|\tvg(s, \cdot) \|_{\Cc^{\frac12}}$,  $\sup_{0 \leq s \leq T}\| \Zgn(s, \cdot) \|_{\Ca}$,  $\sup_{0 \leq s \leq T} \| \ZGn{2}(s, \cdot) \|_{\Ca}$. Here we have absorbed the logarithmic divergence of $A(s)$ as $s \to 0$ into the term $s^{-4 \nu -3 \bar{\ka}}$ by choosing a slightly larger $\bar{\ka}$. If $\al, \bar{\ka} >0$ are small enough, the integral in the last line of \eqref{e:FinaleB17} can be bounded by 
\begin{equation*}
C    \int_0^t  (t-s)^{-\frac13} s^{-\frac16 } \|  \bvg(s) - \tvg(s) \|_{\Cc^{\frac12}}\, ds  + C\|  \Xng  -\Xn \|_{\Ca}  + C \ga^{\frac{\ka}{2}} \;.
\end{equation*}

For the last term on the right-hand side of \eqref{e:bound_diff1}, using \eqref{e:Regular6} once more, we get for any $\bar{\ka}>0$
\begin{align*}
\big\| & \int_0^t \Pg{t-s}  \,\Kg \star \big(\ERR{1}_{\nn}(s, \cdot) +\ERR{2}_{\nn}(s,\cdot) \big)   \, \big) \, ds \big\|_{\Cc^{\frac12}} \\
 &\leq C(\bar{\ka}, T)  \int_0^t (t-s)^{-\frac{1}{4} - \bar{\ka}} \big\| \ERR{1}_{\nn}(s, \cdot) +\ERR{2}_{\nn}(s,\cdot) \big\|_{L^\infty} ds \;.
\end{align*}
Recall that according to \eqref{e:err11}, for $s > \taun$ we have $\ERR{1}_{\nn}(s)=0$. Else, according to \eqref{e:error0control}, we have
\begin{align}
\| \ERR{1}_{\nn}&(s,\cdot) \|_{L^\infty(\T^2)}\notag \\
& \leq C(T,\nn)  \Big(  \ga^{\frac12} s^{-\frac13}   +\ga^{-\frac16} \| \Xg^{\mathrm{high}}(s, \cdot)  \|_{L^{\infty}(\T^2)} + \ga^{-\frac16}\|  Q_{\ga,s}(s,\cdot) \|_{L^\infty(\Le)}    \Big) \;. \notag
\end{align}
For the second term on the right-hand side of this estimate we write
\begin{align*}
 \| &\Xg^{\mathrm{high}}(s, \cdot)  \|_{L^{\infty}(\T^2)} \\
   &\leq  \| \Zg^{\mathrm{high}}(s,\cdot)  \|_{L^{\infty}(\T^2)} + \| \tvg^{\mathrm{high}}(s, \cdot)  \|_{L^{\infty}(\T^2)}  + \| \Yg^{\mathrm{high}}(s, \cdot)  \|_{L^{\infty}(\T^2)} \;\\
   &\leq  \| \Zg^{\mathrm{high}}(s,\cdot)  \|_{L^{\infty}(\T^2)} + C \ga^{1}\| \tvg(s, \cdot)  \|_{\Cc^\frac12}  + C(\nu,\bar{\ka}, T)  \ga^1 t^{-\frac{1 +\al}{2}-\bar{\ka}} \| \Xng  \|_{\Ca} \;.
\end{align*}
Here, $\tvg^{\mathrm{high}} $ and  $\Yg^{\mathrm{high}}(s, \cdot)$ are defined analogously to \eqref{e:XhigH}. From this definition, it follows immediately that $\| \tvg^{\mathrm{high}}(s, \cdot)  \|_{L^{\infty}(\T^2)}$ is controlled by $C \ga^{1}\| \tvg(s, \cdot)  \|_{\Cc^\frac12} $, and similarly for $\Yg^{\mathrm{high}}$. We have also used once more the estimates provided in Lemma~\ref{le:semi-group-regularity} to control the $\Cc^{\frac12}$ norm of $\Yg(s)$.

According to \eqref{e:e2Bound}, we have for any $\bar{\ka}>0$
\begin{equation*}
\| \ERR{2}_{\nn} \|_{L^\infty(\T^2)} \leq C(t, \al, \bar{\ka})  \big( \| E_{\ga,s,\nn}^{\colon 3 \colon }\|_{L^\infty(\T^2)}  + s^{-\al-\bar{\ka}} \| E_{\ga,s,\nn}^{\colon 2 \colon }\|_{L^\infty(\T^2)}\big)\;.
\end{equation*}

Summarising all of these calculations, we get
\begin{align*}
\|  \tvg(t, \cdot)-  \bvg(t, \cdot)  \big\|_{\Cc^1} \leq  &\overline{C}_1 \int_0^t  (t-s)^{-\frac13} s^{-\frac{1}{6} }\|  \tvg(s) - \bvg(s) \|_{\Cc^{1}}   \, ds \notag \\
&+ \overline{C}_1 (\ga^{\frac{\ka}{2}}  + \| \Xng - \Xn \|_{\Ca}) + \ERR{3}(t) \;, %\label{e:Gronwall1}
\end{align*}
for a constant $\overline{C}_1$ that depends on all of the quantities listed below \eqref{e:diff_bound2} and \eqref{e:FinaleB17}. The error term
 takes the form
\begin{align*}
\ERR{3}(t) \leq &C(\bar{\ka},T)  \int_0^t (t-s )^{-\frac14 - \bar{\ka}}   \Big( \ga^{-\frac16}\| \Xg^{\mathrm{high}}(s, \cdot)  \|_{L^{\infty}(\T^2)} + \ga^{-\frac16}\|  Q_{\ga,s}(s,\cdot) \|_{L^\infty(\Le)}  \\
&+ \| E_{\ga,s,\nn}^{\colon 3 \colon }\|_{L^\infty(\T^2)}  + s^{-\al-\bar{\ka}} \| E_{\ga,s,\nn}^{\colon 2 \colon }\|_{L^\infty(\T^2)}\Big)  \,ds \;,
\end{align*}

In particular, we can conclude from Lemmas~\ref{le:highfreq}, \ref{le:quartic_var} and~\ref{p:discrete-Wick} that for any $p\geq 2$,
\begin{align*}
\E \sup_{0 \leq t \leq T}  \big| \ERR{3}(t) \big|^p \leq C(p,T,\la) \ga^{\la p} \;,
\end{align*}
for any $\la < \frac56$. So the desired result follows.
%Now, a suitable version o%
%%%%%%%%%%%%%%%%%%
\end{proof}
%%%%%%%%%%%%%%%%%%

We are now ready to conclude the proof of our main result.

%%%%%%%%%%%%%%%
\begin{proof}[Proof of Theorem~\ref{thm:Main}]
%%%%%%%%%%%%%%%
For $\rr$ and $\nn \ge 1$, we define the events $ \Aa_{\rr}^Z =  \Aa_{\rr}^Z(\ga,\nn)$,  and $ \Aa^{\mathsf{E}} =  \Aa^{\mathsf{E}}(\ga,\nn)$ by
\begin{align*}
 \Aa_{\rr}^Z :=& \, \big\{ \|\Zgn    \|_{\Ca}   \leq \rr  \, , \,   \|\ZGn{2}    \|_{\Ca}   \leq \rr \, , \,  \|\ZGn{3}    \|_{\Ca}   \leq \rr \text{ on } [0,T] \big\}\; , \\
 \Aa^{\mathsf{E}} :=&  \big\{ \sup_{0 \leq t \leq T}  \big| \ERR{3}(t) \big| \leq  \gamma^{\frac12}  \,   \big\} \,.
\end{align*}
Recall that
$$
\bvg =  \Ss_T(\Zgn, \Zgn^{\colon 2 \colon},\Zgn^{\colon 3 \colon}) .
$$
By Theorem~\ref{thm:ContSolution}, there exists a constant $C(T,\rr)$ such that on the event $\Aa_\rr^Z$, 
\begin{equation}
\label{e:estim-vbar}
\sup_{0 \leq t \leq T}\big\| \bvg(t,\cdot) \big\|_{\Cc^{\frac12}} \le C(T,\rr).
\end{equation}
Let $\rr_0 = C(T,\rr) + 2$ and let $\td{T} = \td{T}(\ga)$ be the stopping time defined by
$$
\td{T} = \inf \{ t \ge 0 : \|\bvg(t,\cdot)\|_{\Cc^{\frac12}} \ge \rr_0 \}.
$$
By a suitable version of Gronwall's inequality (see e.g. \cite[Lemma 5.7]{HW10}), we deduce that there exists a constant $C$ depending on $T,\rr$ and the choice of small constants $\ka, \al>0$ in Theorem \ref{thm:Main} such that on the event $\Aa_{\rr}^Z \cap \Aa^{\mathsf{E}}$, we have
\begin{align}
\sup_{0 \leq t \leq \td{T}\wedge T } \Ll\|  \tvg(t, \cdot)-  \bvg(t, \cdot)  \Rr\|_{\Cc^{\frac12}} \leq  & C \Ll(\ga^{\frac{\ka}{2}}  + \| \Xng - \Xn \|_{\Ca} \Rr)\;
\label{e:Gronwall452}
\end{align}
(recall that the quantity in the supremum above is a continuous function of $t$).
In particular, for $\ga$ sufficiently small,
\begin{align*}
\sup_{0 \leq t \leq \td{T} \wedge T } \Ll\|  \tvg(t, \cdot)-  \bvg(t, \cdot)  \Rr\|_{\Cc^{\frac{1}{2}}} \leq  1.
\end{align*}
Together with \eqref{e:estim-vbar}, this implies that
$$
\|\bvg(\td{T} \wedge T,\cdot)\|_{\Cc^{\frac12}} \le \rr_0 - 1.
$$
By continuity of $t \mapsto \|\bvg(t,\cdot)\|_{\Cc^{\frac12}}$, this implies that $\td{T} \ge T$, and thus that \eqref{e:Gronwall452} can be upgraded to
\begin{align}
\sup_{0 \leq t \leq T }\|  \tvg(t, \cdot)-  \bvg(t, \cdot)  \big\|_{\Cc^{\frac12}} \leq  & C \big(\ga^{\frac{\ka}{2}}  + \| \Xng - \Xn \|_{\Ca} \big)\;, 
\label{e:Gronwall454}
\end{align}
on the event $\Aa_{\rr}^Z \cap \Aa^{\mathsf{E}}$ and for $\gamma>0$ small enough (depending in $\rr$). Recalling that
\begin{align*}
\tXgn = \tvg + \Zgn + \Yg \quad \text{ and } \quad \oX = \bvg + \Zgn + Y,
\end{align*}
we deduce from \eqref{e:Gronwall454} and Lemma~\ref{le:ic} that for every $\nn, \rr\geq 1$ and every bounded uniformly continuous mapping $F : \Dd([0,T] , \Ca) \to \R$,
\begin{equation}
\label{e:almostthere}
\lim_{\gamma \to 0 }\E   \Ll[  \Ll| F\Ll(   \oX \Rr) -  F\Ll(\tXgn \Rr)  \Rr| \mathbf{1}_{ \Aa^Z_{\rr} \cap \Aa^{\mathsf{E}}}  \Rr]   =0 \;.
\end{equation}
Let us write $\ov{\Aa}^Z_\rr$ and $\ov{\Aa}^{\mathsf{E}}$ for the complementary events of $\Aa^Z_\rr$ and $\Aa^{\mathsf{E}}$ respectively. 
Note that in \eqref{e:almostthere}, the indicator function of the event $\Aa^\mathsf{E}$ does no harm, since it follows from \eqref{e:very_good_bound_A} and Chebyshev's inequality that
\begin{equation}
\label{e:contovaaE}
\lim_{\ga \to 0} \P\Ll[ \ov{\Aa}^\mathsf{E} \Rr] = 0.
\end{equation}
Let $\eps > 0$. We now argue that by choosing $\rr$ sufficiently large, we can make the probability of the event $\ov{\Aa}^Z_{\rr}$ smaller than $\eps$ as $\ga \to 0$. In order to do so, let us introduce the stopping times 
\begin{align*}
\tau^{Z,\rr}_{\ga,\nn} & = \inf \Ll\{  t \geq 0 : \max\Ll(\| \Zgn(t) \|_{\Ca}, \| \ZGn{2}(t) \|_{\Ca}, \| \ZGn{3}(t) \|_{\Ca}\Rr) \geq \rr \Rr\},\\
\tau^{Z,\rr} & = \inf \Ll\{  t \geq 0 : \max \Ll( \| Z(t) \|_{\Ca}, \| Z^{\colon 2 \colon}(t) \|_{\Ca}, \| Z^{\colon 3 \colon}(t) \|_{\Ca} \Rr) \geq \rr \Rr\}\;. 
\end{align*}
Recall from Theorem~\ref{t:converg-lin-Wickbis} that $(\Zgn,\ZGn{2}, \ZGn{3})$ converges in law to $(Z, Z^{\colon 2 \colon},  Z^{\colon 3 \colon})$ for the topology of $\Dd(\R_+,\Ca)^3$. Arguing as in the proof of Theorem~\ref{t:converg-lin}, we see that for every $\rr$ outside of a countable set $\mathsf{Loc}$ (the set of $\rr$ such that $t \mapsto \max\big(\| \Zgn(t) \|_{\Ca}$, $\| \ZGn{2}(t) \|_{\Ca}, \| \ZGn{3}(t) \|_{\Ca}\big)$ attains the value $\rr$ as a local maximum with positive probability), the random variable $\tau^{Z,\rr}_{\ga,\nn}$ converges in law to $\tau^{Z,\rr}$ as $\ga \to 0$ (keeping $\nn$ fixed). By Proposition~\ref{prop:daPratoDebussche}, the random variable 
$$
\sup_{0 \le t \le T} \max \Ll( \| Z(t) \|_{\Ca}, \| Z^{\colon 2 \colon}(t) \|_{\Ca}, \| Z^{\colon 3 \colon}(t) \|_{\Ca} \Rr)
$$
is finite a.s.,\ so it suffices to choose $\rr$ sufficiently large to ensure that 
$$
\P\Ll[\tau^{Z,\rr} \le T \Rr] \le \eps\;.
$$ 
Enlarging $\rr$ if necessary, we can also make sure that $\rr \notin \mathsf{Loc}$, and thus get
\begin{equation}
\label{e:contovaaZ}
\limsup_{\ga \to 0} \P\Ll[\ov{\Aa}^Z_{\rr}\Rr] \le \limsup_{\ga \to 0} \P[\tau^{Z,\rr}_{\ga,\nn} \le T] \le \P\Ll[\tau^{Z,\rr} \le  T\Rr] \le \eps.
\end{equation}
We decompose
\begin{multline*}
\Ll| \E\Ll[ F\Ll(   \tXgn \Rr)\Rr] -  \E\Ll[ F\Ll(X \Rr) \Rr]\Rr| 
 \leq \Ll| \E\Ll[ F\Ll( \oX \Rr) \Rr] - \E \Ll[ F\Ll(X \Rr) \Rr]  \Rr| \\  
 +\E   \Ll[  \Ll| F\Ll( \oX \Rr) -  F\Ll( \tXgn \Rr)  \Rr| \mathbf{1}_{ \Aa^Z_{\rr} \cap \Aa^{\mathsf{E}}}  \Rr]  
+ \| F\|_{L^\infty}\P \Ll[\ov{\Aa}^Z_{\rr} \cup \ov{\Aa}^\mathsf{E}\Rr].
\end{multline*}
By Lemma~\ref{le:FirstPartOfMainThm}, \eqref{e:almostthere}, \eqref{e:contovaaE} and \eqref{e:contovaaZ}, we obtain  that
$$
\limsup_{\ga \to 0} \Ll| \E\Ll[ F\Ll(   \tXgn \Rr)\Rr] -  \E\Ll[ F\Ll(X\Rr) \Rr]\Rr|  \le \eps \| F\|_{L^\infty}\; .
$$
Since $\eps > 0$ was arbitrary, this proves that $\tXgn$ converges in law to $X$ as $\ga$ tends to~$0$, for any fixed value of $\nn$.

We can now remove $\nn$ by a similar reasoning. Recall the definition of $\taun$ in \eqref{e:deftaug} as 
\begin{equation*}
\taun = \inf \Ll\{  t \geq 0 : \| \tXgn(t) \|_{\Ca} \geq  \nn \Rr\}\;,
\end{equation*}
and set
\begin{equation*}
\tau_\nn = \inf \Ll\{  t \geq 0 : \| X(t) \|_{\Ca} \geq  \nn \Rr\}. 
\end{equation*}
Arguing as above, we obtain that for every $\nn$ outside of a countable set $\mathsf{Loc}'$, the stopping time $\taun$ converges in law to $\tau_\nn$. Moreover, we know from Theorem~\ref{thm:ContSolution} that $\sup_{0 \leq t \leq T+1} \| X(t) \|_{\Ca}$ is almost surely finite. Hence, for any given $\eps > 0$, we can choose $\nn = \nn(T,\eps)$ sufficiently large and outside of $\mathsf{Loc}'$ so that
$$
\limsup_{\ga \to 0} \P[\taun \le T] \le \eps\;.
$$
Recalling that $\tXgn$ and $\Xg$ coincide up to $\taun$, this implies that
\begin{equation*}
\limsup_{\ga \to 0} \P \Ll[ \tXgn \neq \Xg \Rr] \leq \eps\;.
\end{equation*}
Since $\tXgn$ converges in law to $X$, this concludes the proof of Theorem~\ref{thm:Main}.
\end{proof}
%
%
%
%
%
%%%%%%%%%%%%%%%%%%%%%%%%%%%%
\section{Some bounds on the kernels $\Kg$ and $\Pg{t}$}
%%%%%%%%%%%%%%%%%%%%%%%%%%%%
\label{sec:APPB}
%%%%%%%%%%%%%%%%%%%%%%%%%%%%

In this section, we collect some facts about the kernels $\Kg$ and the approximate heat semigroups $\Pg{t}$. We start by summarising some properties of the Fourier transform $\hKg$ of $\Kg$. Recall that for $0< \ga < \frac13$ and $\om \in \{-N, \ldots, N\}^2$,  it is given by 
\begin{align}
\hKg (\om) = 
 \sum_{x\in \Le} \eg^2 \, \Kg(x) \,e^{- i \pi \om \cdot x}  =  \,
\ct\sum_{x\in \ga \Z^2_\star} \ga^{2} \, \KK( x) \,e^{- i \pi (\eps/ \ga) \om \cdot x}\;, \label{e:B1}
\end{align} 
where $\KK$ is the smooth function introduced at the beginning of Section~\ref{s:Setting}. In the second equality, we have used the fact that $\KK$ has compact support in $B(0,3)$ to replace the sum over $\frac{\ga}{\eps} \Le^{\star}$ by a sum over $\ga \Z^2_{\star} := \ga \Z^2 \setminus \{0\}$.
Recall our choice of scaling \eqref{e:scaling1}, \eqref{e:scaling2}, in particular $\eps = \ga^2 \co$.

For some of the following calculations, it is useful to view $\hKg$ as a function of a continuous parameter by evaluating \eqref{e:B1} for all $\om \in \R^2$. The function $\hKg$ defined in this way is smooth and $(2N+1)$ periodic in both coordinates. We will typically evaluate it only for $\om \in [-N- \frac12,  N+\frac12]^2$. Furthermore, we have for all $\om$ and for $j=1,2$
\begin{align}
\partial_j \hKg (\om) &= 
 \,
- i \frac{\eps}{\ga} \pi \ct  \sum_{x\in \Lg} \ga^{2} \, x_j  \,  \KK( x) \,   e^{- i \pi (\eps/ \ga) \om \cdot x   }  \;,
\label{e:B1A}
\end{align}
and
\begin{align}
\partial_j^2 \hKg (\om) &= 
 \,
-  \frac{\eps^2}{\ga^2} \pi^2 \ct   \sum_{x\in \Lg} \ga^{2} \, x_j^2 \,   \KK( x) \,  e^{- i \pi (\eps/ \ga) \om \cdot x   }  \;.
\label{e:B2A}
\end{align}
In \eqref{e:B1A} and \eqref{e:B2A}, we can sum over $\Lg$ instead of $\Lg_{\star}$ because the summand at $x=0$ vanishes. This will be slightly more convenient below.

For small $\ga$, the expression \eqref{e:B1} approximates $\hKK(\ga \om)$, where the continuous Fourier transform $\hKK(\om)$ is defined as
\begin{equation*}
\hKK(\om) := \int_{\R^2} \KK(x)  \,e^{ -i \pi \om \cdot x} \, dx \; \qquad (\om \in \R^2).
\end{equation*}
The following lemmas state that some properties of $\hKK(\ga \om)$ also hold for $\hKg$, uniformly in $\ga$. We begin with pointwise estimates.
%
%
%
%

%%%%%%%%%%%%%%%%%%%%%%%%
\begin{lemma}\label{le:Kg0}
%%%%%%%%%%%%%%%%%%%%%%%%
There exists  a constant $C>0$ such that for all $0 < \ga< \frac13$ and for $|\om| \leq \ga^{-1}$ we have for $j=1,2$
\begin{align}
\big| \ga^{-2} (1 -\hKg(\om) ) -  \pi^2 |\om|^2 \big| &\leq C \ga | \om|^3\;, \label{e:K2.4} \\
\big|  -\ga^{-2}  \partial_j \hKg(\om)  - 2  \pi^2 \om_j \big| &\leq C \ga | \om|^2\;,\label{e:K2.3}\\
\big|  -\ga^{-2}  \partial_j^2 \hKg(\om)  - 2  \pi^2 \big| &\leq C \ga | \om| \;.\label{e:K2.2}
\end{align}
\end{lemma}
\begin{proof}
For any $|\om| \leq \ga^{-1}$, a Taylor expansion yields
\begin{align}
1 -  \hKg (\om) 
=& \,  \ct \sum_{x\in \Lg} \ga^{2} \, \KK( x) \, \big(1 - e^{- i \pi (\eps/ \ga) \om \cdot x} \big) \notag \\
=&  \, \ct \sum_{x\in \Lg} \ga^{2} \, \KK( x) \, \big( i \pi \tfrac{\eps}{\ga} \, \om \cdot x    + \tfrac{1}{2}  \big( \pi  \tfrac{\eps}{ \ga}  \,\om \cdot x\big)^2\Big) + \msf{Err} \;, \notag
\end{align}
where 
\begin{align*}
| \msf{Err}| \leq  \frac{\eps^3}{\ga^3} |\om|^3\; \frac{\pi^3}{6} \,  \ct   \sum_{x\in \Lg} \ga^{2} \, |x|^3  \KK( x) \, \leq   C \, \ga^3 |\om|^3 \  .
\end{align*}
By the symmetry of the kernel $\KK(x)$, we have
\begin{equation*}
\sum_{x\in \Lg} \ga^{2} \, \KK( x) \, \big( \om \cdot x \big)   =0\,\quad \text{and} \quad   \sum_{x\in \Lg} \ga^{2} \,x_1 \, x_2 \, \KK( x) \,  =0 \;.
\end{equation*}
Furthermore, for $j=1,2$, the sums
\begin{equation*}
\ct \sum_{x\in \Lg_\star} \ga^{2} \, \KK( x) \,        x_j^2  
\end{equation*}
converge to $\int_{\R^2} \KK(x) \, x_j^2 \, dx =2$  as $\ga \to 0$ and the error is controlled by $C\ga^2$, so \eqref{e:K2.4} follows.

The remaining bounds \eqref{e:K2.3} and \eqref{e:K2.2} follow in a similar manner: For \eqref{e:K2.3}, we write
\begin{align*}
-\partial_j \hKg(\om)& =  i  \frac{\eps}{\ga} \pi \ct   \,\sum_{x\in \Lg} \ga^{2} \,x_j \, \KK( x)  \, \big( e^{- i \pi (\eps/ \ga) \om \cdot x   }  -1 \big) \\
& =  \ga^2 \pi^2 \om_j \ct \co^2 \sum_{x \in \Lg} \ga^2 x_j^2 \,\KK(x) + \msf{Err}' \; ,
\end{align*}
for an error term $\msf{Err}'$ that is bounded by $C|\om|^2 \ga^3$ uniformly for $|\om| \leq \ga^{-1}$. Here we have used the symmetry of the kernel $\KK$ 
to add the term $-1$ in the fist equality and to remove the sum over $x_1 x_2$ in the Taylor expansion in the second line. The bound then follows as above. 

In the same way, we write
\begin{align*}
- \partial_j^2 \hKg(\om)& =   \frac{\eps^2}{\ga^2} \pi^2   \ct   \,\sum_{x\in \Lg} \ga^{2} \,x_j^2 \, \KK( x)  \,  e^{- i \pi (\eps/ \ga) \om \cdot x   }  \\
& =  \ga^2 \pi^2  \ct \co^2 \sum_{x \in \Lg} \ga^2 x_j^2 \,\KK(x) + \msf{Err}^{''}\, ,
\end{align*}
where $\msf{Err}^{''} \leq C \ga |\om|$ and the bound \eqref{e:K2.2} follows.
\end{proof}

%%%%%%%%%%%%%%%%
\begin{lemma}
%%%%%%%%%%%%%%%%
\label{le:Kg}
%%%%%%%%%%%%%%%%
There exists a constant $C$ such that for all $0<\ga<\frac13$, $\om \in [-N-\frac12, N +\frac12]^2$ and $j=1,2$,
\begin{enumerate}
\item (Estimates most useful for $|\om| \leq \ga^{-1}$)
\begin{align}
|\hKg(\om) | &\leq 1 \;, \label{e:K1}\\
|\partial_j \hKg(\om) | &\leq C \ga \big(  |\ga\om| \wedge 1  \big)  \;, \label{e:K1A}\\
|\partial_j^2 \hKg(\om) | &\leq  C \ga^2 \;, \label{e:K1B}
\end{align}
\item (Estimates most useful for $|\om| \geq \ga^{-1}$)
\begin{align}
    | \ga \om|^2 \;  \big| \hKg (\om) \big| &\leq C \;, \label{e:K3} \\
 |  \ga \om|^2 \;  \big|  \partial_j \hKg (\om) \big| &\leq C \ga \;, \label{e:K3A}\\
  |  \ga \om|^2 \;  \big| \partial_j^2\hKg (\om) \big| &\leq C \ga^2 \;. \label{e:K3B}
\end{align}
\end{enumerate}
Furthermore, there exist constants $C_1>0$ and $\ga_0 >0$ such that for all $0 <  \gamma < \ga_0$ and  $\om \in [-N-\frac12, N +\frac12]^2$ ,
\begin{align}
(1 -\hKg(\om) ) &\geq \frac{1}{C_1}  \big( |\ga\om|^2  \wedge 1 \big)\;.\label{e:K2} 
\end{align} 
\end{lemma}
%

%%%%%%%%%%%%%%%%%%%%%%%
%

%
%
\begin{proof}
To see \eqref{e:K1}, we write
\begin{align*}
|\hKg(\om) | \leq \ct \sum_{x\in \Lg_\star} \ga^{2} \, \KK( x) =1 \;.
\end{align*}
For $|\om| \leq \ga^{-1}$, equations \eqref{e:K1A} and \eqref{e:K1B} follow directly from Lemma~\ref{le:Kg0}. For $|\om| \geq \ga^{-1}$, the bounds \eqref{e:K3} -- \eqref{e:K3B} are stronger, so it suffices to establish those.

The argument for \eqref{e:K3} is very similar to the argument used in the proof of Lemma~\ref{l:Bernstein}. Indeed, for any function $f \colon \Lg \to \R$, we set
\begin{equation*}
\Delg f (x) = \ga^{-2} \sum_{\substack{\bar{x} \in \Lg \\ \bar{x} \sim x}} (f(\bar{x}) - f(x))\;,
\end{equation*}
where $\bar{x} \sim x$ means that $\bar{x}$ and $x$ are adjacent in $\Lg$.   Similar to \eqref{e:IBPd1}, for fixed $\om \in \R^2$   we set  $e_\om : x \mapsto e^{-i \pi (\eps/\ga) \om \cdot x}$, and we get that
\begin{equation*}
-\Delg e_\om = 2 \ga^{-2} \sum_{j = 1}^2(1- \cos( \eps \pi \om_j) ) e_\om\;.
\end{equation*}
After a summation by parts, we get
\begin{align}
\hKg (\om) = 
\frac{1}{ 2 \ga^{-2} \sum_{j = 1}^2(1- \cos(\eps \pi \om_j) ) }\ct \sum_{x\in \Lg} \ga^{2} \, (-\Delg  \tilde{\KK}( x) ) \, e_\om(x) \;,\label{e:kernel1}
\end{align} 
where $\tilde{\KK}( x) = \KK(x)$ for $x \neq 0$ and $\tilde{\KK}( 0) = 0$.

Recalling our scaling (\eqref{e:scaling1} and \eqref{e:scaling2}), it is easy to see that 
 \begin{equation*}
 \frac{1}{ \ga^{-2} \sum_{j = 1}^2(1- \cos(\eps \pi \om_j) ) } \leq C \frac{1}{\ga^2 |\om|^2} \;,
 \end{equation*}
uniformly over $\eps$ and $|\om_i| \leq N+\frac12$.  On the other hand the fact that $\KK$ is a $\Cc^2$ function with compact support shows that the sum on the right-hand side of \eqref{e:kernel1} is bounded when restricted to points $x$ that are not $0$ or adjacent to $0$. For those five points $x$ which are, we bound 
$\Delg \tilde{\KK}(x) \leq \frac{8}{ \gamma^2}$. Hence,  the sum over these five points weighted with $\gamma^2$ is bounded by $40$. This finishes the proof of \eqref{e:K3}.

The arguments for \eqref{e:K3A} and \eqref{e:K3B} are almost identical to the argument for \eqref{e:K3}. Indeed, performing the same summation by parts as in \eqref{e:kernel1}, we get
\begin{align*} 
\big| \partial_j \hKg (\om)\big| \leq C \frac{\ga}{| \ga \om |^2}  \sum_{x\in \Lg} \ga^{2} \, |-\Delg  (x_j \, \KK( x)) |\;,
\end{align*}
and
\begin{align*} 
\big| \partial_j^2 \hKg (\om)\big| \leq C \frac{\ga^2}{| \ga \om |^2}  \sum_{x\in \Lg} \ga^{2} \, |-\Delg  (x_j^2 \, \KK( x)) | \;.
\end{align*}
Note that the problem caused by the fact that $\KK(0)$ was defined replaced by $0$ has disappeared in these expressions. By assumption, both $x_j \, \KK(x)$ and $x^2_j \, \KK(x)$ are $\Cc^2$ functions with compact support, so that the sums appearing here are uniformly bounded. This shows \eqref{e:K3A} and \eqref{e:K3B}.

Let us proceed to the proof of \eqref{e:K2}. From \eqref{e:K3}, we get the existence of $\overline{c}>0$ such that for $\om \geq \overline{c} \ga^{-1}$, we have $\hKg(\om) \leq \frac12$. Hence, \eqref{e:K2} holds for such $\om$. Next, we treat $\om$ with $|\om| <
\underline{c} \ga^{-1}$ for a $\underline{c}>0$ to be fixed below. For such $\om$
\eqref{e:K2.4} implies the existence of $\underline{C}$ such that
\begin{align*}
1 -  \hKg (\om)  \geq \pi^2 |\om|^2 \ga^2  -  \,\underline{C}  |\om|^3 \ga^3  \geq \big( \pi^2 -\underline{C} \, \underline{c}) \big) |\om|^2 \ga^2 \;, 
\end{align*}
which can be bounded from below by $\frac{\pi^2}{2} |\om|^2 \ga^2 $ if we choose $\underline{c}$ small enough. 

Finally, in order to treat the case $\underline{c} \ga^{-1} \leq |\om | \leq \overline{c} \ga^{-1}$, we observe that the Riemann sums 
\begin{equation*}
\Kg\big( \ga^{-1}\om  \big)=  \ct \sum_{x\in \Lg_\star} \ga^{2} \, \KK( x) \,e^{- i \pi \co \,\om \cdot x}
\end{equation*}
approximate $\hKK(\om)$  uniformly for $ |\om| \in [\underline{c}, \overline{c}]$. On the other hand, $\hKK$ is the Fourier transform of a probability measure with a density on $\R^2$, and as such, it is continuous and has $|\hKK(\om)|<1$ if $\om \neq 0$. In particular, $|\hKK(\om)|$ is bounded away from $1$ uniformly for  $ |\om| \in [\underline{c}, \overline{c}]$. Combining these facts, we see that for $\ga$ small enough, $\Kg(\om)$ is bounded away from $1$ uniformly in $\underline{c} \ga^{-1} \leq |\om | \leq \overline{c} \ga^{-1}$. This shows \eqref{e:K2}.
\end{proof}
%%%%%%%%%%%%%%%%%%%%%%%

The next lemma provides some pointwise estimates on the kernels $\Pg{t} \star \Kg$ that are used in the proof of Lemma~\ref{le:ErrIntBound}.  %
On $\Le$, $\Pg{t}(\cdot)$ is a Markov kernel for every $t \geq0$:
\begin{equation}\label{e:Prepresentation}
x \in \Le \Rightarrow \Pg{t}(x) \geq 0   \quad \text{ and } \quad \sum_{x \in \Le} \eps^2 \Pg{t}(x) =1 \,.
\end{equation}
Recall our convention to define $\Kg$ and $\Pg{t}$ for all $x \in \T^2$ by extending them as trigonometric polynomials of degree $\leq N$. We note that for $x \notin \Le$, the properties $\Kg(x) = \ct  \frac{\ga^{2}}{\eps^2} \, \KK( x) $  and $\Pg{t}(x) \geq 0$ do \emph{not} hold in general.

%%%%%%%%%%%%%%%%%%
\begin{lemma}\label{le:Pgt}
%%%%%%%%%%%%%%%%%%
Let $\ga_0>0$ be the constant introduced in Lemma~\ref{le:Kg}. For every $T>0$, there exists a constant $C = C(T)$ such that for all $0 < \gamma<\ga_0$, $0\leq t \leq T$ and $x \in \T^2$, we have
\begin{equation}\label{e:P0}
\big|    \Pg{t} \star  \Kg(x) \big| \leq C \big( t^{-1} \big(\log(\ga^{-1})\big)^2    \wedge \gamma^{-2} \log(\ga^{-1}) \big) \;,
\end{equation}
and
\begin{equation}\label{e:P1}
  |x|^2  \; \big| \Pg{t} \star \Kg  (x) \big| \leq C \, \log(\ga^{-1})  \;.
\end{equation}
\end{lemma}
%%%%%%%%%%%%%%%%%%%
%
%

%%%%%%%%%%%%%%%%%%%
\begin{proof}
We write
\begin{align*}
\big|    \Pg{t} \star  \Kg(x) \big|  \leq  \frac{1}{4} \sum_{\om \in \{-N, \ldots,N \}^2 }  \exp\Big(t \gamma^{-2} (\hKg(\om) - 1) \Big)  \; \big| \hKg(\om) \big| \;.
\end{align*}
We first bound the sum over $\om$ satisfying $|\om| \leq \ga^{-1}$. According to \eqref{e:K1} and \eqref{e:K2}, there exists  $C_1>0$ such that for $0 < \ga < \ga_0$  
\begin{align*}
 \sum_{|\om| \leq \ga^{-1}}  \exp\big(t \gamma^{-2} (\hKg(\om) - 1) \big)  \; \big| \hKg(\om) \big| \leq  \sum_{|\om| \leq \ga^{-1}} \exp\Big(-\frac{t}{C_1} |\om|^2 \Big) \leq C\big( t^{-1}  \wedge \ga^{-2}  \big) \;, 
\end{align*}
for a universal constant $C$. On the other hand, using \eqref{e:K3} and \eqref{e:K2}, we get 
\begin{align*}
\sum_{\substack{\om \in \{-N, \ldots,N \}^2 \\ |\om| > \ga^{-1}}}  & \exp\big(t \gamma^{-2} (\hKg(\om) - 1) \big)  \; \big| \hKg(\om) \big| \\
&\leq C \sum_{\substack{\om \in \{-N, \ldots,N \}^2 \\ |\om| > \ga^{-1}}}   \exp\Big( - \frac{t}{C_1 \gamma^{2}}  \Big) \frac{1}{|\ga \om|^2}\;.
\end{align*}
If $t \geq 4 C_1 \ga^2 \log(\ga^{-1})$ (and recalling that according to \eqref{e:scaling1}, $N \leq \ga^{-2}$), then 
\begin{align*}
\sum_{\substack{\om \in \{-N, \ldots,N \}^2 \\ |\om| > \ga^{-1}}}   \exp\Big( - \frac{t}{C_1 \gamma^{2}}  \Big) \frac{1}{|\ga \om|^2}\leq 5 \ga^{-4}  \exp\Big( - \frac{t}{C_1 \gamma^{2}}  \Big)  \leq  5 \ga^{-4} \ga^4   = 5 \;. 
\end{align*}
Else, if  $t <4 C_1 \ga^2 \log(\ga^{-1}) $, we have
\begin{align}
\sum_{\substack{\om \in \{-N, \ldots,N \}^2 \\ |\om| > \ga^{-1}}}   \exp\Big( - \frac{t}{C_1 \gamma^{2}}  \Big) \frac{1}{|\ga \om|^2} \leq C \ga^{-2} \log(\ga^{-1}) \leq C t^{-1 } \big(\log(\ga^{-1})\big)^2  \;,  \notag
\end{align}
which shows \eqref{e:P0}.

The proof of \eqref{e:P1} is again similar to the proof of Lemma~\ref{l:Bernstein} and the proof of Lemma~\ref{le:Kg}. As in those arguments, we define a discrete Laplace operator, this time acting on functions $f \colon \Z^2 \to \R$ and defined as
\begin{equation}\label{e:PK1}
\Dell f (\om) =  \sum_{\substack{\bar{\om} \in \Z^2 \\ \bar{\om} \sim \om}} (f(\bar{\om}) - f(\om))\;.
\end{equation}
As above, for fixed $x \in \T^2$   we set  $e_x : \om \mapsto e^{i \pi \om \cdot x}$, and we get that
\begin{equation*}
-\Dell e_x = 2\sum_{j = 1}^2(1- \cos( \pi x_j) ) \,e_x\;.
\end{equation*}

In this way, after a summation by parts over $\Z^2$, we get for all $x \in \T^2$ and $t \geq 0$ that
\begin{align}
&\Ll|\Pg{t} \star \Kg (x) \Rr|
 \leq  
\frac{1}{ 2 \sum_{j = 1}^2(1- \cos( \pi x_j) ) }  \notag\\
& \times  \frac14 \sum_{\om \in \{-N-1, \ldots,N+1 \}^2}
  \Ll| \;- \Dell \Ll[  \exp\big(t \gamma^{-2} \big(\hKg(\om) - 1\big) \big)  \;  \hKg(\om)  \Rr] \; \Rr|  \;.\label{e:PK2}
\end{align} 
Recall that the $\hPg{t}$ and $\hKg$ are defined as zero outside of $\{-N. \ldots, N\}^2$. Hence by \eqref{e:PK1}, the discrete Laplacian of $\hPg{t} \hKg$ is zero outside of $\{ -N-1, \ldots,N+1\}^2$.

We treat the sum over the boundary points, i.e. over those $\om = (\om_1,\om_2)$ for which at least one $\om_j$ satisfies $|\om_j |=N$ or $|\om_j |=N+1$, first. For such an $\om$  we bound brutally using \eqref{e:K3}
\begin{align*}
\big|- \Dell \hPg{t} \hKg (\om)\big| \leq   4|\hPg{t} \hKg (\om)\big| +  \sum_{\substack{\bar{\om} \in \Z^2 \\ \bar{\om} \sim \om}} | \hPg{t} \hKg(\bar{\om}) | \leq \frac{C}{|\ga \om|^2} \leq C\ga^2 \;.
\end{align*}
There are $16N+8 \leq C \ga^{-2}$ such boundary points, so that the sum over these points in \eqref{e:PK2} is bounded uniformly in $0<\ga< \ga_0$ and $ t $.

In order to bound the discretised Laplacian appearing on the right-hand side of \eqref{e:PK2} at interior points $\om$, it is useful to pass to continuous coordinates. If $f\colon \R^2 \to \R$ is $\Cc^2$, then for $\om \in \Z^2$, equation \eqref{e:PK1} can be rewritten as 
\begin{align}
\big| \Dell f (\om)  \big| &=  \Ll|  \sum_{j =1,2}   \int_{-1}^1 \partial_j^2 f(\om + \tau e_j) \,  (1 - |\tau|) d \tau  \Rr| \notag\\
& \leq \sum_{j=1,2} \sup_{\tau \in [-1,1]}  \big| \partial_j^2 f(\om + \tau e_j)\big| \,,\label{e:dicsLaptocont}
\end{align}
where $e_1$ and $e_2$ are the standard unit basis vectors of $\R^2$. In order to apply this to bound \eqref{e:PK2}, we calculate for $j =1,2$
\begin{align}
\partial_{j}^2 & \big(  e^{\lambda (\hKg - 1)}  \;  \hKg  \big) \notag \\
&= e^{ \la (\hKg - 1)}   \Big( \la^2 \big( \partial_{j} \hKg \big)^2\hKg  +  \la  \big(  \partial_{j}^2 \hKg \big)  \hKg + \partial_j^2 \hKg + 2 \la (\partial_j \hKg)^2   \Big) \;, \label{e:PK3}
\end{align}
where to improve readability we have dropped the arguments $\om$ and  set $\lambda := t \gamma^{-2}$.

We now use  the bounds derived in Lemmas~\ref{le:Kg0} and \ref{le:Kg} to bound the terms on the right-hand side of this expression one by one. For the first term, we write for $0 < \ga< \ga_0$,   $|\om| \leq \ga^{-1}$ and $j=1,2$ 
\begin{align*}
\sup_{\tau \in [-1,1]} & e^{ \la (\hKg(\om + \tau e_j) - 1)}    \la^2  \Ll(\partial_{j} \hKg(\om + \tau e_j)\Rr)^2\, \hKg(\om + \tau e_j) \\
&\leq C  \sup_{\tau \in [-1,1]} \exp\Big( - \frac{t}{C_1}|\om + \tau e_j|^2  \Big)  \;  t^2  \big| \om + \tau e_j  |^2 \\
&\leq C \exp\Big( - \frac{t}{2 C_1}  \,|\om |^2  \Big)  \,e^{\frac{t}{C_1}}  \,  2t^2  (|\om|^2 +1) \;,
\end{align*}
where we have used \eqref{e:K1}, \eqref{e:K1A}, and \eqref{e:K2} in the first inequality.  The sum over $\{ \om \in \Z^2 \colon |\om| \leq \gamma^{-1} \}$ of the terms appearing in the last line is bounded uniformly in $0 \leq t \leq T$ and $0 < \gamma< 1$.

For the second term, we write in a similar way for $|\om| \leq \ga^{-1}$ and $j=1,2$ 
\begin{align*}
\sup_{\tau \in [-1,1]} & e^{ \la (\hKg(\om + \tau e_j) - 1)}    \la \, \partial_{j}^2 \hKg(\om + \tau e_j) \,  \hKg(\om + \tau e_j) \\
&\leq C \exp\Big( - \frac{t}{2 C_1}  \,|\om |^2  \Big)  \, e^{\frac{t}{C_1}}  \,  t  \;,
\end{align*}
where we have used  \eqref{e:K1}, \eqref{e:K1B}, and \eqref{e:K2}. The sum over $|\om| \leq \ga^{-1}$ of these terms is again uniformly bounded for $0 \leq t \leq T$ and $0 < \ga<1$.

For the third term on the right-hand side of \eqref{e:PK3}  and $|\om| \leq \ga^{-1}$, we use the $\om$-independent bound 
\begin{align*}
 e^{ \la (\hKg - 1)} \big( \partial_j^2 \hKg  \big) \leq C \ga^2 \;,
\end{align*}
which follows immediately from \eqref{e:K1B} once we observe that $ e^{ \la (\hKg - 1)}  \leq 1$.

Finally, for the fourth term we write
\begin{align*}
2  e^{ \la (\hKg(\om) - 1)}   \la (\partial_j \hKg(\om))^2 \leq C \exp\Big( -\frac{t}{C_1}|\om|^2 \Big) \frac{t}{\ga^2} \ga^4 |\om|^2 \leq C \ga^2 \;,
\end{align*}
which shows that the sum over $|\om| \leq \gamma^{-1}$ of these terms is uniformly bounded.

For the sums over $|\om| \geq \ga^{-1}$, we note that by \eqref{e:K2}, for $0 < \ga < \ga_0$ the terms $ e^{ \la (\hKg - 1)}    \la^2$ and $ e^{ \la (\hKg - 1)}  \la$ are uniformly bounded for $|\om| \geq \ga^{-1}$, so that we can bound the whole expression on the right-hand side of \eqref{e:PK3} by
\begin{align*}
C  \Big(  \big( \partial_j \hKg \big)^2\hKg  +    \big(  \partial_j^2 \hKg \big)  \hKg + \partial_j^2 \hKg + 2  (\partial_j \hKg)^2   \Big) \\
 \leq C \bigg(\frac{\ga^2}{|\ga \om|^6}   +  \frac{\ga^2}{|\ga \om|^4} +   \frac{\ga^2}{|\ga \om|^2} \bigg) \;.
\end{align*}
Taking local maxima over $\om + \tau e_j$ only changes the constant in this expression. The sum over $\{ \om \in \{ -N+1, \ldots, N-1\}^2 \colon \, |\om| >\ga^{-1} \}$ of  this expression can be bounded by a constant times $\log(\ga^{-1})$. Hence the proof is complete.
\end{proof}
%%%%%%%%%%%%%%%%%%%%%%%%%
%
%
%
%
%

Finally, we summarise the regularising properties of these approximate heat operators $\Pg{t}$ as well as $\Pg{t} \star \Kg$. We recall that the \emph{usual} heat 
operator $P_t$ satisfies 
\begin{equation}
\| P_t X \|_{\Cc^{\al + \be}} \leq C(\be) \, t^{-\frac{\be}{2}} \, \|X\|_{\Cc^{\al}} \; \label{e:heat-semigroup-reg}
\end{equation}
for all $\al \in \R$, $\beta \geq0$ and $t \geq 0$.

In the following lemma, we discuss analogous bounds for the approximate operators. 
Recall that the operators $\Pg{t}$ and $\Kg \star$ are naturally defined on trigonometric polynomials of degree $\leq N$, and are otherwise viewed as Fourier multiplication operators (that is, they evaluate to $0$ on trigonometric monomials of degree $>N$). Their behaviour is quite different on frequencies that are small compared to $\ga^{-1}$ and on frequencies that are large compared to $\ga^{-1}$, so we treat
those two cases separately.  Our argument essentially follows  \cite[Lemma A.5]{Gubi}.

%%%%%%%%%%%%
\begin{lemma}
%%%%%%%%%%%%
\label{le:semi-group-regularity}
%%%%%%%%%%%%
Let $\ga_0$ be the constant introduced in Lemma~\ref{le:Kg} and fix $c_1,c_2 >0$, $T>0$, and $\ka >0$. 
\begin{enumerate}
\item For every $\beta >0$ and $0 \leq \lambda \leq 1$, there exists a constant $C = C(c_1,T,\ka,\be,\la)$ such that for all functions $X\colon \T^2 \to \R$   with $\hat{X}(\om) = 0$ for $|\om| \geq c_1 \ga^{-1}$, we have for all $0 < \ga < \ga_0$, $0 \leq t \leq T$  and $\al \in \R$  
\begin{align}
\| \Pg{t} X \|_{\Cc^{\al + \be - \ka}} \leq & \; C\; t^{- \frac{\be}{2}} \| X\|_{\Cc^\al} \;,\label{e:Regular1}\\
\| (\Pg{t} - P_t) X \|_{\Cc^{\al - \ka }} \leq &  \; \Big( C  \;  \ga^{\la} t^{-\frac{\la}{2}}  \| X\|_{\Cc^\al} \Big) \wedge \Big(  C  \ga^{\la}   \| X\|_{\Cc^{\al+ \la}}\Big)  \;, \label{e:Regular2}\\
\|  \Kg \star X \|_{\Cc^{\al-\ka}} \leq & \;C \; \| X\|_{\Cc^\al} \;, \label{e:Regular3}\\
\|  \Kg \star X- X \|_{\Cc^{\al-\ka}} \leq & \;C  \; \ga^{2\la}   \| X\|_{\Cc^{\al+2\la}} \; \label{e:Regular4}.
\end{align}
\item For every $\beta >0$, there exists $C = C(c_2,T,\ka, \be)$ such that for any distribution $X$ on $\T^2$ with $\hat{X}(\om) = 0$ for $|\om| \leq c_2 \ga^{-1}$, we have for all $0 < \ga < \ga_0$, $0 \leq t \leq T$ and $\al \in \R$
\begin{align}
\| \Pg{t}  X \|_{\Cc^{\al + \be -\ka }} \leq &  \;C  \;   t^{-\be}  \| X\|_{\Cc^\al} \; \label{e:Regular5}.
\end{align}
If furthermore $0 \leq \be \leq 2$, then
\begin{align}
\| \Pg{t} (\Kg \star X) \|_{\Cc^{\al + \be - \ka}} \leq & \; C\; t^{- \frac{\be}{2}} \| X\|_{\Cc^\al} \;.
\label{e:Regular6}
\end{align}
\end{enumerate}

\end{lemma}
%%%%%%%%%%%%%
%
\begin{remark}
The (arbitrarily small) loss of regularity $\ka$ in these estimates has two reasons. It is caused on the one hand by the logarithmic divergence that appears in the last line of the estimate \eqref{e:RegOp1} below. This could probably be removed 
by performing additional integrations by parts. A second (arbitrarily small) loss of regularity is caused by the fact that in \eqref{e:RegOp1}, we derive a bound on a discrete $L^1$ norm over $\Le$ rather than on the $L^1$ norm over $\T^2$. This allows us to avoid boundary terms in the integration by parts. 
\end{remark}

\begin{remark}
The approximate heat operator $\Pg{t}$ is not as regularising on high frequencies as $P_t$ is, as can be seen in the estimate \eqref{e:Regular5}. The usual rescaling is recovered when one also convolves with 
$\Kg$ as in \eqref{e:Regular6}, but only for $\be \leq 2$.
\end{remark}
%
%%%%%%%%%%%%%%
\begin{proof}
%%%%%%%%%%
We start by discussing the regularisation properties of an abstract Fourier multiplication operator. Suppose that $T$ is a Fourier multiplication operator with symbol $\hat{T} : \Z^2 \to \C$ vanishing for $\om \notin \{-N, \ldots, N \}^2$.  We aim to derive a bound on $\| \dk (T X) \|_{L^\infty(\T^2)} = \| T \dk X \|_{L^\infty(\T^2)}$ for $k$ satisfying $2^k  \frac{8}{3}\leq N$. By definition,  $\dk (T X) =0$ for all larger $k$.

In order to avoid boundary terms in \eqref{e:RegOp1} below, we will actually derive bounds on $\| T \dk X \|_{L^\infty(\Le)}$ and then use   Lemma~\ref{le:LinftyDiscreteCont} to conclude that
\begin{align*}
\| T \dk X \|_{L^\infty(\T^2)} \leq C(\ka) \eps^{-\frac{\ka}{4}} \| T \dk X \|_{L^\infty(\Le)} \leq C(\ka) 2^{\frac{\ka}{2} k }  \| T \dk X \|_{L^\infty(\Le)} \;,
\end{align*}
i.e. we encounter an arbitrarily small loss of regularity.

We start by defining auxiliary cutoff functions $\bar{\chi}_k$ for $k \geq -1$  as follows. Suppose that  $\bar{\chi} \colon \R^2 \to \R_+$ is smooth,  that it coincides with the constant function $1$ on $B(0,3) \setminus B(0,\frac12)$ and that $\bar{\chi}$ vanishes outside of the annulus $B(0,4) \setminus B(0,\frac13)$. For $k \in \N_0$, set $\bar{\chi}_k(\om) = \bar{\chi}(2^{-k} \om)$. Let $\bar{\chi}_{-1} \colon \R^2 \to \R_+$ be smooth and suppose that it coincides with $1$ on $B(0,3/2)$ and that it vanishes outside of $B(0,2)$. Finally, set $\bar{\eta}_k(x) = \frac{1}{4}\sum_{\om \in \Z^2} \bar{\chi}_k(\om) \, e^{i \pi \om \cdot x}$.

For any $k \geq -1$, we have according to Young's inequality
\begin{align*}
 \| T \dk X \|_{L^\infty(\Le)} &= \| T \ae \bar{\eta}_k \ae \dk X \|_{L^\infty(\Le)} \leq  \| T \ae \bar{\eta}_k \|_{L^1(\Le)} \| \dk X \|_{L^{\infty}(\Le)}\\
 & \leq   \| T \ae \bar{\eta}_k \|_{L^1(\Le)}  2^{-k \al} \| \dk X \|_{\Cc^{\al}}  \;.
\end{align*}
In this calculation, we identified $T$ with its integral kernel. In order to bound the $L^1$ norm appearing on the right-hand side of this estimate, we write for $k \geq 4$
\begin{align}
 \| &T \star \bar{\eta}_k \|_{L^1(\Le)}  \notag \\
  =& \sum_{x \in \Le} \eps^2 \Big|     \frac{1}{4} \sum_{\om \in \{ -N, \ldots, N \}^2}  \hat{T}(\om) \, \bar{\chi}(2^{-k} \om ) \, e^{i \pi \om \cdot x}     \Big| \,  \; \notag\\
=&  \sum_{x \in  2^k \Le} 2^{2k} \eps^2  \,   \Bigg|  \frac14 \sum_{\substack{\om \in 2^{-k} \Z^2  \\ |\om_j | \leq 2^{-k}N}}  2^{-2k} \,  \hat{T} (2^k \om) \, \bar{\chi}(\om) \, e^{i \pi \om \cdot x}    \Bigg| \, \notag \\
=& \sum_{x \in  2^k \Le} 2^{2k} \eps^2  \frac{1}{1+  2^{2k} \sum_{j=1,2}  2(1 - \cos(\pi 2^{-k} x_j))}\notag \\
&\qquad \qquad \qquad \times  \Bigg|  \frac14 \sum_{\substack{\om \in 2^{-k} \Z^2  \\ |\om_j | \leq 2^{-k}N}}    2^{-2k} \,  \hat{T} (2^k \om) \, \bar{\chi}(\om) \,(1- \Delk) e^{i \pi \om \cdot x}    \Bigg| \,  \; \notag\\
\leq& C \sum_{x \in  2^k \Le} 2^{2k} \eps^2  \frac{1}{1 + |x|^2} \Bigg|  \frac14 \sum_{\substack{\om \in 2^{-k} \Z^2  \\ |\om_j | \leq 2^{-k}N}}   2^{-2k} \, (1-\Delk) \Big(    \hat{T} (2^k \om) \, \bar{\chi}(\om)\Big) \,   \Bigg| \,  \; \notag\\
\leq& C \log(2^k) \sup_{\frac15\leq |\om| \leq 5} \Big| (1- \Delk) \Big(    \hat{T} (2^k \om) \, \bar{\chi}(\om)\Big) \Big|\; .\label{e:RegOp1}
\end{align}
This time, the discrete Laplacian is defined as 
\begin{equation*}
\Delk f (\om) =  2^{2k} \sum_{\substack{\bar{\om} \in 2^{-k} \Z^2 \\ |\om_j |  \leq 2^{-k} N  \\ \bar{\om} \sim \om}} (f(\bar{\om}) - f(\om))\;,
\end{equation*}
and the nearest neighbour relation is to be understood with periodic boundary conditions on $\{ \om \in 2^{-k} \Z^2 \colon |\om_j| \leq  2^{-k} N\}$. The integration by parts is also understood with periodic boundary conditions on this discrete torus. (This is the reason why we had to restrict the sum to $x \in \Le$ --- only for such $x$ is $\om \mapsto e^{i \pi \om \cdot x}$ an eigenfunction for $\Delk$).
In the fourth inequality in \eqref{e:RegOp1}, we have used the fact that $\bar{\chi}$ has compact support contained in $B(0,4) \setminus B(0,\frac{1}{3})$, which ensures that  for $k \geq 4$ the discrete Laplacian is supported in $B(0,5) \setminus B(0, \frac{1}{5})$.

In all of our applications, $\hat{T}$ is defined and smooth on all of $\R^2$, and not only on the grid points in $\Z^2$.  As in \eqref{e:dicsLaptocont}, we can then replace the local supremum over the discrete Laplacian by a local supremum over the
derivatives in the continuum. We get  for $k \geq 4$
\begin{align}
\sup_{\frac15\leq |\om| \leq 5}  &\Big| (1- \Delk) \Big(    \hat{T} (2^k \om) \, \bar{\chi}(\om)\Big)\Big| \notag\\
\leq&  C \sup_{\frac15\leq |\om| \leq 5}  \big| \hat{T}(2^k \om) \big| +  C \sup_{\substack{ \frac15\leq |\om| \leq 5\\ j =1,2}} \big| \partial_j \hat{T}(2^k \om)\big| +   C \sup_{\substack{ \frac15\leq |\om| \leq 5\\ j =1,2}} \big| \partial_j^2 \hat{T}(2^k \om)\big| \;. \label{e:RegOp1A}
\end{align}
In the case $k =-1, \ldots,3$, we get easily
\begin{equation}\label{e:kleqthree}
 \| T \star \bar{\eta}_{k} \|_{L^1(\T^2)} \leq  \int_{\T^2} \Big|     \frac{1}{4} \sum_{\om \in \{ -N, \ldots, N \}^2}  \hat{T}(\om) \, \bar{\chi}_{k}( \om )    \Big| \, dx \leq C \sup_{|\om | \leq 32} |T(\om)| \;.
\end{equation}

We are now ready to apply these bounds to the operators of interest. The bounds on $\dk X$ for $k \leq 3$ follow from \eqref{e:kleqthree}, so from now on we assume that $k \geq4$.  For $\Pg{t}$ we have for  $0 < \ga < \ga_0$, for $|\om| \leq \ga^{-1}$, and $j =1,2$ 
\begin{align}
\hPg{t}(\om) &= \exp\Big(-  \frac{t}{\ga^2}\big( 1 - \hKg(\om)\big) \Big) \leq \exp \Big(- \frac{t|\om|^2}{C_1} \Big)\;, \notag \\
\partial_j \hPg{t}(\om) &= \hPg{t}(\om)  \;\frac{t}{\ga^2}\partial_j \hKg(\om)  \leq C \sqrt{t}  \exp \Big(- \frac{t|\om|^2}{C_1} \Big) (\sqrt{t}|\om|)  \;, \notag \\
\partial_j^2 \hPg{t}( \om) &= \hPg{t}(\om) \Big(  \frac{t^2}{\ga^4} \big(\partial_j \hKg(\om)\big)^2 + \frac{t}{\ga^2} \partial^2_j \hKg(\om)  \Big)\notag\\
& \qquad \qquad \qquad\leq Ct \exp \Big(- \frac{t|\om|^2}{C_1} \Big) \Big(t |\om|^2 +1 \Big)\;, \label{e:RegOp1B}
\end{align}
where we have used the bounds on $\hKg$ and its derivatives stated in Lemmas~\ref{le:Kg0} and \ref{le:Kg}. The first of these bounds immediately yields, for any $\beta \geq 0$,
\begin{equation*}
 \sup_{\frac15\leq |\om| \leq 5}    \big| \hPg{t}(2^k \om) \big| \leq  \exp \Big(- \frac{t}{ 25 C_1} 2^{2k} \Big) \leq \big( 2^{-k} t^{-\frac12}\big)^\beta  \sup_{x > 0} x^{\beta} e^{-\frac{x^2}{25 C_1}} \;.
\end{equation*}
Similar bounds on the first and second derivatives of $\hPg{t}$ follow in the same way from the remaining estimates in \eqref{e:RegOp1B}. Indeed, both of these bounds behave slightly better for small $t$ -- we gain a factor $t^{\frac12}$ in the bound for the first derivative, and a factor $t$ in the bound for the second derivative. The desired bound  \eqref{e:Regular1} then follows.

To get bounds on $\Pg{t} - P_t$, we write
\begin{align*}
\Big| \hPg{t}(\om) - \hat{P}_t(\om) \Big|  &= \Big| \exp\Big(-  \frac{t}{\ga^2}\big( 1 - \hKg(\om)\big) \Big) - \exp(- t \pi^2 |\om|^2 ) \Big|\\
&\leq \exp \Big(- \frac{t|\om|^2}{C_1} \Big)\;\ \Big| \frac{t}{\ga^2} \big( 1 - \hKg(\om)\big)  - t\pi^2 |\om|^2  \Big| \\
&\leq C  \exp \Big(- \frac{t|\om|^2}{C_1} \Big)\; t |\om|^2  \ga |\om| \leq C(\la) \ga^{\la} \big( t^{-\frac{\la}{2}} \wedge |\om|^\la \big)  \;,
\end{align*}
the last inequality being valid for any $0 \leq \la \leq 1$. In the third inequality we have made use of \eqref{e:K2.4}. For the derivatives, we write
\begin{align*}
\Big| &\partial_j \hPg{t}(\om) - \partial_j \hat{P}_t(\om) \Big| \\
&\leq 2 \pi^2 t |\om|  \;\Big| \hPg{t}(\om) - \hat{P}_t(\om) \Big| + \hat{P}_t(\om)\Big| -\frac{t}{\ga^2} \partial_j \hKg(\om) - 2 t \pi^2 \om_j \Big| \\
&\leq  C  \ga + Ce^{-\frac{t |\om|^2}{C_1}}  t  \ga|\om|^2 \leq C \ga\;.
\end{align*}
Here, in the second inequality we have used \eqref{e:K2.3}. Note that as above in the bound for $\partial_j \hPg{t}$ we gain a factor $\sqrt{t}$ with respect to the bound for $\hPg{t}$.

Finally, for the second derivatives we have using \eqref{e:K2.3} and \eqref{e:K2.2}
\begin{align*}
\Big| &\partial_j^2 \hPg{t}(\om) - \partial_j^2 \hat{P}_t(\om) \Big| \\
&\leq |4 \pi^4 t^2\, \om_j^2 - 2 \pi^2 t | \;\Big| \hPg{t}(\om) - \hat{P}_t(\om) \Big| \\
& \qquad \qquad + \hPg{t}(\om)\Big(\Big| \frac{t^2}{\ga^4}  \big(\partial_j \hKg(\om) \big)^2- 4 \pi^4 t^2 \om_j^2 \Big| + \Big|  \frac{t}{\ga^2}\partial_j^2 \hKg(\om) +  2 \pi^2 t\Big|\Big)  \\
& \leq C t^{\frac12} \ga \;.
\end{align*}
We can now deduce \eqref{e:Regular2} by repeating the argument below \eqref{e:RegOp1B}. 

The bounds \eqref{e:Regular3} and \eqref{e:Regular4} follow immediately from \eqref{e:K1}, \eqref{e:K1A}, and \eqref{e:K1B} as well as the bound
\begin{align*}
|\hKg(\om) -1 | \leq C \ga^2 |\om|^2 \;,
\end{align*}
uniformly in $|\om| \leq \ga^{-1}$, which is an immediate consequence of \eqref{e:K2.4}.

We now pass to the bounds for $|\om| \geq \ga^{-1}$. In order to get   \eqref{e:Regular5}, we write for $|\om| \geq \ga^{-1}$ with $\om \in \{ -N, \ldots, N\}^2$
\begin{align*}
\hPg{t}(\om) &= \exp\Big(-  \frac{t}{\ga^2}\big( 1 - \hKg(\om)\big) \Big) \leq \exp \Big(- \frac{t }{C_1 \ga^2} \Big)\leq \exp\Big(- \frac{t |\om|}{ 2 C_1} \Big) \\
\partial_j \hPg{t}(\om) &= \hPg{t}(\om)  \;\frac{t}{\ga^2}\partial_j \hKg(\om)  \leq C\ga \exp \Big(- \frac{t }{C_1\ga^2}\Big) \frac{t}{\ga^2} \;, \\
\partial_j^2 \hPg{t}( \om) &= \hPg{t}(\om) \Big(  \frac{t^2}{\ga^4} \big( \partial_j \hKg(\om) \big)^2+ \frac{t}{\ga^2} \partial_j^2 \hKg(\om)  \Big) \\
& \leq C\ga^2 \exp\Big(- \frac{t }{ C_1\ga^2}\Big) \Big(\frac{t^2}{\ga^4} + \frac{t}{\ga^2} \Big)\;.
\end{align*}
In this case, we write for $\beta \geq 0$
\begin{equation*}
 \sup_{\frac15\leq |\om| \leq 5}    \big| \hPg{t}(2^k \om) \big| \leq  \exp \Big(- \frac{t}{ 25 C_1} 2^{k} \Big) \leq \big( 2^{-k} t^{-1}\big)^\beta  \sup_{x > 0} x^{\beta} e^{-\frac{x}{25 C_1}} \;.
\end{equation*}
and as before similar bounds on the first and second derivatives of $\hPg{t}$ follow in the same way, so that \eqref{e:Regular5} follows. Observe the exponent $t^{-\be}$ which  yields a weaker bound than the 
$t^{-\frac{\be}{2}}$ one gets in the case of the heat operator $P_t$.

Finally, for $\Pg{t} \Kg$ we write for $ |\om | \geq \ga^{-1}$
\begin{align*}
\big| \hPg{t}(\om) \hKg(\om)\big| \leq \exp \Big(- \frac{t }{C_1 \ga^2} \Big) \frac{C}{\ga^2 |\om|^2} \leq C(\la) \frac{1}{t^\la |\om|^{2\la}} \;, 
\end{align*}
for any $0 \leq \la \leq 1$. For the derivatives, we get
\begin{align*}
\big| \partial_j& \hPg{t}(\om)  \;\hKg(\om)\big|  + \big| \hPg{t}(\om) \partial_j \hKg(\om)\big| \\
&\leq C\ga \exp \Big(- \frac{t }{C_1\ga^2}\Big) \frac{t}{\ga^2} \frac{1}{\ga^2 |\om|^2} + C \exp \Big(- \frac{t }{C_1\ga^2}\Big)  \frac{\ga}{\ga^2 |\om|^2}\leq C(\la) \frac{\ga}{t^\la |\om|^{2\la}} \;,
\end{align*}
and
\begin{align*}
\Big| \partial_j^2 & \big(  e^{\frac{t}{\ga^2} (\hKg - 1)}  \;  \hKg  \big) \Big| \notag \\
&\leq e^{ \frac{t}{C_1\ga^2}}   \Big( \frac{t^2}{\ga^4} \big( \partial_j \hKg \big)^2\hKg  +  \frac{t}{\ga^2}  \big(  \partial_j^2 \hKg \big)  \hKg + \partial_j^2 \hKg + 2 \frac{t}{\ga^2} (\partial_j \hKg)^2   \Big) \\
&\leq C   e^{ \frac{t}{C_1\ga^2}}   \Big(\frac{t^2}{\ga^4}  \frac{\ga^2}{\ga^6 |\om|^6} + \frac{t}{\ga^2} \frac{\ga^2}{\ga^6 |\om|^6} +\frac{\ga^2}{\ga^2|\om|^2} +  2 \frac{t}{\ga^2} \frac{\ga^2}{\ga^4|\om|^4} \Big)\\
&\leq C \frac{\ga^2}{t^\la |\om|^{2\la}}\;.
\end{align*}
Hence \eqref{e:Regular6} follows as well.
%%%%%%%%%%%
\end{proof}
%%%%%%%%%%%%

%%%%%%%%%%%
\begin{corollary}
%%%%%%%%%%%
\label{cor:regPg}
%%%%%%%%%%%
Let $\ga_0$ be the constant appearing in Lemma~\ref{le:Kg}. For any $T>0$, $\ka > 0$,  $\be>0$ and $0 \leq \la \leq 1$, there exists a constant $C = C(T,\ka,\be,\la)$ such that for every $0 < \ga < \ga_0$, $0 \leq t \leq T$, $\al \in \R$ and distribution $X$ on $\T^2$,
\begin{align}
\| (\Pg{t} - P_t) X  \|_{\Cc^{\al -\ka}} \leq   \Big( C  \;  \ga^{\la} t^{-\la}  \| X\|_{\Cc^\al} \Big) \wedge \Big(  C  \ga^{\la}   \| X\|_{\Cc^{\al+ \la}}\Big) \;.\label{e:reg_Cor}
\end{align}
\begin{align}
\| (\Pg{t}\Kg - P_t) X  \|_{\Cc^{\al -\ka}} \leq   \Big( C  \;  \ga^{\la} t^{-\frac{\la}{2}}  \| X\|_{\Cc^\al} \Big) \wedge \Big(  C  \ga^{\la}   \| X\|_{\Cc^{\al+\la}}\Big) \;.\label{e:reg_Cor2}
\end{align}

%%%%
\end{corollary}
%%%%%
\begin{proof}
To see \eqref{e:reg_Cor}, we split $X$ into its Littlewood-Paley blocks $\dk X$, separating those with $2^{k} \leq \ga^{-1}$ from those with $2^k > \ga^{-1}$. For the blocks with $2^{k } \leq \ga^{-1}$, the necessary bound is stated in \eqref{e:Regular2}. Indeed, this bound is 
even better, because it has the term $t^{-\la}$ replaced by $t^{-\frac{\la}{2}}$.

For the blocks with $2^{k} > \ga^{-1}$, we bound brutally
\begin{align}
\| \dk  (\Pg{t} - P_t) X\|_{L^\infty(\T^2) }& \leq \|   \Pg{t}  \dk X  \|_{L^\infty(\T^2) } + \|   P_t  \dk X\|_{L^\infty(\T^2) } \notag \\ 
& \leq 2^{-k (\al + \la)} \Big( \|   \Pg{t}  \dk X  \|_{\Cc^{\al+\la} } + \|   P_t  \dk X\|_{\Cc^{\al+\la} }  \Big) \notag\\
& \leq 2^{-\al k}  \ga^{\la} \Big( \|   \Pg{t}  \dk X  \|_{\Cc^{\al+\la} } + \|   P_t  \dk X\|_{\Cc^{\al+\la} }  \Big)\;. \label{e:P_Cor1}
\end{align}
Then the bound \eqref{e:reg_Cor} follows from \eqref{e:Regular5}, and this is where we pick up the term $t^{-\la}$.

To get \eqref{e:reg_Cor2}, we argue in a similar way. For $\dk X$  with $2^{k} \leq \ga^{-1}$, we bound
\begin{align*}
\| \dk  (\Pg{t} &\Kg - P_t) X\|_{L^\infty(\T^2) } \\
&\leq \| \dk  (\Pg{t} (\Kg - 1)X)\|_{L^\infty(\T^2) } + \| \dk  (\Pg{t} - P_t ) X\|_{L^\infty(\T^2) } \; .
\end{align*}
The bound then follows from a combination of \eqref{e:Regular1}, \eqref{e:Regular2}, and \eqref{e:Regular4}. For $2^k > \ga^{-1}$, we argue as in \eqref{e:P_Cor1} with \eqref{e:Regular5} replaced by \eqref{e:Regular6}.
\end{proof}

\appendix
%%%%%%%%%%%%%%%%%%%%%%%%
\section{Besov spaces}
%%%%%%%%%%%%%%%%%%%%%%%%
\label{sec:Besov}

The aim of this appendix is to recall some facts about Besov spaces on the torus $\T^n$, which we identify with $[-1,1]^n$. We begin with a technical lemma.

\begin{lemma}
\label{l:Bernstein}
Let $\chi$ be a smooth function with compact support. For every $\eps > 0$, let
$$
\eta_{\eps}(x) = \sum_{\om \in \Z^n} e^{i \pi \om \cdot x} \, \chi(\eps \om) \qquad (x \in \T^n).
$$
For every $p,p' \in [1,+\infty]$ with $1/p + 1/p' = 1$, we have
$$
\sup_{0 < \eps \le 1} \eps^{n/p'} \|\eta_{\eps}\|_{L^p(\T^n)} < + \infty.
$$
\end{lemma}
\begin{remark}
This lemma would follow from very simple scaling arguments if we were considering the continuous Fourier transform instead of the discrete Fourier series. 
\end{remark}
\begin{proof}
Since $\|\eta_{\eps}\|_{L^p(\T^n)}^p \le \|\eta_{\eps}\|_{L^\infty(\T^n)}^{p-1} \|\eta_{\eps}\|_{L^1(\T^n)}$, it suffices to prove the lemma for $p = 1$ and $p = \infty$. Since $\chi$ is bounded and compactly supported, the result is clear for $p = \infty$. In order to cover the case $p = 1$, we define
$$
\check{\chi}_\eps(x) = \eps^n \sum_{\om \in \eps \Z^n} e^{i \pi \om \cdot x} \, \chi(\om),
$$
so that $\eps^{n} \eta_\eps(\eps x) =\check{\chi}_\eps(x)$. It is then clear that $\check{\chi}_\eps$ is bounded uniformly over $\eps$ and that $\|\eta_\eps\|_{L^1(\T^n)} = \|\check{\chi}_\eps\|_{L^1(\T^n/\eps)}$. In order to conclude the proof, it thus suffices to check that $\check{\chi}_\eps$ decays uniformly fast enough at infinity.
In order to do so, we introduce
$$
\Del f (\om) = \eps^{-2} \sum_{\substack{\om' \in \eps \Z^n \\ \om' \sim \om}} (f(\om') - f(\om)),
$$
where $\om' \sim \om$ is the usual nearest-neighbour relationship in $\eps \Z^n$. Writing $e_x : \om \mapsto e^{i x\cdot \om}$, we observe that
\begin{equation}
\label{e:IBPd1}
\Del e_x = 2 \eps^{-2} \sum_{i = 1}^n(\cos(\eps \pi x_i) - 1) e_x.
\end{equation}
For every positive integer $k$, repeated integration by parts yields
\begin{equation}
\label{e:IBPd2}
\sum_{\om \in \eps \Z^n} \Del^k e_x(\om) \chi(\om) = \sum_{\om \in \eps \Z^n}  e_x(\om) \Del^k \chi(\om).
\end{equation}
Since $\chi$ is smooth and compactly supported, $|\Del^k \chi|$ is bounded uniformly over $\eps$, and thus
$$
\Ll|\eps^{n} \sum_{\om \in \eps \Z^n}  e_x(\om) \Del^k \chi(\om)\Rr| \le C_k
$$
for some constant $C_k$. Combining this with \eqref{e:IBPd1} and \eqref{e:IBPd2}, we obtain that
$$
\Ll|\check{\chi}_\eps(x)\Rr| \le \frac{C_k}{\Ll[\eps^{-2} \sum_{i = 1}^n(1 - \cos(\eps \pi x_i))\Rr]^k}.
$$
Hence, there exists $C_k'$ such that uniformly over $\eps \in (0,1]$ and $x \in \T^n/\eps$,
\begin{equation}
\label{e:IBPdconcl}
\Ll|\check{\chi}_\eps(x)\Rr| \le \frac{C_k'}{|x|^{2k}},
\end{equation}
and this is sufficient to conclude the proof.
\end{proof}

We learn from \cite[Proposition~2.10]{BCD} that one can find smooth rotationally invariant functions $\td{\chi}$, $\chi$ taking values in $[0,1]$ and such that
\begin{equation}
\label{chi-prop1}
\textrm{supp} \, \td{\chi} \subset B(0,4/3),
\end{equation}
\begin{equation}
\label{chi-prop2}
\textrm{supp} \, {\chi} \subset B(0,8/3) \setminus B(0,3/4),
\end{equation}
\begin{equation}
\label{chi-prop3}
\forall \om \in \R^d, \ \td{\chi}(\om) + \sum_{k = 0}^{+\infty} \chi(\om/2^k) = 1,
\end{equation}
where we write $\textrm{supp}\, f$ to denote the support of $f$, and $B(0,r)$ to denote the Euclidean ball of radius $r$. Recall that any distribution $Z$ on $\T^n$ can be written (see e.g.\ \cite[Proposition~2.12]{BCD}) as the Fourier series
$$
Z(x) = \frac{1}{2^n} \sum_{\om \in \Z^d} \hat{Z}(\om) e^{i\pi \om \cdot x},
$$
where 
$$
\hat{Z}(\om) = \int_{\T^d} Z(x) \; e^{-i\pi \om \cdot x} \, dx;,
$$
(this expression has to be interpreted as testing the distribution $Z$ against the test function $x \mapsto e^{- i\pi  \om \cdot x}$). We use the partition of unity provided by \eqref{chi-prop3} to decompose any distribution $u$ over $\T^n$ as a sum of (smooth) functions with localized spectrum. More precisely, we let
\begin{equation}
\label{e:def:chik}
\chi_{-1} = \td{\chi}, \qquad \chi_k = \chi(\cdot/2^k) \quad (k \ge 0),
\end{equation}
$$
\delta_k Z(x) = \frac{1}{2^n} \sum_{\om \in \Z^d} \chi_k(\om)\hat{Z}(\om) e^{i\pi \om \cdot x} \qquad (k \ge -1),
$$
so that
$$
Z = \sum_{k = -1}^{+\infty} \delta_k Z.
$$
For any $\al \in \R$, and $1 \leq p,q \le \infty$ we define 
\begin{equation*}
\| Z \|_{\Bb^\al_{p,q}}:= \Ll( \sum_{k = -1}^{\infty} \Ll( 2^{k \al } \| \dk Z \|_{L^p(\T^n)} \Rr)^q \Rr)^{\frac{1}{q}}\; = \Ll\| \Ll( 2^{k \al } \| \dk Z \|_{L^p(\T^n)} \Rr)_{k \ge -1} \Rr\|_{\ell^q}
\end{equation*}
with the usual interpretation (as a supremum) for $q = \infty$. We define the Besov space $\Bb^\al_{p,q}$ as the closure of $\Cc^\infty(\T^n)$  with respect to this norm. Note that if $p,q = \infty$, this deviates slightly from the 
standard definition to take all distributions for which this norm is finite. Our choice to take a slightly smaller space has the advantage to yield separable spaces.

The operator $\dk$ is best thought of as a convolution operator. Letting 
\begin{equation}\label{e:DEFek}
\eta_k(x) = \frac{1}{2^n} \sum_{\om \in \Z^d} \chi_k(\om) \, e^{i\pi \om \cdot x} \qquad (k \ge -1),
\end{equation}
we observe that
\begin{equation}
\dk Z = \eta_k \star Z\;, \label{e:dkConvProp}
\end{equation}
where $\star$ denotes the convolution in $\T^n$.

%%%%%%%%%%%%%%%%%%%%%%%
\begin{lemma}
%%%%%%%%%%%%%%%%%%%%%%%
\label{l:unif-Fourier-cut}
%%%%%%%%%%%%%%%%%%%%%%%
For every $\al > 0$, there exists $C = C(\al)$ such that for any $R \ge 1$, if $\hat{Z}(\om) = 0$ for all $\om \in \Z^n$ such that $|\om| > R$, then 
$$
\|Z\|_{L^\infty(\T^n)} \le CR^\al \|Z\|_{\Ca}.
$$
\end{lemma}
%%%%%%%%%%%%%%%%%%%%%%%%
%
%%%%%%%%%%%%%%%%%%%%%%%%
\begin{proof}
By definition of the Besov norm
\begin{equation*}
\|\delta_k Z\|_{L^\infty(\T^n)} \le 2^{k \al} \|Z\|_{\Ca}.
\end{equation*}
Let $l$ be the smallest integer such that $2^l > R$. The assumption on $Z$ ensures that
$$
Z = \sum_{k = 0}^l \delta_k Z,
$$
hence
$$
\|Z\|_{L^\infty(\T^n)}  \le  \sum_{k = 0}^l \|\delta_k Z\|_{L^\infty(\T^n)}
 \le  \|Z\|_{\Ca} \sum_{k = 0}^l 2^{k \al} \le C 2^{l\al} \|Z\|_{\Ca}.
$$
We get the result by observing that $2^{l} \le 2R$.
\end{proof}
%%%%%%%%%%%%%%%%%%%%%%%%

The following \emph{Kolmogorov-like} lemma is used for deriving a priori bounds in Section~\ref{sec:linear}. 

%%%%%%%%%%%%%%%%%%%%%%%
\begin{proposition}\label{prop:Kolmogorov} 
%%%%%%%%%%%%%%%%%%%%%%%
Let $r \mapsto Z(r,\cdot)$ be a cadlag process taking values in $\Cc^{\infty}(\T^n)$.   For any $\al \in \R$, $1 < p < \infty$ and $\al_{\star} < \al - \frac{n}{p}$, there exists a constant $C = C(p,n, \al,\al_{\star})$ such that  
\begin{equation*}
\E  \sup_{0 \leq r < \infty} \| Z(r,\cdot)\|_{\Cc^{\al_{\star}}}^p \leq C \,  \sup_{k \geq 0} \sup_{x \in \T^n} 2^{k \al p } \,  \E \sup_{0 \leq r < \infty}\big|\dk Z(r,x) \big|^p \; .
\end{equation*}
\end{proposition}

%%%%%%%%%%%%%%%%%%%%%%%%%
%
%%%%%%%%%%%%%%%%%%%%%%%%%
\begin{proof} 
Let $\bar{\chi}\colon \R^n \to \R$ be a smooth non-negative function which is identical to $1$ on the unit ball $B(0,3)$ and with compact support contained in $B(0,4)$. For $k \geq -1$, set
\begin{equation*}
\bar{\eta}_k(x) =  \frac{1}{2^n} \sum_{\om \in \Z^n} e^{i \pi \om \cdot x} \, \bar{\chi}( 2^{-k}\om) \qquad (x \in \T^n).
\end{equation*}
For any $k \geq -1$, we have  $\dk Z \star \bar{\eta}_k = \dk Z$, and hence by H\"older's inequality we get that for any $0 \leq r < \infty$,
\begin{equation*}
\sup_{x \in \T^n} |  \dk Z(r,x) | = \sup_{x \in \T^n } \Big| \int_{\T^n}\dk Z (r,x-y) \,\bar{\eta}_k (y) dy \Big| \leq \|\dk Z(r, \cdot) \|_{L^p(\T^n)} \, \| \bar{\eta}_k \|_{L^{p'}(\T^n)}\;,
\end{equation*} 
where $\frac{1}{p} + \frac{1}{p'} =1$. Lemma~\ref{l:Bernstein} implies that
\begin{equation*}
\| \bar{\eta}_k \|_{L^{p'}(\T^n)} \leq C(p,n) 2^{ \frac{kn}{p}}\;,
\end{equation*}
so that
\begin{equation*}
\| \dk Z(r, \cdot) \|_{L^{\infty}(\T^n)}^p    \leq C(p,n)  \,2^{kn} \, \| \dk Z(r, \cdot)\|_{L^{p}(\T^n)}^p\;.
\end{equation*}
Using this estimate, we get
\begin{align*}
\sup_{0 \leq r < \infty}& \sup_{k \geq -1} 2^{ k p \al_{\star}  } \| \dk Z(r, \cdot) \|_{L^{\infty}(\T^n)}^p \\
 &\leq \sup_{0 \leq r < \infty} \sum_{k \geq -1} 2^{k p \al_{\star}  } \| \dk Z(r, \cdot) \|_{L^{\infty}(\T^n)}^p \\
&\leq C(p,n) \sum_{k \geq -1} 2^{k p (\al_{\star}   +\frac{ n}{p})}\sup_{0 \leq r < \infty} \| \dk Z(r, \cdot)\|_{L^{p}(\T^n)}^p\;. 
\end{align*}
Therefore,
\begin{align*}
\E \sup_{0 \leq r < \infty} \| Z(r,\cdot) \|_{\Cc^{\al_\star}}^p &\leq C(p,n) \sum_{k \geq 0} 2^{k p (\al_{\star}   +\frac{ n}{p})} \int_{\T^n } \E\sup_{0 \leq r < \infty} | \dk Z(r,x) |^p  \, dx \\
& \leq C(p,n) \sum_{k \geq 0} 2^{ kp  (\al_{\star} +  \frac{n}{p} -\al)}  \sup_{x \in \T^n}  \sup_{k \geq 0} 2^{ k\al p }\,  \E \sup_{0 \leq r < \infty}\big|\dk Z(r,x)\big|^p  \\
& \leq C(p,n,\al,\al_{\star})  \sup_{x \in \T^n} \sup_{k \geq 0} 2^{ k\al p } \, \E\sup_{0 \leq r < \infty}\big|\dk Z(r,x)\big|^p  \; .
\end{align*}
\end{proof}

In Section~\ref{sec:Nonlinear}, we make use of the following multiplicative inequality.
\begin{lemma}\label{le:Besov-multiplicative}
Let $\beta < 0 < \al$ with $\al + \beta >0$. Then the mapping $(Z_1, Z_2) \mapsto Z_1 Z_2$ defined for $Z_1, Z_2 \in \Cc^\infty$ extends uniquely to a continuous bilinear mapping from $\Cc^\al \times \Cc^\beta \to \Cc^\beta$.
That is, there exists a constant $C$ depending only on $\al$ and $\beta$ such that 
\begin{equation*}
\| Z_1 \, Z_2 \|_{\Cc^\beta} \leq C \|Z_1 \|_{\Cc^\al} \, \| Z_2\|_{\Cc^\beta}\;. 
\end{equation*}
\end{lemma}
The proof of this statement in the case of the full space can be found in Sections 2.6 and 2.5 of \cite{BCD} (see in particular Theorem 2.85). The argument on the torus is exactly the same, so we do not replicate it here.   

Finally, we provide a bound on the $L^{\infty}$ norm of a function defined by extension from a grid by trigonometric polynomials.

%%%%%%%%%%%%%%%%%%%
\begin{lemma}
%%%%%%%%%%%%%%%%%%%
\label{le:LinftyDiscreteCont}
%%%%%%%%%%%%%%%%%%%
For $N \in \N$, set $\eps = \frac{2}{2N+1}$ and let $\Le = \{ x \in \eps \Z^n \colon -1 < x_i < 1 \, \text{ for } i = 1, \ldots, n   \}$. For any $Z \colon \Le \to \R$, we define the extension 
\begin{align*}
\Ex \,Z(x) = 
  \sum_{z \in \Le}  \frac{\eps^n}{2^n} \; Z(z) \prod_{j=1}^n \frac{\sin\Big(\frac{\pi}{2} (2N+ 1)   (x_j-z_j)  \Big)}{ \sin\big(    \frac{\pi}{2}  (x_j-z_j) \big)}  \;,
\end{align*}
defined for all $x \in \T^n$. Then for any $\ka >0$, there exists a constant $C=C(\ka,n)$ such that 
\begin{equation*}
\big\|  \Ex \, Z \big\|_{L^\infty(\T^n)} \leq C \, \eps^{-\ka} \big\|  Z \big\|_{L^\infty(\Le)}\;.
\end{equation*}
\end{lemma}
%%%%%%%%%%%%%%%%%%%
%
%
%%%%%%%%%%%%%%%%%%
\begin{proof}
%%%%%%%%%%%%%%%%%%
For $p >1$ we get from Young's inequality that
\begin{align*}
\sum_{z \in \Le} & \frac{\eps^n}{2^n} \prod_{j=1}^n  \Bigg| \frac{\sin\Big(\frac{\pi}{2} (2N+ 1)   (x_j-z_j)  \Big)}{ \sin\big(    \frac{\pi}{2}  (x_j-z_j) \big)}  \Bigg| \\
&\leq    \,\; \Bigg( \sum_{z \in \Le}  \frac{\eps^n}{2^n} \prod_{j=1}^n  \Bigg| \frac{\sin\Big(\frac{\pi}{2} (2N+ 1)   (x_j-z_j)  \Big)}{\sin\big(    \frac{\pi}{2}  (x_j-z_j) \big)}  \Bigg|^p \Bigg)^{\frac{1}{p}} \\
&=    \,\; (2N+1)^{\frac{n(p-1)}{p}}  \Bigg( \sum_{z \in \Le} \prod_{j=1}^n   \Bigg| \frac{\sin\Big(\frac{\pi}{2} (2N+ 1)   (x_j-z_j)  \Big)}{(2N+ 1) \sin\big(    \frac{\pi}{2}  (x_j-z_j) \big)}  \Bigg|^p \Bigg)^{\frac{1}{p}}\\
&\leq C(n)  \;\eps^{- \frac{n(p-1)}{p}} \sum_{\tilde{z} \in \{ -N, \ldots ,N\}^n } \big(1 \wedge |\tilde{z}|^{-n p} \big) \;.
\end{align*}
Here in the last line we have made use of the index change $\tilde{z_i} =  \lfloor 2(N+1)(x_i-z_i) \rfloor  $ (interpreted with periodic boundary condition on $\{-N, \ldots, N\}$). Now by choosing $p= \frac{n}{n-\ka}>1$ we obtain the desired result.
%
%
%%%%%%%%%%%%%%%%%
\end{proof}
%%%%%%%%%%%%%%%%

%

%%%%%%%%%%%%%%%%%%%%%%%%
\section{Martingales and an It\^o formula with jumps}
%%%%%%%%%%%%%%%%%%%%%%%%
\label{AppA}
Let $M = (M(t))_{t \ge 0}$ be a cadlag square-integrable martingale (with $M(0) = 0$). Its \emph{predictable quadratic variation} $\langle M \rangle_t$ is the unique increasing (in the wide sense) process such that $M^2(t) - \langle M \rangle_t$ is a martingale. Its \emph{bracket process} is
$$
[M]_t = M^2(t) - 2 \int M({s^{-}}) \, d M(s),
$$
where $M(s^{-})$ is the left limit of $M$ at time $s$. A convenient way to think about these processes consists in  noting that the bracket process is the limit in probability of
$$
\sum_{i = 0}^{n-1} \Ll(M({t_{i+1}}) - M({t_i})\Rr)^2
$$
as the subdivision $0 = t_0 \le \cdots \le t_n = t$ gets finer and finer (\cite[Theorem~4.47]{jacshi}), while the predictable quadratic variation is the limit of
$$
\sum_{i = 0}^{n-1} \E\Ll[\Ll(M({t_{i+1}}) - M({t_i})\Rr)^2 \ \vert \ \mathscr{F}_{t_i} \Rr],
$$ 
where $(\mathscr{F}_{t})_{t \ge 0}$ is the underlying filtration (\cite[Proposition~4.50]{jacshi}).
In particular, if $M$ is of finite total variation, then the bracket process is simply the sum of the jumps squared:
$$
[M]_t = \sum_{0 < s \le t} \Ll(\Delta_s M\Rr)^2,
$$
where $\Delta_s M = M(s) - M({s^-})$.
If $M$ has continuous sample paths, then the predictable quadratic variation and the bracket process coincide (\cite[Theorem~4.47]{jacshi}). In every case, $[M]_t - \langle M \rangle_t$ is a martingale (\cite[Proposition~4.50]{jacshi}). 

Let $(\sigma(t))_{t \ge 0}$ be a Feller process. Denote by $\L$ its infinitesimal generator, with domain $\mathcal{D}(\L)$. For $f \in \mathcal{D}(\L)$,  the process 
$$
M_f(t) := f(\sigma(t)) - f(\sigma(0)) - \int_0^t \L f (\sigma(s)) \, d s
$$
is a martingale. Moreover, if $f^2 \in \mathcal{D}(\L)$, then its predictable quadratic variation is given by
$$
\langle M_f \rangle_t = \int_0^t \Gamma_f(\sigma(s)) \, ds,
$$
where $\Gamma_f = \L(f^2) - 2 f \L f$ is the so-called ``carr\'e du champ'' (see for instance \cite[Appendix B]{berry}). 
When $\L$ is the operator defined in \eqref{e:Generator}, we have
$$
\Gamma_f (\sigma)= \sum_{j \in \LN} \cg(\si,j) \big(f(\si^j) -f(\si) \big)^2. 
$$

The Burkholder-Davis-Gundy inequality below plays a key role in our analysis.

\begin{lemma}[Burkholder-Davis-Gundy inequality]
Let $M$ be a cadlag square integrable martingale. For any $p>0$, there exists a constant $C=C(p)$ such that for all $T >0$,
\begin{equation*}
\E \bigg[  \sup_{0 \leq t \leq T} |M(t)|^p \bigg] \leq C \bigg(\E \Big[  \langle M \rangle_t^{\frac{p}{2}} \Big]  + \E \bigg[  \sup_{0 \leq t \leq T}  |\Delta_t M  |^p  \bigg]   \bigg)\;, 
\end{equation*}
where we recall that $\Delta_t M := M(t) - M({t^-})$ denotes the jump of $M$ at $t$.  
\label{le:BDG} 
\end{lemma}
\begin{proof}
The case of discrete-time martingales is covered by \cite[Theorem 2.11]{HaHe}. The continuous-time case can be recovered by approximation. Alternatively, one can consult for instance \cite{lelepr}.
\end{proof}

The following result, which can be found for instance in \cite[Theorem~II.7.31]{protter}, is akin to the fundamental theorem of calculus, but with integrands that may have jumps. 
\begin{lemma}
\label{l:ito-with-jumps}
Let $X = (X^{(1)},\ldots,X^{(n)}) : \R_+ \to \R^n$ be a cadlag process such that for every $i$, $X^{(i)}$ is of bounded total variation. Let $f : \R^n \to \R$ be a continuously differentiable function. Then $f(X) : \R \to \R$ is of bounded total variation and
\begin{multline}
\label{e:ito-with-jumps}
f(X(t)) - f(X(0)) \\
= \sum_{i = 1}^n \int_{s = 0}^t \frac{\partial f}{\partial x_i} (X({s^-})) \, dX^{(i)}(s) + \sum_{0 < s \le t} \Ll\{\Delta_s f(X) - \sum_{i = 1}^n \frac{\partial f}{\partial x_i} (X({s^-})) \, \Delta_s X^{(i)}\Rr\},
\end{multline}
where we recall our notation $\Delta_s F = F(s) - F(s^-)$ and $\int_{s = 0}^t = \int_{s \in (0,t]}$.
\end{lemma}

\begin{remark}
\label{r:ito-with-jumps}
In the setting of Lemma~\ref{l:ito-with-jumps}, let us assume for simplicity that $n = 1$. Recall that the bracket process of $X$ is then given by
$$
[X]_t = \sum_{0 < s \le t} \Ll( \Delta_s X \Rr)^2.
$$
If $f$ is twice continuously differentiable, we can rewrite \eqref{e:ito-with-jumps} as
\begin{multline*}
f(X(t)) - f(X(0)) 
= \int_{s = 0}^t f' (X({s^-})) \, dX(s) + \frac{1}{2} \int_{s = 0}^t f''(X(s^-)) \, d[X]_s \\
+ \sum_{0 < s \le t} \Ll\{\Delta_s f(X) - f'(X({s^-})) \, \Delta_s X - \frac{1}{2} f''(X(s^-))  \Ll(\Delta_s X\Rr)^2\Rr\},
\end{multline*}
which is nothing but It\^{o}'s formula, see \cite[Theorem~II.7.32]{protter}.
\end{remark}

\section{Martingale characterization of the stochastic heat equation}\label{sec:mart}

Recall that we write $(P_t)_{t \ge 0} = (e^{t \Delta})_{t \ge 0}$ for the semigroup associated with the Laplacian on $\T^2$. Let $W$ be a cylindrical Wiener process over $L^2(\T^2)$. The solution ${Z}$ of the stochastic heat equation
\begin{equation}
\label{e:heateq2}
\left\{
\begin{array}{l}
d{Z} = \Delta {Z} \, dt+ \sqrt{2} \, dW \\
{Z}(0,\cdot) = z_0
\end{array}
\right.
\end{equation}
is defined by
$$
{Z}(t) = P_t z_0  + \sqrt{2} \int_0^tP_{t-s}  \, dW(s).
$$
If $\phi \in \Cc^\infty(\T^2)$ is a smooth test function, and if we write $(z,\phi)$ for the evaluation of $z \in \Ss^{\prime}(\T^2)$ against $\phi$, we have 
$$
({Z}(t),\phi) = (z_0, P_t \phi) +\sqrt{2}  \int_0^t  P_{t-s} \phi \, dW(s)
$$
for every $t \ge 0$.

The goal of this section is to give a martingale characterization of the law of ${Z}$. 
\begin{theorem}[Uniqueness for the martingale problem]
\label{t:martchar}
Let $\ov{Z}$ be a random process in $\Cc(\R_+, \Ss'(\T^2))$. For every $\phi \in \Cc^\infty(\T^2)$, define
$$
\Mm_\phi(t) = (\ov{Z}(t), \phi) - (\ov{Z}(0), \phi) - \int_0^t (\ov{Z}(s),\Delta \phi) \, ds
$$
and
$$
\Gamma_\phi(t) = \Mm_\phi^2(t) - 2 t \|\phi\|_{L^2}^2.
$$
If for every $\phi \in \Cc^\infty(\T^2)$, the processes $(\Mm_\phi(t))_{t \ge 0}$ and $(\Gamma_\phi(t))_{t \ge 0}$ are local martingales, then $\ov{Z}$ has the same law as ${Z}$ the solution of \eqref{e:heateq2} with initial condition $z_0 = \ov{Z}(0)$.
\end{theorem}
\begin{remark}
\label{r:continuity}
Since $\ov{Z}$ is assumed to take values in $\Cc(\R_+,\Ss^{\prime}(\T^2))$, the martingales $(\Mm_\phi(t))_{t \ge 0}$ and $(\Gamma_\phi(t))_{t \ge 0}$ are automatically continuous.
\end{remark}
\begin{remark}
\label{r:local-true}
The fact that $(\Gamma_\phi(t))_{t \ge 0}$ is a local martingale identifies $(2 t \|\phi\|_{L^2}^2)_{t \ge 0}$ as the quadratic variation of $(\Mm_\phi(t))_{t \ge 0}$. In particular, for every $t \ge 0$, $\Mm_\phi(t)$ is square-integrable, 
\begin{equation}
\label{e:local-true}
\E\Ll[\Mm_\phi^2(t)\Rr] = 2 t \|\phi\|_{L^2}^2,
\end{equation}
and $(\Mm_\phi(t))_{t \ge 0}$ and $(\Gamma_\phi(t))_{t \ge 0}$ are (true) martingales. 
\end{remark}
\begin{remark}
\label{r:density}
Theorem~\ref{t:martchar} remains true if we only know that $\Mm_\phi$ and $\Gamma_\phi$ are continuous local martingales for every trigonometric polynomial $\phi$. Indeed, it suffices to argue by density using~\eqref{e:local-true}.
\end{remark}
Our proof of Theorem~\ref{t:martchar} begins by showing that the martingale property can be extended to smooth test functions that also depend on the time variable.
\begin{lemma}[Time-dependent martingale characterization]
Let $\ov{Z}$ be a random process in $\Cc(\R_+, \Ss^{\prime}(\T^2))$. For every $\psi \in \Cc^\infty(\R_+ \times \T^2)$, we let
$$
\Mm_\psi(t) = (\ov{Z}(t), \psi(t)) - (\ov{Z}(0), \psi(0)) - \int_0^t (\ov{Z}(s),(\partial_s + \Delta) \psi(s)) \, ds
$$
and
$$
\Gamma_\psi(t) = \Mm_\psi^2(t) - 2 \int_0^t \|\psi(s)\|_{L^2}^2 \, ds.
$$
Under the assumptions of Theorem~\ref{t:martchar}, the processes $(\Mm_\psi(t))_{t \ge 0}$ and $(\Gamma_\psi(t))_{t \ge 0}$ are continuous martingales.
\label{l:time-dep}
\end{lemma}
\begin{proof}
We decompose $(\ov{Z}(t), \psi(t)) - (\ov{Z}(0), \psi(0))$ into
\begin{equation}
\label{e:decomp}
(\ov{Z}(t), \psi(t)) - (\ov{Z}(0), \psi(t)) + (\ov{Z}(0), \psi(t)) - (\ov{Z}(0), \psi(0)).
\end{equation}
By the assumptions of Theorem~\ref{t:martchar}, we can rewrite the first difference as
$$
\int_0^t (\ov{Z}(s), \Delta \psi(t)) \, ds + \Mm_{\psi(t)}(t),
$$
and we decompose the integral further into 
\begin{multline}
\label{e:dec1}
\int_0^t (\ov{Z}(s), \Delta \psi(s)) \, ds + \int_0^t (\ov{Z}(s), \Delta (\psi(t) - \psi(s))) \, ds \\
= \int_0^t (\ov{Z}(s), \Delta \psi(s)) \, ds + \int_{s = 0}^t \int_{u=s}^t (\ov{Z}(s), \Delta \partial_u\psi(u)) \, du \, ds .
\end{multline}
The second difference appearing in \eqref{e:decomp} can be rewritten as
$$
\int_0^t (\ov{Z}(0), \partial_s \psi(s)) \, ds = \int_0^t (\ov{Z}(s), \partial_s \psi(s)) \, ds - \int_0^t (\ov{Z}(s)-\ov{Z}(0), \partial_s \psi(s)) \, ds.
$$
By the assumptions of Theorem~\ref{t:martchar}, we can rewrite the second integral in the right-hand side above as
$$
\int_{s = 0}^t \Ll\{ \int_{u = 0}^s (\ov{Z}(u), \Delta \partial_s(\psi(s)) \, du + \Mm_{\partial_s \psi(s)}(s)  \Rr\} \, ds.
$$
The double integral cancels with that appearing in the right-hand side of \eqref{e:dec1}. We thus obtain
\begin{multline*}
(\ov{Z}(t), \psi(t)) - (\ov{Z}(0), \psi(0)) 
\\ = \int_0^t (\ov{Z}(s) , (\partial_s + \Delta) \psi(s)) \, ds + \Mm_{\psi(t)}(t) - \int_{0}^t \Mm_{\partial_s \psi(s)}(s)  \, ds,
\end{multline*}
so that 
\begin{equation}
\label{mtphi}
\Mm_\psi(t) = \Mm_{\psi(t)}(t) - \int_{0}^t \Mm_{\partial_s \psi(s)}(s)  \, ds.
\end{equation}
By Remark~\ref{r:local-true}, $\Mm_\psi(t)$ is square-integrable for any $t\ge 0$. If we denote by $(\mathcal{F}_t)_{t \ge 0}$ the underlying filtration, we observe that 
$$
\E[\Mm_\psi(t+s) - \Mm_\psi(t) \ | \ \mathcal{F}_t] = \Mm_{\psi(t+s)}(t) - \Mm_{\psi(t)}(t) - \int_t^{t+s} \Mm_{\partial_r \psi(r)}(t) \, dr.
$$
By linearity, this is equal to $0$, and thus $(\Mm_\psi(t))_{t \ge 0}$ is a martingale. 
We now compute its quadratic variation. Since 
$$
t \mapsto \int_{0}^t \Mm_{\partial_s \psi(s)}(s)  \, ds
$$
and $t \mapsto \Mm_{\psi(t)}(t_0)$ are of bounded variation, we can write
$$
d \langle \Mm_\psi \rangle_t = d \langle \Mm_{\psi(t)} \rangle_t = 2 \|\psi(t)\|_{L^2}^2 \, dt,
$$
that is,
$$
\langle \Mm_\psi \rangle_t = 2 \int_0^t \|\psi(s)\|_{L^2}^2 \, ds,
$$
and the proof is complete.
\end{proof}
\begin{proof}[Proof of Theorem~\ref{t:martchar}]
Let $\phi \in \Cc^\infty(\T^2)$, and $\ov{\phi} \in \Cc^\infty([0,t] \times \T^2)$ be the bounded solution of the backward heat equation
$$
\Ll\{
\begin{array}{ll}
(\partial_s + \Delta) \ov{\phi}(s) = 0 & \text{ for } s \in (0,t), \\
\ov{\phi}(t) = \phi. & 
\end{array}
\Rr.
$$
We learn from Lemma~\ref{l:time-dep} that 
$$
\Mm_{\ov{\phi}}(t) = (\ov{Z}(t),\ov{\phi}(t)) - (\ov{Z}(0),\ov{\phi}(0))
$$
is a martingale with quadratic variation 
$$
2 \int_0^t \|\ov{\phi}(s)\|_{L^2}^2 \, ds.
$$
Clearly, $\ov{\phi}(s) = P_{t-s} \phi$, so we can rewrite $\Mm_{\ov{\phi}}(t)$ as
$$
\Mm_{\ov{\phi}}(t) = (\ov{Z}(t),\phi) - (\ov{Z}(0),P_t \phi)
$$
and the quadratic variation as
$$
2 \int_0^t \|P_{t-s} \phi\|_{L^2}^2 \, ds.
$$
By polarization, it follows that the quadratic covariation at time $t$ between $\Mm_{\ov{\phi_1}}$ and $\Mm_{\ov{\phi_2}}$ is
$$
2 \int_0^t (P_{t-s} \phi_1, P_{t-s} \phi_2) \, ds.
$$
for every $\phi_1, \phi_2 \in \Cc^\infty(\T^2)$.
Since $\Mm_{\ov{\phi}}$ is a continuous martingale for every~$\phi$, the fact that the quadratic covariations are deterministic implies that the law of $(\Mm_{\ov{\phi}}(t))_{t \ge 0, \phi \in \Cc^\infty(\T^2)}$ is determined uniquely. Letting ${Z}$ be the solution of \eqref{e:heateq2} with $z_0 = \ov{Z}(0)$, we observe that the quadratic covariations of 
$$
{\Mm}'_{\phi}(t) := ({Z}(t),\phi) - ({Z}(0),P_t \phi)
$$
are the same as those of $(\Mm_{\ov{\phi}}(t))_{t \ge 0, \phi \in \Cc^\infty(\T^2)}$, so the proof is complete.
\end{proof}

\bibliographystyle{abbrv}
\bibliography{Ising}
\end{document}